\documentclass[11pt,a4paper]{article}

\usepackage{mathtools} 
\usepackage{amssymb,amsthm}
\usepackage{algorithm}
\usepackage[noend]{algpseudocode}
\usepackage{tikz}
\usetikzlibrary{calc, shapes.geometric}
\usepackage{url}
\usepackage{hyperref}
\usepackage{doi}

\newtheorem{theorem}{Theorem}
\newtheorem{corollary}{Corollary}
\newtheorem{lemma}{Lemma}

\newtheorem{proposition}{Proposition}

\pgfdeclarelayer{bg}    
\pgfsetlayers{bg,main}

\begin{document}

\title{Generation and New Infinite Families of \texorpdfstring{$K_2$}{K2}-hypohamiltonian Graphs}

\author{{\sc Jan GOEDGEBEUR\footnote{Department of Computer Science, KU Leuven Campus Kulak-Kortrijk, 8500 Kortrijk, Belgium}\;\footnote{Department of Applied Mathematics, Computer Science and Statistics, Ghent University, 9000 Ghent, Belgium}}\;,
Jarne RENDERS\footnotemark[1]\;,
and {\sc Carol T. ZAMFIRESCU\footnotemark[2]\;\footnote{Department of Mathematics, Babe\c{s}-Bolyai University, Cluj-Napoca, Roumania}}\;\footnote{E-mail addresses: jan.goedgebeur@kuleuven.be; jarne.renders@kuleuven.be; czamfirescu@gmail.com}}

\date{}

\maketitle

\begin{center}
\begin{minipage}{125mm}
{\bf Abstract.} We present an algorithm which can generate all pairwise non-isomorphic \emph{$K_2$-hypohamiltonian graphs}, i.e.\ non-hamiltonian graphs in which the removal of any pair of adjacent vertices yields a hamiltonian graph, of a given order. We introduce new bounding criteria specifically designed for $K_2$-hypohamiltonian graphs, allowing us to improve upon earlier computational results. Specifically, we characterise the orders for which $K_2$-hypohamiltonian graphs exist and improve existing lower bounds on the orders of the smallest planar and the smallest bipartite $K_2$-hypohamiltonian graphs. Furthermore, we describe a new operation for creating $K_2$-hypohamiltonian graphs that preserves planarity under certain conditions and use it to prove the existence of a planar $K_2$-hypohamiltonian graph of order $n$ for every integer $n\geq 134$. Additionally, motivated by a theorem of Thomassen on hypohamiltonian graphs, we show the existence $K_2$-hypohamiltonian graphs with large maximum degree and size.

\bigskip


\bigskip


\end{minipage}
\end{center}

\section{Introduction}

A graph is called \textit{hypohamiltonian} if it is non-hamiltonian and the removal of any vertex leads to a hamiltonian graph. We call a graph \textit{$K_2$-hamiltonian} if the removal of any pair of adjacent vertices leads to a hamiltonian graph. If it is also non-hamiltonian, it is a \textit{$K_2$-hypohamiltonian} graph. In this article we will study $K_2$-hypohamiltonian graphs. We continue work from~\cite{GRWZ22}; for motivation of these concepts, we refer to the introduction and references given in a paper by the third author~\cite{Za21}.

Historically, in the hypohamiltonian case, it was of great interest to investigate the existence of hypohamiltonian graphs for a given order. For instance, in 1966 Herz, Duby and Vigué used a computer search~\cite{HDV66} to prove that there are no hypohamiltonian graphs of orders $11$ and $12$, and in 1967 Lindgren showed the existence of an infinite family of hypohamiltonian graphs~\cite{Li67}. Aldred, McKay and Wormald completed the characterisation of the orders for which hypohamiltonian graphs exist in 1997~\cite{AMW97}, demonstrating the absence of a hypohamiltonian graph of order $17$. This proved that there exists a hypohamiltonian graph of order $n$ if and only if $n\in\{10,13,15,16\}$, or $n$ is an integer greater than or equal to $18$. A historical overview including the remaining orders was given by Holton and Sheehan in~\cite{HS93}.

In~\cite{GRWZ22}, using mathematical constructions as well as a computational filter method, we showed that the Petersen graph is the smallest $K_2$-hypohamiltonian graph. Furthermore, we were able to prove the existence of $K_2$-hypohamiltonian graphs on orders $10,13,15,16$, and from $18$ onwards, as well as their non-existence up to order $9$ and for orders $11$ and $12$. However, whether such graphs exist for orders $14$ and $17$ remained unknown, and the question of the existence of $K_2$-hypohamiltonian graphs with girth $3$ and $4$ persisted.

To address these questions, we describe a specialised generator that exhaustively generates all pairwise non-isomorphic $K_2$-hypohamiltonian graphs of a given order. We used this to determine whether or not $K_2$-hypohamiltonian graphs exist on these two open orders, to increase the known lower bounds on the order of the smallest $K_2$-hypohamiltonian graphs with girth $3$ or girth $4$ and to exhaustively determine which graphs are $K_2$-hypohamiltonian up to higher orders. We did this by proving that there are several \emph{obstructions} for $K_2$-hypohamiltonicity, i.e.\ properties that cannot hold for $K_2$-hypohamiltonian graphs. We then used these obstructions as bounding criteria in the generation algorithm. This method is inspired by a generator for hypohamiltonian graphs showcased in~\cite{GZ17}, which was later also extended for generating \emph{platypus graphs}~\cite{GNZ20} (i.e.\ non-hamiltonian graphs for which every vertex-deleted subgraph is traceable).

Currently, the order of a smallest planar $K_2$-hypohamiltonian graph is not known. This is also the case for the smallest planar hypohamiltonian graph. However, results by Aldred, McKay and Wormald~\cite{AMW97} implied a lower bound of $18$ for the latter. This was later sharpened to $23$ in 2017~\cite{GZ17}. Restricting our generator to the class of planar graphs we increase the lower bound from~\cite{GRWZ22} for the order of a smallest planar $K_2$-hypohamiltonian graph from $18$ to $24$. The smallest known planar $K_2$-hypohamiltonian graph was given in the same paper and has $48$ vertices. 

In the hypohamiltonian case bipartite graphs cannot exist. However, a question by Gr\"otschel \cite[Problem 4.56]{GGL95} asked whether there exist bipartite \emph{hypotraceable graphs}, i.e.\ graphs for which every vertex-deleted subgraph has a hamiltonian path, but which do not contain one themselves. We study here the existence of bipartite $K_2$-hypohamiltonian graphs. In~\cite{GRWZ22} a lower bound of $17$ on the order of a smallest bipartite $K_1$-hypohamiltonian graph was given. Similarly as in the planar case, by restricting the generation to bipartite graphs we increase this lower bound to $30$. However, the existence of bipartite $K_2$-hypohamiltonian graphs is not known.

In Section~\ref{sec:algorithm} we present the exhaustive generation algorithm and the obstructions for $K_2$-hypohamiltonian graphs and prove the correctness of the generator.
In Section~\ref{sec:operation} we introduce new ways of constructing infinite families of $K_2$-hypohamiltonian graphs. In particular, we introduce an operation for the construction of $K_2$-hypohamiltonian graphs, which are not hypohamiltonian, and which preserve planarity under certain conditions. We use this to classify the orders for which planar $K_2$-hypohamiltonian graphs exist, building on work from~\cite{GRWZ22}.

We also describe a new infinite family of hypohamiltonian $K_2$-hamiltonian graphs. Members of this family with $n$ vertices have a maximum degree of approximately $n/3$ and approximately $2n$ edges. It is of interest what the maximum degree and maximum number of edges in a $K_2$-hypohamiltonian graph can be, as was historically the case for hypohamiltonian graphs. For the hypohamiltonian case this question was answered by Thomassen~\cite{Th81} by describing a family of hypohamiltonian graphs attaining up to some constant the upper bounds on the maximum degree and size, which are $(n-3)/2$ and $n^2/4$ for a graph of order $n$, respectively. However, in the $K_2$-hypohamiltonian case no non-trivial upper bounds for the maximum degree and size are known.

Throughout this paper we assume all graphs to be simple, finite and connected unless mentioned otherwise. $K_2$-hypohamiltonian graphs are $3$-connected and hence have minimum degree at least $3$. We will use this fact tacitly throughout the paper. A path with end-vertex $v$ is a \emph{$v$-path}, and a 
 $v$-path with end-vertex $w \ne v$ is a \emph{$vw$-path}. We denote the complement of a graph $G$ by $G^c$, and by $N_G[v]$ the subset of $V(G)$ containing $v$ and all its neighbours in a graph $G$. We will write $N[v]$ if it is clear we are referring to a subset of $V(G)$.
 
\section{Generating \texorpdfstring{\boldmath{$K_2$}}{K2}-hypohamiltonian graphs}
\label{sec:algorithm}

In this section, we present our algorithm to generate all pairwise non-isomorphic $K_2$-hypohamiltonian graphs of a given order. The fundamental idea behind the algorithm is based on the generator for hypohamiltonian graphs from~\cite{GZ17}, which in turn is based on work by Aldred, McKay and Wormald~\cite{AMW97}. However, we introduce essential new bounding criteria specifically designed for $K_2$-hypohamiltonian graphs which are vital for the efficiency of the algorithm and adapt the algorithm from~\cite{GZ17} for use in the $K_2$-hypohamiltonian case.

\subsection{Obstructions for \texorpdfstring{\boldmath{$K_2$}}{K2}-hypohamiltonicity}

We will start this section by presenting several lemmas on properties that hold true for $K_2$-(hypo)hamiltonian graphs. A selection of these will be used as bounding criteria in the algorithm to reduce the search space, i.e.\ if a graph does not satisfy one of the properties, it cannot be $K_2$-(hypo)hamiltonian. The selection of which obstructions or bounding criteria we apply is based on an experimental analysis as there is usually a trade-off between how quickly these conditions can be computed and how much of the search space they reduce.

In~\cite{AMW97} Aldred, McKay and Wormald used what they called type A, B and C obstructions. These are specific configurations that cannot occur in hypohamiltonian graphs. We will use a similar terminology, but note however that here they indicate obstructions for $K_2$-hypohamiltonicity instead of hypohamiltonicity. The statements of the obstructions and lemmas contain some subtle differences compared to the hypohamiltonian case. Moreover, we will introduce several lemmas specifically for $K_2$-hypohamiltonian graphs. While some of these are adaptations from~\cite{AMW97} and~\cite{GZ17}, we also introduce some new ones. The precise definition of these obstructions for $K_2$-hypohamiltonicity will be given later in this section.

A \emph{non-trivial partition} $(W,X)$ of a set $V$ is a pair of two disjoint subsets $W$ and $X$ of $V$ such that their union is $V$ and neither is empty. Let $G$ be a possibly disconnected graph. We denote by $p(G)$ the minimum number of pairwise disjoint paths needed to cover all the vertices of $G$, by $V_1(G)$ all vertices of degree $1$ in $G$ and by $I(G)$ the set of all isolated vertices and all isolated $K_2$'s of $G$, i.e.\ isolated edges with their endpoints. Define 
\[k(G) = \begin{cases*}
0 & if $G$ is empty,\\
\max{\left\{1, \left\lceil\dfrac{\lvert V_1\rvert}{2}\right\rceil\right\}} & if $I(G)=\emptyset$ but $G$ is not empty,\\
\lvert I(G)\rvert + k(G - I(G)) & else.
\end{cases*}\]

\begin{lemma}[Aldred, McKay, Wormald~\cite{AMW97}]\label{lem:kLessThanP}
	For a graph $G$, we have $k(G) \leq p(G)$.
\end{lemma}
Let $\mathfrak{p}$ be a path or cycle in $G$ and $S$ a subset of the vertices of $G$.  We define $c_{\mathfrak{p}}\vert_S$ to be the number of components of $\mathfrak{p}$ restricted to $G[S]$.
	\begin{lemma}\label{lem:typeA}
		Let $G$ be a $K_2$-hamiltonian graph and $(W,X)$ a non-trivial partition $V(G)$. If $G[X]$ contains adjacent vertices and $\lvert X\rvert \geq 3$, then 
		$$p(G[W]) < \lvert X \rvert - 1\quad\text{and}\quad k(G[W]) < \lvert X \rvert - 1.$$
	\end{lemma}
	\begin{proof}
		Let $v$ and $w$ be adjacent vertices of $G[X]$. Since $G$ is $K_2$-hamiltonian, there is a hamiltonian cycle $\mathfrak{h}$ in $G - v - w$. Then $c_{\mathfrak{h}}\vert_W$ must be at least $p(G[W])$, hence $c_{\mathfrak{h}}\vert_{X-v-w}$ must also be at least $p(G[W])$, since $c_{\mathfrak{h}}\vert_W = c_{\mathfrak{h}}\vert_{X-v -w}$. As $c_{\mathfrak{h}}\vert_{X-v -w}$ is at most $\lvert X \rvert - 2$, we get $$p(G[W])\leq c_{\mathfrak{h}}\vert_{X-v-w} \leq \lvert X\rvert - 2 < \lvert X\rvert - 1.$$ By Lemma~$\ref{lem:kLessThanP}$ the result follows.
	\end{proof}
Consider a graph $G$ containing a non-trivial partition $(W,X)$ of the vertices of $G$ such that $X$ has at least three vertices and contains a pair of adjacent vertices. If $p(G[W])\geq \lvert X \rvert - 1$, we call $(W,X)$ a \textit{type A obstruction} and if $k(G[W])\geq \lvert X \rvert - 1$, we call $(W,X)$ a \textit{type B obstruction}. For the efficiency of our algorithm, we only consider type A obstructions where $G[W]$ is a union of disjoint paths with the extra condition that none of these paths in $G[W]$ have adjacent endpoints. In this way computing $p(G[W])$ is equivalent to counting the degree $1$ vertices of $G[W]$, dividing this number by two and adding the number of degree $0$ vertices.
	
The case when $X$ does not contain adjacent vertices might be of theoretical interest. We explore this in the next lemma.

	\begin{lemma}\label{lem:nonAdjTypeAAndB}
		Let $G$ be a $K_2$-hamiltonian graph and $(W,X)$ a non-trivial partition of $V(G)$ with $\lvert X\rvert > 1$. If $G[X]$ does not contain adjacent vertices, then 
		$$p(G[W - w]) < \lvert X \rvert\quad\text{and}\quad k(G[W-w]) < \lvert X \rvert$$ for any $w\in W$ adjacent to a vertex in $X$.
	\end{lemma}
	\begin{proof}
		Let $v,w$ be adjacent vertices of $G$ such that $v$ is in $X$ and $w$ in $W$. We can find such vertices since $G$ is connected. As $G$ is $K_2$-hamiltonian, we have a hamiltonian cycle $\mathfrak{h}$ in $G - v - w$. If $W$ contains only one element, we are done. Assume otherwise; then the number of components of $\mathfrak{h}$ restricted to $W$ is at least $p(G[W - w])$. Hence, the number of components of $\mathfrak{h}$ restricted to $X$ is at least $p(G[W - w])$. Since this number of components is at most $\lvert X\rvert - 1$, we are done by Lemma~\ref{lem:kLessThanP}.
	\end{proof}
	Note that adding a vertex to a subgraph of $G$ can only increase the number of pairwise disjoint paths needed to cover that subgraph by at most one. This yields the following corollary.
\begin{corollary}
	Under the conditions of Lemma~\ref{lem:nonAdjTypeAAndB}, we have $p(G[W]) < \lvert X \rvert + 1$.
\end{corollary}

\begin{lemma}\label{lem:typeC}
	Let $G$ be a $K_2$-hamiltonian graph and $(W,X)$ a non-trivial partition of the vertices of $G$, such that $W$ is an independent set. Suppose  $G[X]$ contains adjacent vertices $v$ and $w$. Let $n_1$ and $n_2$ be the number of vertices of $X - v - w$ joined to one or more than one vertex of $W$, respectively. Then $2n_2+n_1 \geq 2\lvert W \rvert$.
\end{lemma}
\begin{proof}
	Let $\mathfrak{h}$ be a hamiltonian cycle in $G - v - w$. Since $W$ is an independent set, the number of edges of $\mathfrak{h}$ between $W$ and $X$ must be $2\lvert W \rvert$. On the other hand $X$ can supply at most $2n_2+n_1$ edges. The result follows.
\end{proof}
Let $G$ be a graph which is not $K_2$-hamiltonian and suppose that for a non-trivial partition $(W,X)$ of the vertices of $G$ the remaining assumptions of Lemma~\ref{lem:typeC} are met, but $2n_2+n_1 < 2\lvert W\rvert$ for some adjacent $v,w\in X$. Then we call $(W,X,vw)$ a \textit{type C obstruction}.

\begin{lemma}\label{lem:starObstruction}
	Let $G$ be a $K_2$-hypohamiltonian graph containing a triangle $uvw$. Then $u$ has at least two neighbours not in $N[v]\cup N[w]$.
\end{lemma}
\begin{proof}
	Suppose that $u$ has at most one neighbour not in $N[v]\cup N[w]$. There exists a hamiltonian cycle in $G - v - w$. Note that this cycle does not use the edge $ux$, otherwise replacing $ux$ with $uvwx$, we would obtain a hamiltonian cycle in G. Hence, it must contain an edge $ux$ where $x$ is a neighbour of $u$ that also neighbours either $v$ or $w$ in $G$. However, we can then replace this edge by $xvwu$ or $xwvu$ which yields a hamiltonian cycle in $G$ which is a contradiction.
\end{proof}
Let $G$ be a graph and $u$ a vertex of $G$ lying on a triangle $uvw$ such that $u$ has at most one neighbour not in $N[v]\cup N[w]$. Then we say $(u, uvw)$ is a \emph{triangle obstruction} in $G$.
\begin{corollary}[\cite{Za21}]\label{lem:threeCycle}
	Let $G$ be a $K_2$-hypohamiltonian graph. The vertices of a triangle in $G$ have degree at least $4$.
\end{corollary}
\begin{lemma}\label{lem:arrowObstruction}
	Let $G$ be a $K_2$-hypohamiltonian graph containing a $4$-cycle $uvwx$. Then $u$ has at least two neighbours not in $N[w]$. 
\end{lemma}
\begin{proof}
	Assume that $u$ has at most one neighbour not in $N[w]$. There exists a hamiltonian cycle in $G - v - w$ and it must contain the edge $uy$ where $y$ is a neighbour of $u$ adjacent to $w$ in $G$. Replacing this edge with $ywvu$ yields a hamiltonian cycle in $G$, which is a contradiction.
\end{proof}
Suppose that $u$ belongs to a $4$-cycle $uvwx$ such that $u$ has at most one neighbour which is not $w$ and which is non-adjacent to $w$. Then we say $(u, uvwx)$ is a \emph{$4$-cycle obstruction}.
\begin{corollary}[\cite{Za21}]\label{lem:fourCycle}
	Let $G$ be a $K_2$-hypohamiltonian graph. The vertices of a $4$-cycle in $G$ have degree at least $4$.
\end{corollary}
A \textit{diamond} is a $K_4$ from which we remove one edge. Its \textit{central edge} is the edge between the two cubic vertices.
\begin{corollary}\label{lem:diamond}
	Let $G$ be a $K_2$-hypohamiltonian graph containing a diamond with vertices $a,b,c,d$ and central edge $ac$. Then the degrees of $a$ and $c$ (in $G$) are at least $5$.
\end{corollary}

 Let $G$ be a non-hamiltonian graph. If adding a non-edge $e$ to $G$ makes $G$ hamiltonian, we call $e$ a \emph{bad edge}. Fix an obstruction $O$ in $G$. We say a non-edge $e$ \textit{destroys} $O$ in $G$ if $G + e$ does not contain $O$. Sometimes it is necessary to add multiple non-edges of $G$ in order to destroy an obstruction $O$. We say for each of these non-edges that they \emph{work towards the destruction} of $O$. We will call such a non-edge of $G$ a \textit{good edge}. We will more formally define good edges for the different obstructions.

Let $G'$ be a $K_2$-hypohamiltonian graph such that $G$ is a spanning subgraph.
Let $G$ contain a type A obstruction $(W,X)$, i.e.\ $(W,X)$ is a non-trivial partition of $V(G)$, $G[X]$ contains at least three vertices some of which are adjacent and $p(G[W])\geq \lvert X\rvert - 1$. $G'$ cannot contain a type A obstruction by Lemma~\ref{lem:typeA} since $G'$ is $K_2$-hypohamiltonian, hence there exists an edge in $E(G')\setminus E(G)$ with endpoints in different components of $G[W]$. We will call such an edge a \textit{good $(W,X)$ A-edge}. 
Similarly, let $G$ contain a type B obstruction $(W,X)$, i.e.\ $(W,X)$ is a non-trivial partition of $V(G)$, $G[X]$ contains at least three vertices some of which are adjacent and $k(G[W])\geq \lvert X\rvert - 1$. By definition of $k(G)$, we can only destroy a type B obstruction by destroying the isolated vertices and isolated $K_2$'s in $G[W]$ or by destroying the vertices of degree $1$ in $G[W] - I(G[W])$. Hence, an edge in $G^c$ is a \textit{good $(W,X)$ B-edge} if both of its endpoints lie in $W$ and at least one of its endpoints has degree at most $1$ in $G[W]$. It follows from Lemma~\ref{lem:kLessThanP} that every type B obstruction is a type A obstruction but not vice versa. However, experimental evidence showed that a type A obstruction is only very sporadically not a type B obstruction. Since the computational overhead of searching for type B obstructions is greater than that of searching for type A obstructions we obtained better results by only searching for type A obstructions. We therefore omit the type B obstructions entirely from Algorithm~\ref{alg:AddEdges} below and our implementation.

Let $G$ contain a type C obstruction $(W,X,vw)$, i.e.\ $(W,X)$ is a non-trivial partition, $W$ and independent set, $v$ and $w$ are adjacent vertices in $G[X]$, $n_1$ and $n_2$ are the number of vertices of $X - v - w$ joined to one or more than one vertex, respectively, of $W$ and $2n_2 + n_1 < 2\lvert W\rvert$.
A non-edge $e$ of $G$ destroys or works towards destruction of $(W,X,vw)$ if adding it to $G$ increases $2n_2 + n_1$ or if $W$ is not an independent set of $G + e$.
Hence, we call an edge in $G^c$ a \textit{good $(W,X,vw)$-edge} if it connects a vertex of $X - v - w$ which has at most one neighbour in $W$ to $W$ or if both of its endpoints are in $W$.

If $G$ contains a triangle obstruction $(u, uvw)$, i.e.\ $uvw$ is a triangle in $G$ such that $\lvert N(u)\setminus (N(v)\cup N(w))\rvert\leq 1$, then there must be an edge $ux\in E(G')\setminus E(G)$ such that $x$ is not an element of $N[v]\cup N[w]$. We call such an edge a \textit{good $(u,uvw)$-edge}.

Similarly, if $G$ contains a $4$-cycle obstruction $(u, uvwx)$, i.e.\ $uvwx$ is a $4$-cycle of $G$ and $\lvert N(u)\setminus N[w]\rvert\leq 1$, then there must be an edge $uy\in E(G')\setminus E(G)$ such that $y$ is not an element of $N[w]$. We call such an edge a \textit{good $(u, uvwx)$-edge}.

\subsection{Algorithm}
The goal of our algorithm is to generate all pairwise non-isomorphic $K_2$-hypohamiltonian graphs of a given order $n$. Its pseudocode is split into two parts given by Algorithm~\ref{alg:Start} and Algorithm~\ref{alg:AddEdges}. The former contains the initialisation of the algorithm, which is only called once, while the latter contains the main routine, which recursively calls itself many times. We initialise the generation by taking an $(n-2)$-cycle and a copy of $K_2$, i.e.\ an isolated edge $uv$. In this way $G - u - v$ is hamiltonian. We then get new graphs by adding two edges with endpoint $u$ and two edges with endpoint $v$ to the graph in all possible ways that do not make the graph hamiltonian, filtering out isomorphic copies. Note that any $K_2$-hypohamiltonian graph of order $n$ must have a subgraph isomorphic to one of these new graphs, as $K_2$-hypohamiltonian graphs are $3$-connected. Then on each of these new graphs we call Algorithm~\ref{alg:AddEdges}, which will recursively add edges until we can prune the search. We use obstructions to minimise the number of edges we need to add in each iteration. During both Algorithm~\ref{alg:Start} and Algorithm~\ref{alg:AddEdges}, we only add edges between existing vertices of the graph. Hence, if a graph is hamiltonian, we prune the search and do not examine its supergraphs. In particular, suppose we add an edge $e$ to a non-hamiltonian graph $G$ and obtain a hamiltonian graph. Then we forbid the adding of $e$ to any graph containing $G$ as a subgraph by keeping and dynamically updating a list of  ``forbidden edges" associated with the graph which is currently being constructed by the algorithm. For efficiency reasons, we store this list of forbidden edges as a bitvector.

Computing whether or not a graph is $K_2$-hypohamiltonian is computationally quite expensive as determining if a graph is hamiltonian is already NP-complete. Therefore, before doing this we look for obstructions that make it impossible for the graph to be $K_2$-hypohamiltonian and perform the $K_2$-hypohamiltonicity check in the few cases that the graph does not contain any obstructions. (E.g.\ for $n = 18$, only $98$ of the $6\,116\,186$ generated intermediate candidate graphs contained no obstructions. Three of these were $K_2$-hypohamiltonian.) Moreover, since we know that the parent graph was non-hamiltonian, if the current graph would be hamiltonian, the hamiltonian cycle must go through the most recently added edge. For the efficiency of the algorithm it is important that in each call of Algorithm~\ref{alg:AddEdges} as few edges as possible are added. To this end, after finding an obstruction, we only need to add edges that destroy or work towards the destruction of this obstruction instead of adding all non-forbidden edges. In Theorem~\ref{thm:inductionProof} we show this is sufficient. In Algorithm~\ref{alg:AddEdges}, we investigate which obstructions are present in the graph and count how many non-hamiltonian successors each investigated obstruction has. The obstruction with the fewest will be the one we destroy or work towards the destruction of. The good edges for this obstruction which, if added, will not yield a hamiltonian graph will be stored in the list \texttt{EdgesToBeAdded}. While checking the different obstructions, we will encounter edges that, if added, give a hamiltonian graph. We add these to the list \texttt{NewForbiddenEdges} as they should not be added to any supergraph of the one we are currently inspecting.

Adding edges would in many cases lead to graphs isomorphic to ones generated earlier. In order to prevent this, we keep a list of the canonical forms of all generated intermediate candidate graphs (computed by \texttt{nauty}~\cite{MP13}). Every time we inspect a candidate graph, its canonical form is compared to those in the list. If it is already present we can reject it, otherwise it is added.

Note that more advanced isomorphism rejection methods exist, for example the canonical construction path method by McKay~\cite{Mc98}. However, such methods are not compatible with the abovementioned obstructions. Hence, this seemingly naive isomorphism rejection method is very suitable for this problem.

In the following theorem, we show that the algorithm produces all $K_2$-hypohamiltonian graphs of a given order.
\begin{theorem}\label{thm:inductionProof}
	Let $n$ be a positive integer. When Algorithm~\ref{alg:Start} terminates, we have output precisely all pairwise non-isomorphic $K_2$-hypohamiltonian graphs of order $n$.
\end{theorem}
\begin{proof}
	By line~\ref{line:Output} of Algorithm~\ref{alg:AddEdges}, we see that only $K_2$-hypohamiltonian graphs are output.  By line~\ref{line:iso} we see that no isomorphic copies are ever output.
	It remains to show that the algorithm produces all such pairwise non-isomorphic graphs of order $n$. Let $G$ be a $K_2$-hypohamiltonian graph. It is easy to see that there is a spanning subgraph $G_0$ of $G$ consisting of an $(n-2)$-cycle $\mathfrak{h}$ and a copy of $K_2$, i.e.\ an edge $uv$, disjoint from $\mathfrak{h}$ such that $u$ is adjacent to two vertices of $\mathfrak{h}$ and $v$ is adjacent to two vertices of $\mathfrak{h}$. Clearly if $G_0$ is hamiltonian, then $G$ is as well, which is a contradiction. 
	
We now proceed by induction and show that if we call Algorithm~\ref{alg:AddEdges} on a graph which is isomorphic to a spanning subgraph of $G$ on $m$ edges with $\lvert E(G_0)\rvert \leq m < \lvert E(G)\rvert $, then we will call Algorithm~\ref{alg:AddEdges} on a graph isomorphic to a spanning subgraph of $G$ on $m+1$ edges. Assume that $G'$ is such a graph with $m$ edges for $\lvert E(G_0)\rvert \leq m < \lvert E(G)\rvert$ and that we call Algorithm~\ref{alg:AddEdges} on $G'$. Without loss of generality, we assume that $G'$ is a subgraph of $G$.

Suppose that the algorithm did not find any obstructions for $G'$. Since $G'$ is a strict subgraph of $G$, there must be an edge $e$ in $E(G)\setminus E(G')$. Clearly $G' + e$ cannot be hamiltonian. Hence, we add $e$ on line~\ref{line:end}, since we do this for all non-forbidden and non-bad edges in $G'^c$.

Now suppose that $G'$ does contain an obstruction. Assume the algorithm found a type A obstruction $(W,X)$ which has the fewest non-forbidden good edges among all obstructions of the graph that were checked. By Lemma~\ref{lem:typeA}, $(W,X)$ is not a type A obstruction in $G$. Hence, $G$ contains a good $(W,X)$ A-edge $e\in E(G)\setminus E(G')$. Since $G'+e$ is not  hamiltonian, $e$ is not forbidden and it is added to $G'$ on line~\ref{line:AddObstructionEdges} of Algorithm~\ref{alg:AddEdges}. An analogous reasoning holds for the other obstructions.

Thus, we call Algorithm~\ref{alg:AddEdges} on a graph isomorphic to a spanning subgraph of $G$ with $m+1$ edges. By induction, we will call Algorithm~\ref{alg:AddEdges} on a graph isomorphic to $G$, say $H$, and since it is $K_2$-hypohamiltonian, we will output $H$ on line~\ref{line:Output}.
\end{proof}

\begin{algorithm}[!htb]
\caption{Generate all $K_2$-hypohamiltonian graphs of order $n$}
\label{alg:Start}
\begin{algorithmic}[1]
	\State let $G:=C_{n-2}+\{\{u,v\}, \{uv\}\}$
	\For{every way of adding two edges incident to $u$ and two edges incident to $v$}
		\State Call the resulting graph $G'$
		\If{$G'$ is hamiltonian}
			\State Go to the next iteration of the loop.
		\EndIf
		\If{$G'$ is isomorphic to a previously constructed graph}
			\State Go to the next iteration of the loop.
		\EndIf
		\State Create list of bitvectors \texttt{ForbiddenEdges} containing only edges which, if added, \hspace*{4mm} make $G'$ hamiltonian
		\State \textsc{AddEdges}($G'$, \texttt{ForbiddenEdges}) // i.e.\ perform Algorithm~\ref{alg:AddEdges}.
	\EndFor
\end{algorithmic}
\end{algorithm}
\begin{algorithm}[!htb]
\caption{\textsc{AddEdges}(Graph $G$, List \texttt{ForbiddenEdges})}\label{alg:AddEdges}
\begin{algorithmic}[1]
	\If{an isomorphic copy of $G$ was generated before}\label{line:iso}
		\State \Return 
	\EndIf
	\State Create List \texttt{NewForbiddenEdges} = \texttt{ForbiddenEdges}
	\State Create List \texttt{EdgesToBeAdded} containing all edges of $G^c$
	\For{type A obstructions $(W,X)$}
		\ForAll{non-forbidden good $(W,X)$ A-edges $e$}
			\State Add bad edges $e$ to \texttt{NewForbiddenEdges}
		\EndFor
		\If {$\# \text{non-forbidden good } (W,X) \text{ A-edges} < \# \texttt{EdgesToBeAdded} $}
			\State \texttt{EdgesToBeAdded} = non-forbidden good $(W,X)$ A-edges 
		\EndIf
	\EndFor
	\ForAll{type C obstructions $(W,X)$}
		\ForAll{non-forbidden good C-edges $e$}
			\State Add bad edges $e$ to \texttt{NewForbiddenEdges}
		\EndFor
		\If {$\# \text{non-forbidden good } (W,X,vw)\text{-edges} < \# \texttt{EdgesToBeAdded} $}
			\State \texttt{EdgesToBeAdded} = non-forbidden good $(W,X,vw)$-edges 
		\EndIf
	\EndFor
	\ForAll{vertices $w$ of degree $2$}
		\ForAll{edges $e$ in $G^c$ having $w$ as an endpoint}
			\State Add bad edges $e$ to \texttt{NewForbiddenEdges}
		\EndFor
		\If {$\#\text{non-forbidden edges in }G^c \text{ incident to }w < \# \texttt{EdgesToBeAdded}$}
			\State \texttt{EdgesToBeAdded} = non-forbidden edges in $G^c$ incident to $w$
		\EndIf
	\EndFor
	\ForAll{triangle obstructions $(v, uvw)$}
		\ForAll{good $(v,uvw)$-edges $e$}
			\State Add bad edges $e$ to \texttt{NewForbiddenEdges}
		\EndFor
		\If {$\# \text{non-forbidden good } (v, uvw)\text{-edges} < \# \texttt{EdgesToBeAdded}$}
			\State \texttt{EdgesToBeAdded} = non-forbidden good $(v, uwv)$-edges 
		\EndIf
	\EndFor
	\ForAll{$4$-cycle obstructions $(u, uvwx)$}
		\ForAll{good $(u, uvwx)$-edges $e$}
			\State Add bad edges $e$ to \texttt{NewForbiddenEdges}
		\EndFor
		\If {$\# \text{non-forbidden good } (u, uvwx)\text{-edges} < \# \texttt{EdgesToBeAdded}$}
			\State \texttt{EdgesToBeAdded} = non-forbidden good $(u, uwvx)$-edges 
		\EndIf
	\EndFor
\If{we found any of the above obstructions}
	\ForAll{edges $e$ in \texttt{EdgesToBeAdded}}
		\State \textsc{AddEdges}($G+e$, \texttt{NewForbiddenEdges})\label{line:AddObstructionEdges}
	\EndFor
	\State \Return
\Else
	\If{$G$ is $K_2$-hypohamiltonian}\label{line:Output}
		\State Output $G$ 
	\EndIf
	\ForAll{edges $e$ in \texttt{EdgesToBeAdded}}
		\If{$e$ is not forbidden and $G+e$ is not hamiltonian}
			\State \textsc{AddEdges}($G + e$, \texttt{NewForbiddenEdges}) \label{line:end}
		\EndIf
	\EndFor
\EndIf
\end{algorithmic}
\end{algorithm}

\subsection{Results}
We created an implementation of this algorithm which can be found on GitHub~\cite{GRZ23} and used this to improve the results given in~\cite{GRWZ22}. We generated all pairwise non-isomorphic $K_2$-hypohamiltonian graphs of order $n$ for all $n\leq 19$. Note that previously this was only done up to order $13$. We were also able to extend the results for given lower bounds on the girth. Indeed, since we only add edges in the algorithm, we can simply prune the search as soon as a cycle of forbidden length appears. The longest computations were parallellised and performed on the supercomputer of the VSC (Flemish Supercomputer Center) as they took 1-2 CPU-years. The results are summarised in Table~\ref{table:general_counts}. The complete lists can be downloaded from the \textit{House of Graphs}~\cite{CDG23} at \url{https://houseofgraphs.org/meta-directory/K2-hypohamiltonian}.

\begin{table}[!htb]
\centering
\begin{tabular}{c | c | c | c | c | c | }
	Order & $g\geq 3$ & $g\geq 4$ & $g\geq 5$ & $g\geq 6$ & $g\geq7$\\\hline
	10 & 1 & 1 & 1 & 0 & 0\\
	11 & 0 & 0 & 0 & 0 & 0\\
	12 & 0 & 0 & 0 & 0 & 0\\
	13 & 1 & 1 & 1 & 0 & 0\\
	14 & \textbf{0} & 0 & 0 & 0 & 0\\
	15 & \textbf{1} & 1 & 1 & 0 & 0\\
	16 & \textbf{4} & 4 & 4 & 0 & 0\\
	17 & \textbf{0} & \textbf{0} & 0 & 0 & 0\\
	18 & \textbf{3} & \textbf{3} & 3 & 0 & 0\\
	19 & \textbf{28} & \textbf{28} & 28 & 0 & 0\\
	20 & ? & \textbf{2} & 2 & 0 & 0\\
	21 & ? & ? & 31 & 0 & 0\\
	22 & ? & ? & \textbf{332} & 0 & 0\\
	23 & ? & ? & \textbf{19} & 0 & 0\\
	24 & ? & ? & \textbf{613} & 0 & 0\\
	25 & ? & ? & ? & \textbf{1} & 0\\
	26 & ? & ? & ? & \textbf{0} & \textbf{0}\\
	27 & ? & ? & ? & \textbf{0} & \textbf{0}\\
	28--30 & ? & ? & ? & ? & \textbf{0}
\end{tabular}
\caption{Counts of $K_2$-hypohamiltonian graphs. Columns with $g\geq k$ indicate that these are only the graphs with girth at least $k$. Bold entries indicate new results.}
\label{table:general_counts}
\end{table}

\interfootnotelinepenalty=10000

In particular, we show that there exist no $K_2$-hypohamiltonian graphs of order $14$ or order $17$. Combined with our results from~\cite{GRWZ22}, we can now characterise the orders for which there exist $K_2$-hypohamiltonian graphs, thereby mirroring the characterisation of orders of hypohamiltonian graphs given by Aldred, McKay and Wormald~\cite{AMW97}.
\begin{theorem}
	There exists a $K_2$-hypohamiltonian graph of order $n$ if and only if $n\in \{10,13,15,16\}$ or $n\geq 18$.
\end{theorem}

Currently it is unknown whether $K_2$-hypohamiltonian graphs of girth $3$ or girth $4$ exist. This is different from the hypohamiltonian setting where the smallest hypohamiltonian graph of girth $3$ and the smallest of girth $4$ have order $18$. From Table~\ref{table:general_counts} we infer the following.
\begin{proposition}
	A $K_2$-hypohamiltonian graph of girth $3$ has order at least $20$. A $K_2$-hypohamiltonian graph of girth $4$ has order at least $21$.
\end{proposition}

We remarked in~\cite{GRWZ22} that the smallest $K_2$-hypohamiltonian graph which is not hypohamiltonian has order at least $14$ and at most $18$. This upper bound comes from a graph of girth $5$. As our generation results did not yield any $K_2$-hypohamiltonian graphs of girth $3$ or $4$ up to order $18$, we have shown the following.
\begin{proposition}
	The smallest non-hypohamiltonian $K_2$-hypohamiltonian graph has order $18$. 
\end{proposition}
There is only one such graph of order $18$ and it is obtained by applying the operation detailed in Section~\ref{sec:operation} to two copies of the Petersen graph, see \cite[Fig.~4b]{GRWZ22}\footnote{This graph can also be inspected on the \textit{House of Graphs}~\cite{CDG23} at \url{https://houseofgraphs.org/graphs/49001}.}.

In~\cite{GZ17} it is shown that a smallest planar hypohamiltonian graph has order at least 23, improving a result by Aldred, McKay and Wormald~\cite{AMW97}. In~\cite{GRWZ22} it was proven that a smallest planar $K_2$-hypohamiltonian graph has order at least $18$ and at most $48$. There is no mention of planar $K_2$-hypohamiltonian graphs of girth at least $4$. We improve the lower bound on the order of a smallest planar $K_2$-hypohamiltonian graph using our generator and also give a lower bound when the girth is at least $4$.

Note that it is relatively simple and efficient to restrict the algorithm to the class of planar graphs. In particular, checking planarity at the start of Algorithm~\ref{alg:AddEdges} and pruning if the graph is non-planar will generate all planar $K_2$-hypohamiltonian graphs. We implemented this using the planarity algorithm distributed with \texttt{nauty}~\cite{MP13} based on a method by Boyer and Myrvold~\cite{BM04} and this allowed us to show the following.

\begin{proposition}
	A smallest planar $K_2$-hypohamiltonian graph has order at least $24$. A smallest planar $K_2$-hypohamiltonian graphs of girth at least $4$ has order at least $26$.
\end{proposition}
The smallest planar $K_2$-hypohamiltonian graph of girth $5$ was determined by the authors in~\cite{GRWZ22}. It has $48$ vertices\footnote{This graph can also be inspected on the \textit{House of Graphs}~\cite{CDG23} at \url{https://houseofgraphs.org/graphs/49121}.}. Note that $3$-connected planar graphs have girth at most $5$, hence there are no planar $K_2$-hypohamiltonian graphs of higher girth. 

Similarly, we can extend our algorithm to exhaustively search for all bipartite $K_2$-hypohamiltonian graphs. Since we start from a cycle and a $K_2$ in the algorithm there are two options. If the order is odd, the cycle will be odd and the graph can never become bipartite. If the order is even, the graph is bipartite and after adding the first edge with endpoint on the cycle and on the $K_2$ the colour classes are completely fixed. We can then add all edges between vertices of the same colour class to the list of forbidden edges. Our generator will then only add edges that preserve bipartiteness, which allows to advance several orders further than in the general case.

For a bipartite $K_2$-hypohamiltonian graph $G$ it must hold that $G$ is \emph{balanced}, i.e.\ both colour classes have the same number of vertices, hence that every vertex-deleted subgraph is non-hamiltonian. A hypohamiltonian graph cannot be bipartite, since we cannot have for all vertices that their removal yields a balanced bipartite graph. However, a question by Gr\"otschel~\cite[Problem 4.56]{GGL95} in the same vein asked whether there exist bipartite hypotraceable graphs, i.e.\ a graph for which every vertex-deleted subgraph has a hamiltonian path, but which itself does not contain one.

Our implementation yields the following results.
\begin{proposition}
	A smallest bipartite $K_2$-hypohamiltonian graph has order at least $30$. If its girth is at least $6$ its order is at least $32$. If its girth is at least $8$ its order is at least $36$.
\end{proposition}

The correctness of our generation algorithm was proven in Theorem~\ref{thm:inductionProof}. To verify we did not make any implementation mistakes, we performed various correctness tests.

In~\cite{GRWZ22} we already gave counts for $K_2$-hypohamiltonian graphs for given lower bounds on the girth. Our results coincide. 

It is important to note that we use the same routines for checking hamiltonicity and $K_2$-hypohamiltonicity as in that paper. As mentioned in~\cite{GRWZ22}, these were extensively tested using independent programs.

It is easy to adapt our algorithm to only generate graphs with minimum degree at least $x$ and maximum degree at most $y$. Using this we can generate cubic $K_2$-hypohamiltonian graphs. Since there exist very efficient generators for cubic graphs (e.g.\ \texttt{snarkhunter}~\cite{BGM11}), using one of these and then filtering the $K_2$-hypohamiltonian graphs is much faster. Counts of cubic $K_2$-hypohamiltonian graphs were given in~\cite{GRWZ22}. We adapted our algorithm and generated all cubic $K_2$-hypohamiltonian graphs up to order 22. The results are the same in both cases.

Our implementation of the algorithm is open source software and can be found on GitHub~\cite{GRZ23} where it can be verified and used by other researchers.

\section{Infinite families of \texorpdfstring{\boldmath{$K_2$}}{K2}-hypohamiltonian graphs}\label{sec:operation}
\subsection{Amalgam of \texorpdfstring{\boldmath{$K_2$}}{K2}-hypohamiltonian graphs}
Next to the exhaustive generation of $K_2$-hypohamiltonian graphs, operations preserving $K_2$-hypohamiltonicity are of interest in order to describe infinite families. Here, we introduce an operation creating graphs from smaller ones which will preserve $K_2$-hypohamiltonicity as well as planarity under certain conditions. Using this we improve a result from~\cite{GRWZ22} attempting to characterise the orders for which planar $K_2$-hypohamiltonian graphs exist.

 Let $G$ be a graph and $a,a',b,b'\in V(G)$ be pairwise distinct vertices. We say the tuple $(G, a, a', b, b')$ \textit{satisfies the gluing property} if
\begin{itemize}
 \item $a$ and $b$ have degree three and $a'$ and $b'$ have degree at least three,
 \item $aa', bb', ab\in E(G)$ and $ab', a'b, a'b' \not\in E(G)$.
\end{itemize} 
Let $(G_i, a_i, a'_i, b_i, b'_i)$, for $i=1,2$, be two tuples satisfying the gluing property and $E_1 = E(G_1) \setminus \{a_1b_1,b_1b'_1\}$ and $E_2=E(G_2)\setminus \{ a_2b_2 ,b_2b'_2\}$. Take the disconnected graph
 $$\left(V(G_1)\cup V(G_2), E_1 \cup E_2 \cup \{b_1b_2, b_1b'_2, b'_1b_2\}\right)$$
 and identify $a_1$ with $a_2$ and $a'_1$ with $a'_2$ to get a connected graph $G$. We say $G$ is the \textit{amalgam} of $G_1$ and $G_2$ if it is clear which tuples we started with and we define the vertices $a:=a_1 = a_2$, $a':=a'_1=a'_2$ of $G$. 
 
 \begin{figure}[!htb]
	\centering
	\begin{tikzpicture}[vertex/.style = {circle,fill=black,minimum size=5pt,inner sep=0pt}]
		\node[vertex](a) at (0,0) {};
		\path (a) ++(0,0.3) node {$a$};
		\node (al) at (-1.5,0) {};
		\node (ar) at (1.5,0) {};
		\node[vertex] (a') at (0,-1) {};
		\path (a') ++(0.25,0.35) node {$a'$};
		\node (a'lb) at (-1.5, -0.3) {};
		\node (a'lo) at (-1.5,-1.3) {};
		\node[label={20:\vdots}] (a'lm) at (-1.5,-1.1) {};
		\node (a'rb) at (1.5, -0.3) {};
		\node (a'ro) at (1.5,-1.3) {};
		\node[label={160:\vdots}] (a'rm) at (1.5,-1.1) {};
		\node[vertex] (b'1) at (-0.5, -2) {};
		\path (b'1) ++(0,0.35) node {$b'_1$};
		\node[vertex] (b'2) at (0.5, -2) {};
		\path (b'2) ++(0,0.35) node {$b'_2$};
		\node[vertex] (b1) at (-0.5, -3) {};
		\path (b1) ++(-0.1,0.3) node {$b_1$};
		\node[vertex] (b2) at (0.5, -3) {};
		\path (b2) ++(0.1,0.3) node {$b_2$};
		\node (b1l) at (-1.5,-3) {};
		\node (b2r) at (1.5, -3) {};
		\node (b'1lb) at (-1.5, -1.5) {};
		\node (b'1lo) at (-1.5, -2.5) {};
		\node[label={20:\vdots}] (b'1lm) at (-1.5,-2.3) {};
		\node (b'2rb) at (1.5, -1.5) {};
		\node (b'2ro) at (1.5, -2.5) {};
		\node[label={160:\vdots}] (b'2rm) at (1.5,-2.3) {};
		\draw (al) to (a) to (ar);
		\draw (a) to (a');
		\draw (a') to (a'lb);
		\draw (a') to (a'lo);
		\draw (a') to (a'lm);
		\draw (a') to (a'rb);
		\draw (a') to (a'ro);
		\draw (a') to (a'rm);
		\draw (b'1) to (b'1lb);
		\draw (b'1) to (b'1lo);
		\draw (b'1) to (b'1lm);
		\draw (b'2) to (b'2rb);
		\draw (b'2) to (b'2ro);
		\draw (b'2) to (b'2rm);
		\draw (b'1) to (b2) to (b1) to (b'2);
		\draw (b1l) to (b1);
		\draw (b2r) to (b2);
		\draw[rounded corners] (-1.35,-3.2) rectangle (-3,0.2);
		\draw[rounded corners] (1.35,-3.2) rectangle (3,0.2);
		\node at (-2.2,-1.5) {$G[L]$};
		\node at (2.2,-1.5) {$G[R]$};
	\end{tikzpicture}
	\caption{Visualisation of the amalgam $G$ of $G_1$ and $G_2$.}
\end{figure}
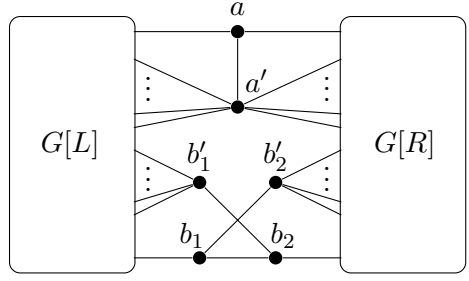
 
 In what follows, we show that under certain constraints this operation preserves non-hamiltonicity, $K_2$-hamiltonicity and planarity. To aid in the notation of the proofs, we define subsets of $V(G)$: $$L := V(G_1)\setminus \{a_1, a'_1, b_1, b'_1\}, \quad R:= V(G_2)\setminus \{a_2, a'_2, b_2, b'_2\}.$$

Let $S$ be a subset of $V(G) - L - R$ and $\mathfrak{h}$ be a cycle in $G$. Suppose we have a pair of edges $(lx, ry)$ such that $l\in L$, $r\in R$ and $x,y\in S$ such that $lx$ and $ry$ are edges in $\mathfrak{h}$ and $x = y$ or there is an $xy$-path in $\mathfrak{h}$ containing only vertices in $S$. Then we call this pair a \textit{traversal} of $\mathfrak{h}$ through $S$.

 Let $s_1,\ldots, s_4$ be all pairwise different traversals of some cycle $\mathfrak{h}$ through $\{a, a'\}$ up to symmetry, i.e.\ up to switching $L$ and $R$, defined in Fig.~\ref{fig:passingOptions}. Let $t_1,\ldots,t_4$ be all pairwise different traversals of some cycle $\mathfrak{h}$ through $\{b_1,b'_1,b_2,b'_2\}$ up to symmetry defined in Fig.~\ref{fig:passingOptions}.
For $s_i$, $i=1,\ldots,4$, we define the \textit{traversal number} $\phi(s_i)$ to be the number of traversals of $\mathfrak{h}$ through $\{a, a'\}$, e.g.\ $\phi(s_1) = 2$, $\phi(s_3) = 0$. Similarly for $t_i$, $i=1,\ldots,4$ we define $\phi(t_i)$ the be the number of traversals of $\mathfrak{h}$ through $\{b_1,b'_1,b_2,b'_2\}$, e.g.\ $\phi(t_1) = 1$.
 
\begin{figure}[htb]
\centering
\begin{minipage}{29mm}
\centering
		\begin{tikzpicture}[vertex/.style={circle,fill=black,minimum size=5pt,inner sep=0pt}]
		
			\node[vertex, label={[label distance=-2.5pt]$a$}](a) at (0,0) {};
			\node[vertex] (onder) at (0,-1) {};
			\path (a') ++(0.2,0.35) node {$a'$};
			\draw[color=gray] (a) to (a');
			\node (l) at (-1,0) {};
			\draw[color=gray] (l) to (a);
			\node (r) at (1,0) {};
			\draw[color=gray] (r) to (a);
			
			\node (lb) at (-1, -1.25) {};
			\draw[color=gray] (lb) to (a');
			\node (ls) at (-1, -1) {};
			\draw[color=gray] (ls) to (a');
			\node[color=gray] at (-0.8,-0.6) {$\vdots$};
			\node (lt) at (-1, -0.25) {};
			\draw[color=gray] (lt) to (a');
			
			\node (rb) at (1, -1.25) {};
			\draw[color=gray] (rb) to (a');
			\node (rs) at (1, -1) {};
			\draw[color=gray] (rs) to (a');
			\node[color=gray] at (0.8,-0.6) {$\vdots$};
			\node (rt) at (1, -0.25) {};
			\draw[color=gray] (rt) to (a');
			
			\draw[ultra thick] (l) to (a.center) to (r);
			\draw[ultra thick] (ls) to (onder.center) to (rs);
			
		\end{tikzpicture}\\[-3mm]
		$s_1$
\end{minipage}
\begin{minipage}{29mm}
\centering
		\begin{tikzpicture}[vertex/.style={circle,fill=black,minimum size=5pt,inner sep=0pt}]
			\node[vertex, label={[label distance=-2.5pt]$a$}](a) at (0,0) {};
			\node[vertex] (onder) at (0,-1) {};
			\path (a') ++(0.2,0.35) node {$a'$};
			\draw[color=gray] (a) to (a');
			\node (l) at (-1,0) {};
			\draw[color=gray] (l) to (a);
			\node (r) at (1,0) {};
			\draw[color=gray] (r) to (a);
			
			\node (lb) at (-1, -1.25) {};
			\draw[color=gray] (lb) to (a');
			\node (ls) at (-1, -1) {};
			\draw[color=gray] (ls) to (a');
			\node[color=gray] at (-0.8,-0.6) {$\vdots$};
			\node (lt) at (-1, -0.25) {};
			\draw[color=gray] (lt) to (a');
			
			\node (rb) at (1, -1.25) {};
			\draw[color=gray] (rb) to (a');
			\node (rs) at (1, -1) {};
			\draw[color=gray] (rs) to (a');
			\node[color=gray] at (0.8,-0.6) {$\vdots$};
			\node (rt) at (1, -0.25) {};
			\draw[color=gray] (rt) to (a');
			
			\draw[ultra thick] (l) to (a.center) to (r);
			\draw[ultra thick] (lt) to (a'.center) to (lb);
		\end{tikzpicture}
		\\[-3mm]
		$s_2$
\end{minipage}
\begin{minipage}{29mm}	
\centering
		\begin{tikzpicture}[vertex/.style={circle,fill=black,minimum size=5pt,inner sep=0pt}]
			\node[vertex, label={[label distance=-2.5pt]$a$}](a) at (0,0) {};
			\node[vertex] (onder) at (0,-1) {};
			\path (a') ++(0.2,0.35) node {$a'$};
			\draw[color=gray] (a) to (a');
			\node (l) at (-1,0) {};
			\draw[color=gray] (l) to (a);
			\node (r) at (1,0) {};
			\draw[color=gray] (r) to (a);
			
			\node (lb) at (-1, -1.25) {};
			\draw[color=gray] (lb) to (a');
			\node (ls) at (-1, -1) {};
			\draw[color=gray] (ls) to (a');
			\node[color=gray] at (-0.8,-0.6) {$\vdots$};
			\node (lt) at (-1, -0.25) {};
			\draw[color=gray] (lt) to (a');
			
			\node (rb) at (1, -1.25) {};
			\draw[color=gray] (rb) to (a');
			\node (rs) at (1, -1) {};
			\draw[color=gray] (rs) to (a');
			\node[color=gray] at (0.8,-0.6) {$\vdots$};
			\node (rt) at (1, -0.25) {};
			\draw[color=gray] (rt) to (a');
			
			\draw[ultra thick] (l) to (a.center) to (a'.center) to (ls);
		\end{tikzpicture}\\[-3mm]
		$s_3$
\end{minipage}
\begin{minipage}{29mm}
\centering
		\begin{tikzpicture}[vertex/.style={circle,fill=black,minimum size=5pt,inner sep=0pt}]
			\node[vertex, label={[label distance=-2.5pt]$a$}](a) at (0,0) {};
			\node[vertex] (onder) at (0,-1) {};
			\path (a') ++(0.2,0.35) node {$a'$};
			\draw[color=gray] (a) to (a');
			\node (l) at (-1,0) {};
			\draw[color=gray] (l) to (a);
			\node (r) at (1,0) {};
			\draw[color=gray] (r) to (a);
			
			\node (lb) at (-1, -1.25) {};
			\draw[color=gray] (lb) to (a');
			\node (ls) at (-1, -1) {};
			\draw[color=gray] (ls) to (a');
			\node[color=gray] at (-0.8,-0.6) {$\vdots$};
			\node (lt) at (-1, -0.25) {};
			\draw[color=gray] (lt) to (a');
			
			\node (rb) at (1, -1.25) {};
			\draw[color=gray] (rb) to (a');
			\node (rs) at (1, -1) {};
			\draw[color=gray] (rs) to (a');
			\node[color=gray] at (0.8,-0.6) {$\vdots$};
			\node (rt) at (1, -0.25) {};
			\draw[color=gray] (rt) to (a');
			
			\draw[ultra thick] (r) to (a.center) to (a'.center) to (ls);
		\end{tikzpicture}
		\\[-3mm]
		$s_4$
\end{minipage}

\begin{minipage}{29mm}
\centering
		\begin{tikzpicture}[vertex/.style={circle,fill=black,minimum size=5pt,inner sep=0pt}]
			\node[vertex, label={[label distance=-2.5pt]$b_1$}] (b1) at (-0.5,0) {};
			\node[vertex, label={[label distance=-2.5pt]$b_2$}] (b2) at (0.5,0) {};
			\node[vertex, label={[label distance=-2.5pt]$b'_1$}] (bi1) at (-0.5,1) {};
			\node[vertex, label={[label distance=-2.5pt]$b'_2$}] (bi2) at (0.5,1) {};
			\draw[color=gray] (bi1) to (b2) to (b1) to (bi2);
			\node (l) at (-1.25,0) {};
			\draw[color=gray] (l) to (b1); 
			\node (r) at (1.25,0) {};
			\draw[color=gray] (r) to (b2);
			\node (lt) at (-1.25,1.25) {};
			\draw[color=gray] (lt) to (bi1);
			\node (ls) at (-1.25,1) {};
			\draw[color=gray] (ls) to (bi1);
			\node[color=gray] at (-1.1, 0.8) {$\vdots$};
			\node (lb) at (-1.25,0.25) {};
			\draw[color=gray] (lb) to (bi1);
			
			\node (rt) at (1.25,1.25) {};
			\draw[color=gray] (rt) to (bi2);
			\node (rs) at (1.25,1) {};
			\draw[color=gray] (rs) to (bi2);
			\node[color=gray] at (1.1, 0.8) {$\vdots$};
			\node (rb) at (1.25,0.25) {};
			\draw[color=gray] (rb) to (bi2);
			
			\draw[ultra thick] (l) to (b1.center) to (b2.center) to (r);
			\draw[ultra thick] (lt) to (bi1.center) to (lb);
			\draw[ultra thick] (rt) to (bi2.center) to (rb);
		\end{tikzpicture}\\
		$t_1$
\end{minipage}
\begin{minipage}{29mm}
\centering
		\begin{tikzpicture}[vertex/.style={circle,fill=black,minimum size=5pt,inner sep=0pt}]
			\node[vertex, label={[label distance=-2.5pt]$b_1$}] (b1) at (-0.5,0) {};
			\node[vertex, label={[label distance=-2.5pt]$b_2$}] (b2) at (0.5,0) {};
			\node[vertex, label={[label distance=-2.5pt]$b'_1$}] (bi1) at (-0.5,1) {};
			\node[vertex, label={[label distance=-2.5pt]$b'_2$}] (bi2) at (0.5,1) {};
			\draw[color=gray] (bi1) to (b2) to (b1) to (bi2);
			\node (l) at (-1.25,0) {};
			\draw[color=gray] (l) to (b1); 
			\node (r) at (1.25,0) {};
			\draw[color=gray] (r) to (b2);
			\node (lt) at (-1.25,1.25) {};
			\draw[color=gray] (lt) to (bi1);
			\node (ls) at (-1.25,1) {};
			\draw[color=gray] (ls) to (bi1);
			\node[color=gray] at (-1.1, 0.8) {$\vdots$};
			\node (lb) at (-1.25,0.25) {};
			\draw[color=gray] (lb) to (bi1);
			
			\node (rt) at (1.25,1.25) {};
			\draw[color=gray] (rt) to (bi2);
			\node (rs) at (1.25,1) {};
			\draw[color=gray] (rs) to (bi2);
			\node[color=gray] at (1.1, 0.8) {$\vdots$};
			\node (rb) at (1.25,0.25) {};
			\draw[color=gray] (rb) to (bi2);
			\draw[ultra thick] (ls) to (bi1.center) to (b2.center) to (b1.center) to (bi2.center) to (rs);
		\end{tikzpicture}\\
		$t_2$
\end{minipage}
\begin{minipage}{29mm}
\centering
		\begin{tikzpicture}[vertex/.style={circle,fill=black,minimum size=5pt,inner sep=0pt}]
			\node[vertex, label={[label distance=-2.5pt]$b_1$}] (b1) at (-0.5,0) {};
			\node[vertex, label={[label distance=-2.5pt]$b_2$}] (b2) at (0.5,0) {};
			\node[vertex, label={[label distance=-2.5pt]$b'_1$}] (bi1) at (-0.5,1) {};
			\node[vertex, label={[label distance=-2.5pt]$b'_2$}] (bi2) at (0.5,1) {};
			\draw[color=gray] (bi1) to (b2) to (b1) to (bi2);
			\node (l) at (-1.25,0) {};
			\draw[color=gray] (l) to (b1); 
			\node (r) at (1.25,0) {};
			\draw[color=gray] (r) to (b2);
			\node (lt) at (-1.25,1.25) {};
			\draw[color=gray] (lt) to (bi1);
			\node (ls) at (-1.25,1) {};
			\draw[color=gray] (ls) to (bi1);
			\node[color=gray] at (-1.1, 0.8) {$\vdots$};
			\node (lb) at (-1.25,0.25) {};
			\draw[color=gray] (lb) to (bi1);
			
			\node (rt) at (1.25,1.25) {};
			\draw[color=gray] (rt) to (bi2);
			\node (rs) at (1.25,1) {};
			\draw[color=gray] (rs) to (bi2);
			\node[color=gray] at (1.1, 0.8) {$\vdots$};
			\node (rb) at (1.25,0.25) {};
			\draw[color=gray] (rb) to (bi2);
			
			\draw[ultra thick] (ls) to (bi1.center) to (b2.center) to (b1.center) to (l);
			\draw[ultra thick] (rt) to (bi2.center) to (rb);
		\end{tikzpicture}\\
		$t_3$
\end{minipage}
\begin{minipage}{29mm}
\centering
		\begin{tikzpicture}[vertex/.style={circle,fill=black,minimum size=5pt,inner sep=0pt}]
			\node[vertex, label={[label distance=-2.5pt]$b_1$}] (b1) at (-0.5,0) {};
			\node[vertex, label={[label distance=-2.5pt]$b_2$}] (b2) at (0.5,0) {};
			\node[vertex, label={[label distance=-2.5pt]$b'_1$}] (bi1) at (-0.5,1) {};
			\node[vertex, label={[label distance=-2.5pt]$b'_2$}] (bi2) at (0.5,1) {};
			\draw[color=gray] (bi1) to (b2) to (b1) to (bi2);
			\node (l) at (-1.25,0) {};
			\draw[color=gray] (l) to (b1); 
			\node (r) at (1.25,0) {};
			\draw[color=gray] (r) to (b2);
			\node (lt) at (-1.25,1.25) {};
			\draw[color=gray] (lt) to (bi1);
			\node (ls) at (-1.25,1) {};
			\draw[color=gray] (ls) to (bi1);
			\node[color=gray] at (-1.1, 0.8) {$\vdots$};
			\node (lb) at (-1.25,0.25) {};
			\draw[color=gray] (lb) to (bi1);
			
			\node (rt) at (1.25,1.25) {};
			\draw[color=gray] (rt) to (bi2);
			\node (rs) at (1.25,1) {};
			\draw[color=gray] (rs) to (bi2);
			\node[color=gray] at (1.1, 0.8) {$\vdots$};
			\node (rb) at (1.25,0.25) {};
			\draw[color=gray] (rb) to (bi2);
			
			\draw[ultra thick] (ls) to (bi1.center) to (b2.center) to (r);
			\draw[ultra thick] (rs) to (bi2.center) to (b1.center) to (l);
		\end{tikzpicture}\\
		$t_4$
\end{minipage}
 	\caption{All pairwise different traversals of some cycle $\mathfrak{h}$ (in black) through $\{a, a'\}$ (top row) and through $\{b_1, b'_1, b_2, b'_2\}$ (bottom row) up to symmetry. Edges to the left are going to some vertex in $L$, edges to the right are going to some vertex in $R$. \label{fig:passingOptions}}
\end{figure}
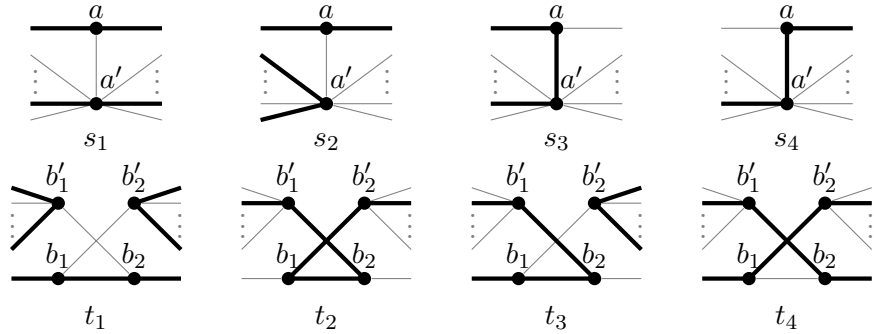

 \begin{lemma}\label{lem:nonHamiltonian}
 	Let $(G_i, a_i, a'_i, b_i, b'_i)$, for $i=1,2$, be two tuples satisfying the gluing property such that $G_1$ and $G_2$ are non-hamiltonian. Then the amalgam $G$ of $G_1$ and $G_2$ is non-hamiltonian.
 \end{lemma}
 \begin{proof}
 	Suppose that $G$ does contain a hamiltonian cycle $\mathfrak{h}$. It must pass from $L$ to $R$ and back (at least once) via precisely one $s_i$ and one $t_i$ configuration of Fig.~\ref{fig:passingOptions}. For $i,j=1,\ldots,4$, we have that $\phi(s_i)+\phi(t_j)$ must be even and at least $2$ for all such $i$ and $j$. It remains to show all possible combinations up to symmetry lead to a contradiction.
 	
 	Suppose we are in the case of $s_1$ and $t_3$. Then $\mathfrak{h}$ consists of two $aa'$-paths, one of which contains $a$, $a'$, all vertices of $L$ and $b'_1, b_2$ and $b_1$. Removing $b_2$ and adding edges $aa'$ and $b_1b'_1$ to this path yields a hamiltonian cycle in $G_1$, a contradiction. In the case of $s_1$ and $t_4$, we can split $\mathfrak{h}$ into four paths, two of which are a $k_1l_1$-path and a $k_2l_2$-path where all interior vertices are in $L$ and $\{k_1,k_2\}=\{a,a'\}$ and $\{l_1,l_2\} = \{b_1, b'_1\}$. Adding $aa'$ and $b_1b'_1$ to the union of these two paths yields a hamiltonian cycle in $G_1$.
 	
 In the case of $s_2$ and $t_1$, we get an $ab_1$-subpath of $\mathfrak{h}$ where the interior vertices are vertices of $L$ or $a'$ or $b'_1$. Adding $ab_1$ to this path yields a hamiltonian cycle in $G_1$. In the case of $s_2$ and $t_2$ we get an $ab_1$-subpath of $\mathfrak{h}$ in which the interior vertices are vertices of $L$, $a'$, $b'_1$ or $b_2$. Removing $b_2$ and adding edges $b_1b'_1$ and $ab_1$ to this path yields a hamiltonian cycle in $G_1$.
 
In the case of $s_3$ and $t_4$, we get a $b_1b'_1$-subpath of $\mathfrak{h}$ in which the interior vertices are $a$, $a'$ and vertices of $L$. Adding edge $b_1b'_1$ to this path yields a hamiltonian cycle in $G_1$. 

In the case of $s_4$ and $t_1$, we get an $ab_1$-subpath of $\mathfrak{h}$ in which the interior vertices are $a', b'_1$ and vertices of $L$. Adding edge $ab_1$ to this path gives us a hamiltonian cycle in $G_1$. In the case of $s_4$ and $t_2$, we get an $ab_1$-subpath of $\mathfrak{h}$ in which the interior vertices are $a'$, $b'_1$, $b_2$ or vertices of $L$. Removing $b_2$ and adding edges $ab_1$ and $b_1b'_1$ to this path yields a hamiltonian cycle in $G_1$. 

Since the unmentioned cases have either $\phi(s_i)+\phi(t_j)$ odd or zero, we conclude that $G$ must be non-hamiltonian. 
 \end{proof}
 
 In the following proofs, we will often create a cycle $C$ in $G$ from cycles or paths in $G_1$ and $G_2$. To make this more concise we define the following notation. Let $H_1$ be a subgraph of $G_1$; $H_2$ be a subgraph of $G_2$; and $e_1, \ldots, e_m$ be edges in $G$. We can identify $H_1 - b_1b'_1 - a_1b_1$ with a subgraph of $G$. Similarly, we can identify $H_2 - b_2b'_2 - a_2b_2$ with a subgraph in $G$. We define $C(H_1,H_2, e_1, \ldots, e_m)$ to be the subgraph in $G$, which is the union of all these identified subgraphs in $G$ together with the edges $e_1, \ldots, e_m$. We allow this set of edges to be empty in which case we write $C(H_1,H_2)$.
 
\begin{lemma}\label{lem:K2Hamiltonian}
	Let $(G_i, a_i, a'_i, b_i, b'_i)$, for $i=1,2$, be two tuples satisfying the gluing property such that $G_1$ and $G_2$ are $K_2$-hamiltonian, $G_i - a'_i$ has at least two hamiltonian cycles, one containing the edge $b_ib'_i$ and one which does not, and $G_i - v$ is hamiltonian for all $v\in N_{G_i}[b_i]$. Then the amalgam $G$ of $G_1$ and $G_2$ is $K_2$-hamiltonian.
\end{lemma}
\begin{proof}
	We show for all adjacent $u,v\in V(G)$ that $G - u - v$ is hamiltonian.
	
	Suppose $u,v\not\in\{a,a',b_i,b'_i \mid i=1,2\}$. Without loss of generality assume that $u,v\in L$. There exists a hamiltonian cycle $\mathfrak{h}_1$ in $G_1 - u - v$. Suppose $a_1b_1\not\in E(\mathfrak{h}_1)$. Then $\mathfrak{h}_1$ must contain $b_1b'_1$ as $b_1$ has degree $3$.
	 We also have a hamiltonian cycle $\mathfrak{h}_2$ in $G_2 - a_2 - a'_2$, which must contain $b_2b'_2$. Then $C(\mathfrak{h}_1, \mathfrak{h}_2, b_1b'_2, b'_1b_2)$ is a hamiltonian cycle in $G - u - v$.
	 Conversely, if $\mathfrak{h}_1$ does contain $a_1b_1$, it can either contain $b_1b'_1$ or not. Suppose the former, then let $\mathfrak{h}_2$ be a hamiltonian cycle containing $b_2b'_2$ in $G_2 - a'_2$. Then $C(\mathfrak{h}_1, \mathfrak{h}_2,b_1b_2,b_1b'_2,b'_1b_2)$ is a hamiltonian cycle in $G - u - v$. Finally, suppose that $\mathfrak{h}_1$ contains $a_1b_1$ but not $b_1b'_1$. Let $\mathfrak{h}_2$ be a hamiltonian cycle in $G_2 - a'_2$ not containing $b_2b'_2$, then $C(\mathfrak{h}_1, \mathfrak{h}_2, b_1b_2)$ is a hamiltonian cycle in $G - u - v$.
	
Suppose $u=a$ and $v\in N_G(a)$. Without loss of generality assume that $v\in L \cup \{a'\}$. Let $\mathfrak{h}_1, \mathfrak{h}_2$ be hamiltonian cycles in $G_1 - a_1 - v$ and $G_2 - a_2 - a'_2$, respectively. Then $C(\mathfrak{h}_1, \mathfrak{h}_2, b_1b'_2, b'_1b_2)$ is a hamiltonian cycle in $G - u - v$.

Suppose $u\neq a$, $v = a'$. Without loss of generality assume that $u\in L$. Let $\mathfrak{h}_1$ be a hamiltonian cycle in $G_1 - u - a'_1$. It can either contain $b_1b'_1$ or not. In the former case, let $\mathfrak{h}_2$ be a hamiltonian cycle in $G_2 - a'_2$ containing $b_2b'_2$. Then $C(\mathfrak{h}_1, \mathfrak{h}_2,b_1b_2,b_1b'_2,b'_1b_2)$ is a hamiltonian cycle in $G - u - v$. In the latter case, let $\mathfrak{h}_2$ be a hamiltonian cycle in $G_2 - a'_2$ not containing $b_2b'_2$. Then $C(\mathfrak{h}_1, \mathfrak{h}_2, b_1b_2)$ is a hamiltonian cycle in $G - u - v$.

Suppose $u = b_1$ and $v = b_2$. Let $\mathfrak{h}_1, \mathfrak{h}_2$ be hamiltonian cycles in $G_1 - b_1$ and $G_2 - b_2$, respectively. Then $C(\mathfrak{h}_1, \mathfrak{h}_2) - aa'$ is a hamiltonian cycle in $G - u - v$.

Suppose $u = b_1$ and $v = b'_2$. Let $\mathfrak{h}_1$ be a hamiltonian cycle in $G_1 - a'_1$ containing $b_1b'_1$ and let $\mathfrak{h}_2$ be a hamiltonian cycle in $G_2 - b'_2$. Then $C(\mathfrak{h}_1, \mathfrak{h}_2, b'_1b_2) - b_1$ is a hamiltonian cycle in $G - u - v$.

Suppose $u = b_1$ and $v\not\in\{b_2, b'_2\}$. Then $v\in L$. Let $\mathfrak{h}_1$ be a hamiltonian cycle in $G_1 - v$ and let $\mathfrak{h}_2$ be a hamiltonian cycle in $G_2 - a'_2$ not containing $b_2b'_2$. Then $C(\mathfrak{h}_1, \mathfrak{h}_2, b'_1b_2) - b_1$ is a hamiltonian cycle in $G - u - v$.

Finally, let $u\neq b_2$ and $v = b'_1$. Then $u\in L$. Let $\mathfrak{h}_1$ be a hamiltonian cycle in $G_1 - u - v$  and $\mathfrak{h}_2$ be a hamiltonian cycle in $G_2 - a'_2$ not containing $b_2b'_2$. Then $C(\mathfrak{h}_1, \mathfrak{h}_2, b_1b_2)$ is a hamiltonian cycle in $G - u - v$. 

Since all unmentioned cases are symmetric to the ones above, we conclude that $G$ is $K_2$-hamiltonian.
\end{proof}

\begin{lemma}\label{lem:preserveEdges}
	Let $(G_i, a_i, a'_i, b_i, b'_i)$, for $i=1,2$, be two tuples satisfying the conditions of Lemma~\ref{lem:nonHamiltonian} and $G$ be the amalgam of $G_1$ and $G_2$. For $v\in V(G_i - b_i)$, if $G_i - v$ is hamiltonian, so is $G - v$. Moreover, for any edge $e\in E(G_i - b_i)$ which is also an edge of $G$, if there is a hamiltonian cycle in $G_i - v$ containing $e$ (not containing $e$), then there exists a hamiltonian cycle in $G - v$ containing $e$ (not containing $e$). Lastly, $G - b_i$ is non-hamiltonian.
\end{lemma}
\begin{proof}
Without loss of generality, let $i = 1$. Let $v\in V(G_1 - b_1)$ and suppose $\mathfrak{h}_1$ is a hamiltonian cycle in $G_1 - v$. 
	
	Assume $a_1b_1\not\in E(\mathfrak{h}_1)$. Let $\mathfrak{h}_2$ be a hamiltonian cycle in $G_2 - a_2 - a'_2$. We see that $C(\mathfrak{h}_1, \mathfrak{h}_2, b_1b'_2,b'_1b_2)$ is a hamiltonian cycle in $G - v$. 	
	Conversely, suppose that $a_1b_1\in E(\mathfrak{h}_1)$. Then either $b_1b'_1\in E(\mathfrak{h}_1)$ or not. Suppose the former and let $\mathfrak{h}_2$ be a hamiltonian cycle in $G_2 - a'_2$ containing $b_2b'_2$. Then $C(\mathfrak{h}_1, \mathfrak{h}_2, b_1b_2, b_1b'_2, b'_1b_2)$ is a hamiltonian cycle in $G - v$. Suppose the latter and let $\mathfrak{h}_2$ be a hamiltonian cycle in $G_2 - a'_2$ not containing $b_2b'_2$. Then $C(\mathfrak{h}_1, \mathfrak{h}_2, b_1b_2)$ is a hamiltonian cycle in $G - v$.
	
	Note that by construction of these cycles any edge of $G_1$ on $\mathfrak{h}_1$ which is an edge in $G$ lies on the hamiltonian cycle of $G - v$ and any edge of $G_1$ not on $\mathfrak{h}_1$ does not lie on the hamiltonian cycle in $G - v$.
	
	Now suppose there exists a hamiltonian cycle $\mathfrak{h}$ in $G - b_1$. Then $b'_1b_2\in E(\mathfrak{h})$ and $ub'_2w$ is a subpath of $\mathfrak{h}$, where $u$ and $w$ are vertices of $R$. Denote the traversals of $\mathfrak{h}$ through $\{b'_1,b_2,b'_2\}$ by $t_5$. Then $\phi(t_5) = 1$. We show that all a priori possible traversals of $\mathfrak{h}$ through $\{a, a'\}$ lead to a contradiction. Since $\phi(s_i) + \phi(t_5)$ is even and at least $2$, we only need to look at $s_2$, $s_4$ and their mirror images, i.e.\ with $L$ and $R$ swapped.
	
	Assume we are in case $s_2$ or $s_4$ and $t_5$, then the intersection of $\mathfrak{h}$ with $G_1$ yields an $a_1b'_1$-path in which the interior vertices are precisely $a'_1$ and the vertices of $L$. Adding edges $b_1b'_1$ and $a_1b_1$  to this path yields a hamiltonian cycle in $G_1$, a contradiction.
	
	Suppose we are in the case where $s_2$ is mirrored or $s_4$ is mirrored and $t_5$. Then the intersection of $\mathfrak{h}$ with $G_2$ yields an $a_2b_2$-path in which the interior vertices are precisely $a'_2$, $b'_2$ and all vertices of $R$. Adding edge $a_2b_2$ to this path leads to a hamiltonian cycle in $G_2$, a contradiction.
\end{proof}

This lemma shows that the operation, when applied to hypohamiltonian graphs, never yields a hypohamiltonian graph.

 Let $G$ be a planar graph and let $V_F$ be the set of all vertices that lie on the boundary of a face $F$ in a plane embedding of $G$. We say vertices $v_0, \ldots, v_{k-1} \in V(G)$ are \emph{co-facial} in a plane embedding of $G$ if they all belong to the same set $V_F$ for some face $F$ of the embedding. We say $v_0,\ldots,v_{k-1}$ \emph{appear in order} in $F$ if for every  $i=0,\ldots, k-1$ there is a $v_iv_{i+1}$-path, taking indices $\bmod\:k$, containing only edges on the boundary of $F$ such that this path does not contain any other $v_j$ for $j=0,\ldots, k-1$. 
 
 \begin{lemma}\label{lem:planar}
 	Let $(G_i, a_i, a'_i, b_i, b'_i)$, for $i=1,2$, be two tuples satisfying the gluing property such that $G_1$ and $G_2$ are planar. If $G_1$ admits a plane embedding such that $a'_1, a_1, b_1, b'_1$ are co-facial and appear in this order and $G_2$ admits a plane embedding such that $a_2, a'_2, b_2, b'_2$ are not co-facial, then the amalgam $G$ of $G_1$ and $G_2$ is a planar graph.
 \end{lemma}
 \begin{proof}
 As $a_2$ and $b_2$ have degree $3$ in $G_2$, $a'_2, a_2, b_2$ and $a_2, b_2, b'_2$ are co-facial in $G_2$ and appear in these orders in their respective faces.
 Take a plane embedding of $G_1$ for which the vertices $a_1, a'_1, b_1, b'_1$ lie on the outer face. Note that this is always possible.
 	Hence, since $G_1$ and $G_2$ are plane, removing $a_ib_i$ and $b_ib'_i$ from $G_i$ for $i=1,2$ yields two plane graphs such that $a_i, a'_i, b_i, b'_i$ are co-facial for $i=1,2$. This is depicted at the top of Figure~\ref{fig:amalgam}. Denote these new graphs by $G'_1$ and $G'_2$, respectively. Identifying $a_1$ with $a_2$ and $a'_1$ with $a'_2$ in the disjoint union of $G'_1$ and $G'_2$ also yields a plane graph in which the following vertices are co-facial and appear in order $a', b'_1, b_1, a, b'_2, b_2$ as can be seen at the bottom of Figure~\ref{fig:amalgam}. Hence, we can add edges $b_1b_2$, $b_1b'_2$ and $b'_1b_2$ without creating any crossings.   
 \end{proof}
 
 \begin{figure}[!htb]
		\centering
		\begin{tikzpicture}[vertex/.style={circle,fill=black,minimum size=5pt,inner sep=0pt}]
		
			\node[vertex, label=right:{$a'_1$}](la') at (-2,2) {};
			\node[vertex, label=right:{$a_1$}](la) at (-2,1) {};
			\node[vertex, label={$b_1$}](lb) at (-2,-0.5) {};
			\node[vertex, label={$b'_1$}](lb') at (-2,-2) {};
			\draw (la') to (la);
			\path (la) ++(-1,0) node (lal) {};
			\draw (la) to (lal);
			\path (lb) ++(-1,0) node (lbl) {};
			\draw (lb) to (lbl);
			\path (la') ++(-1,0.5) node (la't) {};
			\path (la') ++(-1,-0.5) node (la'b) {};
			\path (la') ++(-0.8,0.1) node {$\vdots$};
			\path (la't) ++(0,0.3) node {$G_1 - a_1b_1 - b_1b'_1$};
			\draw (la't) to (la') to (la'b);
			\path (lb') ++(-1,0.5) node (lb't) {};
			\path (lb') ++(-1,-0.5) node (lb'b) {};
			\path (lb') ++(-0.8,0.1) node {$\vdots$};
			\draw (lb't) to (lb') to (lb'b);
			\draw[dashed, bend left] (lb'b) to[out=90, in=90] (la't);
			\draw[dashed, bend left] (lal) to[out=-90, in=-90] (lbl);
			\draw[dashed, bend left] (lbl) to[out=-90, in=-90] (lb't);
		
			\node[vertex, label=left:{$a'_2$}](ra') at (1,2) {};
			\node[vertex, label=right:{$a_2$}](ra) at (1,1) {};
			\node[vertex, label={$b_2$}](rb) at (1,-0.5) {};
			\node[vertex, label={$b'_2$}](rb') at (1,-2) {};
			\draw (ra') to (ra);
			\path (ra') ++(0,0.8) node {$G_2 - a_2b_2 - b_2b'_2$};
			\path (rb) ++(-1,0.3) node {$F$};
			\path (ra) ++(-1,0) node (ral) {};
			\draw (ra) to (ral);
			\path (rb) ++(1,0) node (rbr) {};
			\draw (rb) to (rbr);
			\path (ra') ++(1,0.5) node (ra't) {};
			\path (ra') ++(1,-0.5) node (ra'b) {};
			\path (ra') ++(0.8,0.1) node {$\vdots$};
			\draw (ra't) to (ra') to (ra'b);
			\path (rb') ++(1,0.5) node (rb't) {};
			\path (rb') ++(1,-0.5) node (rb'b) {};
			\path (rb') ++(0.8,0.1) node {$\vdots$};
			\draw (rb't) to (rb') to (rb'b);
			\node (t) at (1,4) {};
			\node (b) at (1,-3) {};
			\draw[dashed, bend left] (t) to[out=-90, in=-90] (b);
			\draw[dashed, bend left] (t) to[out=90, in=90] (b);
			\draw[dashed, bend left] (rbr) to[out=-90, in=-90] (ra'b);
			\draw[dashed, bend left] (ral) to[out=-90, in=-90] (rb'b);
			\draw[dashed, bend left] (rbr) to[out=90, in=90] (rb't);
		\end{tikzpicture}
		\begin{tikzpicture}[vertex/.style={circle,fill=black,minimum size=5pt,inner sep=0pt}]
		
			\node[vertex, label=right:{$a'_1$}](la') at (-2,2) {};
			\node[vertex, label=right:{$a_1$}](la) at (-2,1) {};
			\node[vertex, label={$b_1$}](lb) at (-2,-0.5) {};
			\node[vertex, label={$b'_1$}](lb') at (-2,-2) {};
			\draw (la') to (la);
			\path (la) ++(-1,0) node (lal) {};
			\draw (la) to (lal);
			\path (lb) ++(-1,0) node (lbl) {};
			\draw (lb) to (lbl);
			\path (la') ++(-1,0.5) node (la't) {};
			\path (la') ++(-1,-0.5) node (la'b) {};
			\path (la') ++(-0.8,0.1) node {$\vdots$};
			\path (la't) ++(0,0.3) node {$G_1-a_1b_1 - b_1b'_1$};
			\draw (la't) to (la') to (la'b);
			\path (lb') ++(-1,0.5) node (lb't) {};
			\path (lb') ++(-1,-0.5) node (lb'b) {};
			\path (lb') ++(-0.8,0.1) node {$\vdots$};
			\draw (lb't) to (lb') to (lb'b);
			\draw[dashed, bend left] (lb'b) to[out=90, in=90] (la't);
			\draw[dashed, bend left] (lal) to[out=-90, in=-90] (lbl);
			\draw[dashed, bend left] (lbl) to[out=-90, in=-90] (lb't);
		
			\node[vertex, label=left:{$a'_2$}](ra') at (1,2) {};
			\node[vertex, label=left:{$a_2$}](ra) at (1,1) {};
			\node[vertex, label={$b_2$}](rb) at (0,-2) {};
			\node[vertex, label={$b'_2$}](rb') at (0,-0.5) {};
			\draw (ra') to (ra);
			\path (ra) ++(1,0) node (rar) {};
			\draw (ra) to (rar);
			\path (rb) ++(1,0) node (rbr) {};
			\draw (rb) to (rbr);
			\path (ra') ++(1,0.5) node (ra't) {};
			\path (ra') ++(1,-0.5) node (ra'b) {};
			\path (ra') ++(0.8,0.1) node {$\vdots$};
			\path (ra't) ++(0,0.3) node {$G_2 - a_2b_2 - b_2b'_2$};
			\path(rar) ++(1.5,-0.5) node {$F$};
			\draw (ra't) to (ra') to (ra'b);
			\path (rb') ++(1,0.5) node (rb't) {};
			\path (rb') ++(1,-0.5) node (rb'b) {};
			\path (rb') ++(0.8,0.1) node {$\vdots$};
			\draw (rb't) to (rb') to (rb'b);
			\draw[dashed] (ra) to (rb't);
			\draw[dashed] (rb'b) to (rbr);
			\draw[dashed, bend left] (rbr) to[out=-90, in=-90] (ra't);
	\draw[dotted] (rb') to (lb) to (rb) to (lb');
		\end{tikzpicture}
		\caption{A visualisation of the amalgam of $G_1$ and $G_2$. Dashed lines represent (part of) the boundary of some face in the embedding. The top image depicts $G_1$, with $a'_1, a_1, b_1, b'_1$ lying on the boundary of the outer face, and $G_2$ after the removal of $a_ib_i$ and $b_ib'_i$. We see that $a'_1, a_1, b_1, b'_1$ still lie on the outer face and that $a'_2, a_2, b_2, b'_2$ now lie on the boundary of some face $F$. In the bottom image we have taken the embedding of $G_2 - a_2b_2 - b_2b'_2$ for which $F$ is the outer face. It is now easily seen that the identification of $a'_1$ with $a'_2$ and $a_1$ with $a_2$ leaves a plane graph and that the addition of the dotted edges $b_1b_2, b'_1b_2$ and $b_1 b'_2$ does as well.}
		\label{fig:amalgam}
	\end{figure}
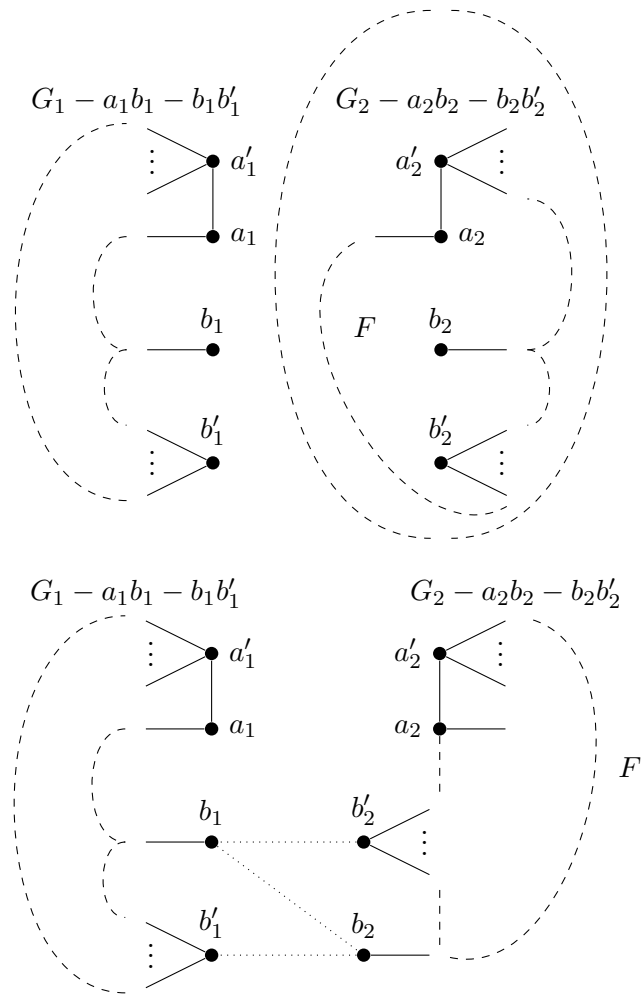
 
 In~\cite{GRWZ22}, the authors showed the existence of a planar $K_2$-hypohamiltonian graph for every order from order $177$ onwards. Using the operation described above, we can improve this bound. 
 
Consider a graph $G$ containing a 5-cycle $C = v_0 \ldots v_4$ such that every $v_i$ is cubic. Denote the neighbour of $v_i$ not on $C$ by $v'_i$. Then the cycle $C$ is called \emph{extendable} if for any $i$, taking indices $\bmod \; 5$, (i)~there exists a hamiltonian cycle $\mathfrak{h}$ in $G - v_i$ with $v_{i-2}v_{i+2} \notin E(\mathfrak{h})$ and (ii)~there exists a hamiltonian cycle $\mathfrak{h}'$ in $G - v'_i$ with $\mathfrak{h}' \cap C = v_{i-2}v_{i-1}v_iv_{i+1}v_{i+2}$. From Lemma~\ref{lem:preserveEdges} we immediately get the following.

\begin{corollary}\label{cor:extCycle}
	Let $(G_i, a_i, a'_i, b_i, b'_i)$, for $i=1,2$, be two tuples satisfying the conditions of Lemma~\ref{lem:preserveEdges} such that $G_1$ contains an extendable $5$-cycle $v_0\ldots v_4$. If $a_1, a'_1, b_1$ and $b'_1$ are disjoint from $v_0, \ldots, v_4$ and their neighbours, then the amalgam $G$ of $G_1$ and $G_2$ contains an extendable $5$-cycle.
\end{corollary}

\begin{figure}[!htb]
	\centering
	\begin{minipage}{36mm}
	\begin{tikzpicture}[scale=0.034]
		\definecolor{marked}{rgb}{0.25,0.5,0.25}
		\node [circle,fill,scale=0.25] (50) at (38.980281,41.636865) {};
		\node [circle,fill,scale=0.25] (49) at (45.744355,47.512005) {};
		\node [circle,fill,scale=0.25] (48) at (48.339635,41.125983) {};
		\node [circle,fill,scale=0.25] (47) at (40.860326,34.668438) {};
		\node [circle,fill,scale=0.25] (46) at (32.052724,34.014509) {};
		\node [circle,fill,scale=0.25] (45) at (23.725352,40.594666) {};
		\node [circle,fill,scale=0.25] (44) at (25.503221,48.022887) {};
		\node [circle,fill,scale=0.25] (43) at (32.165119,54.429344) {};
		\node [circle,fill,scale=0.25] (42) at (38.162870,59.047716) {};
		\node [circle,fill,scale=0.25] (41) at (51.496884,56.544395) {};
		\node [circle,fill,scale=0.25] (40) at (55.318279,48.891386) {};
		\node [circle,fill,scale=0.25] (39) at (59.323593,33.391234) {};
		\node [circle,fill,scale=0.25] (38) at (48.114847,27.618269) {};
		\node [circle,fill,scale=0.25] (37) at (42.106877,21.967917) {};
		\node [circle,fill,scale=0.25] (36) at (32.604477,26.187799) {};
		\node [circle,fill,scale=0.25] (35) at (17.094105,37.376111) {};
		\node [circle,fill,scale=0.25] (34) at (17.809341,56.432001) {};
		\node [circle,fill,scale=0.25] (33) at (32.880352,65.454172) {};
		\node [circle,fill,scale=0.25] (32) at (41.606213,68.744251) {};
		\node [circle,fill,scale=0.25] (31) at (55.308062,71.308878) {};
		\node [circle,fill,scale=0.25] (30) at (57.014407,61.918871) {};
		\node [circle,fill,scale=0.25] (29) at (65.770920,45.008685) {};
		\node [circle,fill,scale=0.25] (28) at (67.569224,34.852356) {};
		\node [circle,fill,scale=0.25] (27) at (57.821600,25.983447) {};
		\node [circle,fill,scale=0.25] (26) at (45.631961,15.081231) {};
		\node [circle,fill,scale=0.25] (25) at (24.604068,25.084296) {};
		\node [circle,fill,scale=0.25] (24) at (8.950650,48.268110) {};
		\node [circle,fill,scale=0.25] (23) at (21.477472,67.497700) {};
		\node [circle,fill,scale=0.25] (22) at (35.097579,77.327065) {};
		\node [circle,fill,scale=0.25] (21) at (43.516911,79.922344) {};
		\node [circle,fill,scale=0.25] (20) at (63.727394,72.514559) {};
		\node [circle,fill,scale=0.25] (19) at (72.933484,63.073464) {};
		\node [circle,fill,scale=0.25] (18) at (71.666496,52.559517) {};
		\node [circle,fill,scale=0.25] (17) at (75.947685,29.549402) {};
		\node [circle,fill,scale=0.25] (16) at (60.028610,17.288241) {};
		\node [circle,fill,scale=0.25] (15) at (30.172679,13.528150) {};
		\node [circle,fill,scale=0.25] (14) at (0.000000,49.994891) {};
		\node [circle,fill,scale=0.25] (13) at (23.306428,79.043628) {};
		\node [circle,fill,scale=0.25] (12) at (43.813222,87.677530) {};
		\node [circle,fill,scale=0.25] (11) at (60.447532,82.783282) {};
		\node [circle,fill,scale=0.25] (10) at (80.412792,65.321344) {};
		\node [circle,fill,scale=0.25] (9) at (81.485644,42.413405) {};
		\node [circle,fill,scale=0.25] (8) at (71.247574,15.377542) {};
		\node [circle,fill,scale=0.25] (7) at (25.002556,6.702770) {};
		\node [circle,fill,scale=0.25] (6) at (25.002556,93.297229) {};
		\node [circle,fill,scale=0.25] (5) at (73.188923,83.059160) {};
		\node [circle,fill,scale=0.25] (4) at (90.180852,49.586185) {};
		\node [circle,fill,scale=0.25] (3) at (74.997445,6.702770) {};
		\node [circle,fill,scale=0.25] (2) at (74.997445,93.297229) {};
		\node [circle,fill,scale=0.25] (1) at (99.999999,49.994891) {};
		\draw [black] (50) to (47);
		\draw [black] (50) to (44);
		\draw [black] (50) to (48);
		\draw [black] (49) to (48);
		\draw [black] (49) to (43);
		\draw [black] (49) to (40);
		\draw [black] (48) to (39);
		\draw [black] (47) to (38);
		\draw [black] (47) to (46);
		\draw [black] (46) to (45);
		\draw [black] (46) to (36);
		\draw [black] (45) to (35);
		\draw [black] (45) to (44);
		\draw [black] (44) to (34);
		\draw [black] (43) to (34);
		\draw [black] (43) to (42);
		\draw [black] (42) to (33);
		\draw [black] (42) to (41);
		\draw [black] (41) to (40);
		\draw [black] (41) to (30);
		\draw [black] (40) to (29);
		\draw [black] (39) to (27);
		\draw [black] (39) to (28);
		\draw [black] (38) to (27);
		\draw [black] (38) to (37);
		\draw [black] (37) to (26);
		\draw [black] (37) to (36);
		\draw [black] (36) to (25);
		\draw [black] (35) to (24);
		\draw [black] (35) to (25);
		\draw [black] (34) to (23);
		\draw [black] (34) to (24);
		\draw [black] (33) to (23);
		\draw [black] (33) to (32);
		\draw [black] (32) to (22);
		\draw [black] (32) to (30);
		\draw [black] (31) to (20);
		\draw [black] (31) to (30);
		\draw [black] (31) to (21);
		\draw [black] (30) to (18);
		\draw [black] (29) to (28);
		\draw [black] (29) to (18);
		\draw [black] (28) to (17);
		\draw [black] (27) to (16);
		\draw [black] (26) to (15);
		\draw [black] (26) to (16);
		\draw [black] (25) to (15);
		\draw [black] (24) to (14);
		\draw [black] (23) to (13);
		\draw [black] (22) to (21);
		\draw [black] (22) to (13);
		\draw [black] (21) to (12);
		\draw [black] (20) to (19);
		\draw [black] (20) to (11);
		\draw [black] (19) to (10);
		\draw [black] (19) to (18);
		\draw [black] (18) to (9);
		\draw [black] (17) to (8);
		\draw [black] (17) to (9);
		\draw [black] (16) to (8);
		\draw [black] (15) to (7);
		\draw [black] (14) to (6);
		\draw [black] (14) to (7);
		\draw [black] (13) to (6);
		\draw [black] (12) to (11);
		\draw [black] (12) to (6);
		\draw [black] (11) to (5);
		\draw [black] (10) to (4);
		\draw [black] (10) to (5);
		\draw [black] (9) to (4);
		\draw [black] (8) to (3);
		\draw [black] (7) to (3);
		\draw [black] (6) to (2);
		\draw [black] (5) to (2);
		\draw [black] (4) to (1);
		\draw [black] (3) to (1);
		\draw [black] (2) to (1);
		\begin{pgfonlayer}{bg}
			\path[fill=darkgray] (5.center) -- (10.center) -- (19.center) -- 
			(20.center) -- (11.center);
		\end{pgfonlayer}
		\path[draw=black, very thick] (3) circle[radius=2];
		\path[draw=black, very thick] (8) circle[radius=2];
		\path[draw=black, very thick] (14) circle[radius=2];
		\path[draw=black, very thick] (39) circle[radius=2];
		\path[draw=black, very thick] (48) circle[radius=2];
	\end{tikzpicture}
\end{minipage}\hfill
\begin{minipage}{36mm}
	\begin{tikzpicture}[scale=0.034]
    \definecolor{marked}{rgb}{0.25,0.5,0.25}
    \node [circle,fill,scale=0.25] (52) at (38.308632,49.071934) {};
    \node [circle,fill,scale=0.25] (51) at (46.838770,57.048704) {};
    \node [circle,fill,scale=0.25] (50) at (42.075296,39.010368) {};
    \node [circle,fill,scale=0.25] (49) at (32.717041,49.326587) {};
    \node [circle,fill,scale=0.25] (48) at (40.880959,59.224579) {};
    \node [circle,fill,scale=0.25] (47) at (52.006942,61.992198) {};
    \node [circle,fill,scale=0.25] (46) at (50.485076,48.276636) {};
    \node [circle,fill,scale=0.25] (45) at (48.616901,40.386796) {};
    \node [circle,fill,scale=0.25] (44) at (39.784644,29.537793) {};
    \node [circle,fill,scale=0.25] (43) at (29.371875,40.285709) {};
    \node [circle,fill,scale=0.25] (42) at (21.811524,44.018757) {};
    \node [circle,fill,scale=0.25] (41) at (21.516960,54.728020) {};
    \node [circle,fill,scale=0.25] (40) at (28.828115,58.544523) {};
    \node [circle,fill,scale=0.25] (39) at (38.476756,69.711106) {};
    \node [circle,fill,scale=0.25] (38) at (52.393446,70.120364) {};
    \node [circle,fill,scale=0.25] (37) at (59.657253,54.787101) {};
    \node [circle,fill,scale=0.25] (36) at (58.679839,47.998928) {};
    \node [circle,fill,scale=0.25] (35) at (55.620490,36.037370) {};
    \node [circle,fill,scale=0.25] (34) at (55.525663,29.137399) {};
    \node [circle,fill,scale=0.25] (33) at (46.029418,23.815029) {};
    \node [circle,fill,scale=0.25] (32) at (32.301909,29.945348) {};
    \node [circle,fill,scale=0.25] (31) at (17.367916,37.365777) {};
    \node [circle,fill,scale=0.25] (30) at (16.937763,61.448556) {};
    \node [circle,fill,scale=0.25] (29) at (31.333726,69.155890) {};
    \node [circle,fill,scale=0.25] (28) at (44.730302,75.818371) {};
    \node [circle,fill,scale=0.25] (27) at (61.572465,66.709118) {};
    \node [circle,fill,scale=0.25] (26) at (66.617622,54.595479) {};
    \node [circle,fill,scale=0.25] (25) at (64.281692,41.550647) {};
    \node [circle,fill,scale=0.25] (24) at (65.937031,29.681171) {};
    \node [circle,fill,scale=0.25] (23) at (46.525209,16.152365) {};
    \node [circle,fill,scale=0.25] (22) at (26.006089,25.138729) {};
    \node [circle,fill,scale=0.25] (21) at (8.725079,49.299199) {};
    \node [circle,fill,scale=0.25] (20) at (25.124735,73.949719) {};
    \node [circle,fill,scale=0.25] (19) at (45.516276,83.483839) {};
    \node [circle,fill,scale=0.25] (18) at (65.863244,70.345696) {};
    \node [circle,fill,scale=0.25] (17) at (75.207470,58.590860) {};
    \node [circle,fill,scale=0.25] (16) at (74.583801,42.748095) {};
    \node [circle,fill,scale=0.25] (15) at (61.864907,18.563000) {};
    \node [circle,fill,scale=0.25] (14) at (30.919694,13.834937) {};
    \node [circle,fill,scale=0.25] (13) at (0.000000,49.127333) {};
    \node [circle,fill,scale=0.25] (12) at (29.747768,85.459113) {};
    \node [circle,fill,scale=0.25] (11) at (60.518727,81.671322) {};
    \node [circle,fill,scale=0.25] (10) at (80.102070,65.193001) {};
    \node [circle,fill,scale=0.25] (9) at (80.264372,35.714097) {};
    \node [circle,fill,scale=0.25] (8) at (71.768394,16.250516) {};
    \node [circle,fill,scale=0.25] (7) at (25.748541,6.259267) {};
    \node [circle,fill,scale=0.25] (6) at (24.236782,92.867894) {};
    \node [circle,fill,scale=0.25] (5) at (70.596343,84.518490) {};
    \node [circle,fill,scale=0.25] (4) at (90.014078,50.620256) {};
    \node [circle,fill,scale=0.25] (3) at (75.753431,7.132105) {};
    \node [circle,fill,scale=0.25] (2) at (74.241672,93.740732) {};
    \node [circle,fill,scale=0.25] (1) at (99.999999,50.872839) {};
    \draw [black] (52) to (48);
    \draw [black] (52) to (50);
    \draw [black] (52) to (49);
    \draw [black] (51) to (47);
    \draw [black] (51) to (46);
    \draw [black] (51) to (48);
    \draw [black] (50) to (44);
    \draw [black] (50) to (45);
    \draw [black] (49) to (43);
    \draw [black] (49) to (40);
    \draw [black] (48) to (39);
    \draw [black] (47) to (38);
    \draw [black] (47) to (37);
    \draw [black] (46) to (36);
    \draw [black] (46) to (45);
    \draw [black] (45) to (35);
    \draw [black] (44) to (32);
    \draw [black] (44) to (33);
    \draw [black] (43) to (42);
    \draw [black] (43) to (32);
    \draw [black] (42) to (31);
    \draw [black] (42) to (41);
    \draw [black] (41) to (30);
    \draw [black] (41) to (40);
    \draw [black] (40) to (29);
    \draw [black] (39) to (28);
    \draw [black] (39) to (29);
    \draw [black] (38) to (27);
    \draw [black] (38) to (28);
    \draw [black] (37) to (36);
    \draw [black] (37) to (26);
    \draw [black] (36) to (25);
    \draw [black] (35) to (34);
    \draw [black] (35) to (25);
    \draw [black] (34) to (24);
    \draw [black] (34) to (33);
    \draw [black] (33) to (23);
    \draw [black] (32) to (22);
    \draw [black] (31) to (21);
    \draw [black] (31) to (22);
    \draw [black] (30) to (20);
    \draw [black] (30) to (21);
    \draw [black] (29) to (20);
    \draw [black] (28) to (19);
    \draw [black] (27) to (18);
    \draw [black] (27) to (26);
    \draw [black] (26) to (16);
    \draw [black] (25) to (16);
    \draw [black] (24) to (15);
    \draw [black] (24) to (16);
    \draw [black] (23) to (14);
    \draw [black] (23) to (15);
    \draw [black] (22) to (14);
    \draw [black] (21) to (13);
    \draw [black] (20) to (12);
    \draw [black] (19) to (11);
    \draw [black] (19) to (12);
    \draw [black] (18) to (17);
    \draw [black] (18) to (11);
    \draw [black] (17) to (10);
    \draw [black] (17) to (16);
    \draw [black] (16) to (9);
    \draw [black] (15) to (8);
    \draw [black] (14) to (7);
    \draw [black] (13) to (6);
    \draw [black] (13) to (7);
    \draw [black] (12) to (6);
    \draw [black] (11) to (5);
    \draw [black] (10) to (4);
    \draw [black] (10) to (5);
    \draw [black] (9) to (8);
    \draw [black] (9) to (4);
    \draw [black] (8) to (3);
    \draw [black] (7) to (3);
    \draw [black] (6) to (2);
    \draw [black] (5) to (2);
    \draw [black] (4) to (1);
    \draw [black] (3) to (1);
    \draw [black] (2) to (1);
\begin{pgfonlayer}{bg}
\path[fill=darkgray] (5.center) -- (10.center) -- (17.center) -- (18.center) -- 
(11.center);
\end{pgfonlayer}
\end{tikzpicture}
\end{minipage}\hfill
\begin{minipage}{36mm}
	\begin{tikzpicture}[scale=0.034]
		\definecolor{marked}{rgb}{0.25,0.5,0.25}
		\node [circle,fill,scale=0.25] (53) at (44.493319,93.184638) {};
		\node [circle,fill,scale=0.25] (52) at (39.384647,86.558118) {};
		\node [circle,fill,scale=0.25] (51) at (31.094948,83.097855) {};
		\node [circle,fill,scale=0.25] (50) at (22.755693,84.109674) {};
		\node [circle,fill,scale=0.25] (49) at (28.930459,99.636726) {};
		\node [circle,fill,scale=0.25] (48) at (70.130444,88.083616) {};
		\node [circle,fill,scale=0.25] (47) at (55.115556,85.746962) {};
		\node [circle,fill,scale=0.25] (46) at (44.855992,77.858383) {};
		\node [circle,fill,scale=0.25] (45) at (41.422030,70.070603) {};
		\node [circle,fill,scale=0.25] (44) at (33.458860,73.090793) {};
		\node [circle,fill,scale=0.25] (43) at (20.636641,71.322144) {};
		\node [circle,fill,scale=0.25] (42) at (8.495548,62.049863) {};
		\node [circle,fill,scale=0.25] (41) at (9.644218,31.530242) {};
		\node [circle,fill,scale=0.25] (40) at (0.002854,70.204265) {};
		\node [circle,fill,scale=0.25] (39) at (70.197203,99.999999) {};
		\node [circle,fill,scale=0.25] (38) at (91.164728,66.426166) {};
		\node [circle,fill,scale=0.25] (37) at (78.965010,74.106907) {};
		\node [circle,fill,scale=0.25] (36) at (63.675964,81.622988) {};
		\node [circle,fill,scale=0.25] (35) at (52.706387,75.351964) {};
		\node [circle,fill,scale=0.25] (34) at (45.539094,60.316275) {};
		\node [circle,fill,scale=0.25] (33) at (29.753491,65.749447) {};
		\node [circle,fill,scale=0.25] (32) at (17.613968,62.234641) {};
		\node [circle,fill,scale=0.25] (31) at (12.441720,45.834343) {};
		\node [circle,fill,scale=0.25] (30) at (23.706966,21.527889) {};
		\node [circle,fill,scale=0.25] (29) at (0.360419,28.923397) {};
		\node [circle,fill,scale=0.25] (28) at (99.643788,71.066687) {};
		\node [circle,fill,scale=0.25] (27) at (88.645726,49.233167) {};
		\node [circle,fill,scale=0.25] (26) at (80.377390,48.672076) {};
		\node [circle,fill,scale=0.25] (25) at (75.728544,61.277995) {};
		\node [circle,fill,scale=0.25] (24) at (64.013198,64.261690) {};
		\node [circle,fill,scale=0.25] (23) at (53.131293,62.883726) {};
		\node [circle,fill,scale=0.25] (22) at (38.392990,56.697892) {};
		\node [circle,fill,scale=0.25] (21) at (29.713027,49.994653) {};
		\node [circle,fill,scale=0.25] (20) at (23.804395,39.436415) {};
		\node [circle,fill,scale=0.25] (19) at (29.730423,27.267321) {};
		\node [circle,fill,scale=0.25] (18) at (34.438473,8.485093) {};
		\node [circle,fill,scale=0.25] (17) at (29.802796,0.000000) {};
		\node [circle,fill,scale=0.25] (16) at (99.997146,29.795734) {};
		\node [circle,fill,scale=0.25] (15) at (90.754367,32.495078) {};
		\node [circle,fill,scale=0.25] (14) at (74.762514,35.547339) {};
		\node [circle,fill,scale=0.25] (13) at (68.261065,55.522353) {};
		\node [circle,fill,scale=0.25] (12) at (57.403538,55.465230) {};
		\node [circle,fill,scale=0.25] (11) at (40.713583,48.474372) {};
		\node [circle,fill,scale=0.25] (10) at (33.129008,40.761791) {};
		\node [circle,fill,scale=0.25] (9) at (43.156725,22.154880) {};
		\node [circle,fill,scale=0.25] (8) at (49.632730,12.383889) {};
		\node [circle,fill,scale=0.25] (7) at (71.069539,0.363275) {};
		\node [circle,fill,scale=0.25] (6) at (77.335217,24.166300) {};
		\node [circle,fill,scale=0.25] (5) at (66.502555,38.219234) {};
		\node [circle,fill,scale=0.25] (4) at (53.346958,41.866971) {};
		\node [circle,fill,scale=0.25] (3) at (46.866510,29.954433) {};
		\node [circle,fill,scale=0.25] (2) at (66.408372,11.454558) {};
		\node [circle,fill,scale=0.25] (1) at (58.762003,29.165055) {};
		\draw [black] (53) to (47);
		\draw [black] (53) to (52);
		\draw [black] (53) to (49);
		\draw [black] (52) to (46);
		\draw [black] (52) to (51);
		\draw [black] (51) to (50);
		\draw [black] (51) to (44);
		\draw [black] (50) to (43);
		\draw [black] (50) to (49);
		\draw [black] (49) to (39);
		\draw [black] (49) to (40);
		\draw [black] (48) to (36);
		\draw [black] (48) to (39);
		\draw [black] (48) to (37);
		\draw [black] (47) to (35);
		\draw [black] (47) to (36);
		\draw [black] (46) to (35);
		\draw [black] (46) to (45);
		\draw [black] (45) to (34);
		\draw [black] (45) to (44);
		\draw [black] (44) to (33);
		\draw [black] (43) to (32);
		\draw [black] (43) to (33);
		\draw [black] (42) to (31);
		\draw [black] (42) to (40);
		\draw [black] (42) to (32);
		\draw [black] (41) to (29);
		\draw [black] (41) to (31);
		\draw [black] (41) to (30);
		\draw [black] (40) to (29);
		\draw [black] (39) to (28);
		\draw [black] (38) to (27);
		\draw [black] (38) to (37);
		\draw [black] (38) to (28);
		\draw [black] (37) to (25);
		\draw [black] (36) to (24);
		\draw [black] (35) to (23);
		\draw [black] (34) to (22);
		\draw [black] (34) to (23);
		\draw [black] (33) to (22);
		\draw [black] (32) to (21);
		\draw [black] (31) to (20);
		\draw [black] (30) to (18);
		\draw [black] (30) to (19);
		\draw [black] (29) to (17);
		\draw [black] (28) to (16);
		\draw [black] (27) to (26);
		\draw [black] (27) to (15);
		\draw [black] (26) to (14);
		\draw [black] (26) to (25);
		\draw [black] (25) to (13);
		\draw [black] (24) to (12);
		\draw [black] (24) to (13);
		\draw [black] (23) to (12);
		\draw [black] (22) to (11);
		\draw [black] (21) to (10);
		\draw [black] (21) to (11);
		\draw [black] (20) to (10);
		\draw [black] (20) to (19);
		\draw [black] (19) to (9);
		\draw [black] (18) to (17);
		\draw [black] (18) to (8);
		\draw [black] (17) to (7);
		\draw [black] (16) to (15);
		\draw [black] (16) to (7);
		\draw [black] (15) to (6);
		\draw [black] (14) to (5);
		\draw [black] (14) to (6);
		\draw [black] (13) to (5);
		\draw [black] (12) to (4);
		\draw [black] (11) to (4);
		\draw [black] (10) to (3);
		\draw [black] (9) to (8);
		\draw [black] (9) to (3);
		\draw [black] (8) to (2);
		\draw [black] (7) to (2);
		\draw [black] (6) to (2);
		\draw [black] (5) to (1);
		\draw [black] (4) to (1);
		\draw [black] (3) to (1);
		\draw [black] (2) to (1);
		\begin{pgfonlayer}{bg}
			\path[fill=darkgray] (6.center) -- (14.center) -- (26.center) -- 
			(27.center) -- (15.center);
		\end{pgfonlayer}
		\path[draw=black, very thick] (4) circle[radius=2];
		\path[draw=black, very thick] (21) circle[radius=2];
		\path[draw=black, very thick] (24) circle[radius=2];
		\path[draw=black, very thick] (32) circle[radius=2];
		\path[draw=black, very thick] (36) circle[radius=2];
	\end{tikzpicture}
\end{minipage}
	\caption{Planar $K_2$-hypohamiltonian graphs on $50$, $52$, and $53$
		vertices. A face for which the boundary is an extendable $5$-cycle is filled in. All vertices for which the vertex-deleted subgraphs are non-hamiltonian are circled. Proofs of these facts are given in~\cite{GRWZ22}.}
	\label{fig:planar_graphs} 
\end{figure}
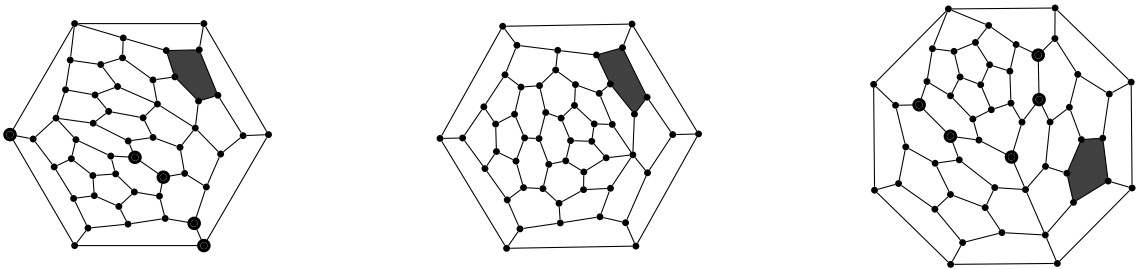
 \begin{theorem}\label{thm:134Planar}
 	There exists a planar $K_2$-hypohamiltonian graph of order $n$ if $n\geq 134$.
 \end{theorem}
 \begin{proof}
 	In~\cite{GRWZ22} it was shown that this holds for $n\geq 177$. Starting from the five planar $K_2$-hypohamiltonian graphs of Figure~$6$ in that paper and gluing them together in the appropriate way using Lemma~$4$ from that paper yields planar $K_2$-hypohamiltonian graphs containing an extendable $5$-cycle on orders 134--149.
 	 We will show the existence of planar $K_2$-hypohamiltonian graphs containing an extendable $5$-cycle on orders 150--153.
 	  Denote by $G_{50}$, $G_{52}$ and $G_{53}$, respectively, the graphs of Figure~\ref{fig:planar_graphs}. In Appendix~\ref{app:planarThm}, we show that these three planar $K_2$-hypohamiltonian graphs have a tuple $(G_i, a_i, a'_i, b_i, b'_i)$ for $i=50,52,53$, that satisfies the gluing property and in which $G_i - a'_i$ has at least two hamiltonian cycles, one containing the edge $b_ib'_i$ and one which does not and in which $G_i - v$ is hamiltonian for all $v\in N_{G_i}[b_i]$. Moreover, $(G_{52}, c_{52}, c'_{52}, d_{52}, d'_{52})$ is a tuple satisfying the same properties, such that $a_{52}, a'_{52}, b_{52}, b'_{52}$ and $c_{52}, c'_{52}, d_{52}, d'_{52}$ are pairwise disjoint. We see that there is an embedding such that $a_i, a'_i, b_i, b'_i$ are co-facial for $i=50,52$ and $c_{52}, c'_{52}, d_{52}, d'_{52}$ and $a_{53}, a'_{53}, b_{53}, b'_{53}$ are not co-facial. 
 	
 	The amalgam $G_{100}$ of $(G_{50}, a_{50}, a'_{50}, b_{50}, b'_{50})$ and $(G_{52},c_{52}, c'_{52}, d_{52}, d'_{52})$ is by Lemma~\ref{lem:nonHamiltonian} and Lemma~\ref{lem:K2Hamiltonian} a $K_2$-hypohamiltonian graph. By Lemma~\ref{lem:planar} it is planar and by Corollary~\ref{cor:extCycle} it contains an extendable $5$-cycle. Moreover, by Lemma~\ref{lem:preserveEdges} $(G_{100}, a_{52}, a'_{52}, b_{52}, b'_{52})$ satisfies all the conditions of Lemma~\ref{lem:K2Hamiltonian}. It is easy to see that this extendable $5$-cycle and $a_{52}, a'_{52}, b_{52}, b'_{52}$ are disjoint and that the latter vertices are co-facial. Hence, the amalgam of $(G_{100}, a_{52}, a'_{52}, b_{52}, b'_{52})$ and  $(G_{52},c_{52}, c'_{52}, d_{52}, d'_{52})$ is then a planar $K_2$-hypohamiltonian graph of order $150$ containing an extendable $5$-cycle by the same reasons as before.
 	
Similarly, we get a planar $K_2$-hypohamiltonian graph of order $151$ with an extendable $5$-cycle by taking the amalgam of $(G_{100}, a_{52}, a'_{52}, b_{52}, b'_{52})$ and $(G_{53},a_{53}, a'_{53}, b_{53}, b'_{53})$.
 
 We get planar $K_2$-hypohamiltonian graphs of order $152$ and $153$ containing an extendable $5$-cycle by taking the amalgam of the appropriate tuples of $G_{52}$ with itself and then taking the amalgam with either $G_{52}$ or $G_{53}$. Applying, ad infinitum, the dodecahedron operation described in~\cite{Za21} to these twenty planar $K_2$-hypohamiltonian graphs, each with an extendable $5$-cycle and having consecutive orders, yields the result. 
 \end{proof}
 
 \subsection{\texorpdfstring{\boldmath{$K_2$}}{K2}-hypohamiltonian graphs with large maximum degree}\label{sec:infinite_family}
 
Thomassen showed in~\cite{Th81} that the maximum degree and maximum number of edges in an $n$-vertex hypohamiltonian graph can be $n/2 - 9$ and $(n-20)^2/4+32$. A similar result is not known for $K_2$-hypohamiltonian graphs, since no non-trivial upper bound on the maximum degree and size of such a graph is known. The third author has shown in~\cite{Za21} that for any integer $d\geq 3$ there exists a $K_2$-hypohamiltonian graph with maximum degree $d$. But in that construction the order of the graph becomes very large  as $d$ increases. In this section, we show the existence of an infinite family consisting of $n$-vertex $K_2$-hypohamiltonian graphs with maximum degree $(n-1)/3$ and maximum size $2n-5$. This family contains both the smallest (the Petersen graph) and second smallest \mbox{($K_2$-)hypohamiltonian} graphs and thus can be seen as yet another generalisation of Petersen's famous graph.

   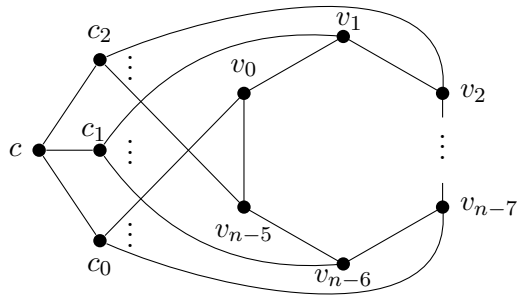
\begin{figure}[!htb]
 \centering
 	\begin{tikzpicture}[vertex/.style = {circle,fill=black,minimum size=5pt,inner sep=0pt}, scale=0.8]
 		\node[vertex, label=left:$c$] (c) at (0,0) {};
 		\node[vertex, label=below:$c_0$] (c0) at (1,-1.5) {};
 		\path (c0) ++(0.5,0.23) node {$\vdots$};
 		\node[vertex] (c1) at (1,0) {};
 		\path(c1) ++(-0.1,0.3) node {$c_1$};
 		\path (c1) ++(0.5,0.1) node {$\vdots$};
 		\node[vertex, label=$c_2$] (c2) at (1,1.5) {};
 		\path (c2) ++(0.5,0) node {$\vdots$};
 		\draw (c) -- (c0);
 		\draw (c) -- (c1);
 		\draw (c) -- (c2);
 		
 		\node[rotate=30,draw=none,minimum size=3cm,regular polygon,regular polygon sides=6] (a) at (5, 0) {};

 		\node[vertex] at (a.corner 1) {};
 		\path (a.corner 1) ++(0.1,0.25) node {$v_1$};
 		\node[vertex, label=above:$v_0$] at (a.corner 2) {};
 		\node[vertex, label=below:$v_{n-5}$] at (a.corner 3) {};
 		\node[vertex] at (a.corner 4) {};
 		\path (a.corner 4) ++(0,-0.25) node {$v_{n-6}$};
 		\node[vertex, label=right:$v_{n-7}$] at (a.corner 5) {};
 		\node[vertex, label=right:$v_2$] at (a.corner 6) {};
 		\path (a.corner 6) ++(0,-0.75) node {$\vdots$};
 		\draw ($ (a.corner 6) + (0,-0.4) $) -- (a.corner 6) -- (a.corner 1) -- (a.corner 2) -- (a.corner 3) -- (a.corner 4) -- (a.corner 5) -- ($ (a.corner 5) + (0,0.4) $);
 		
 		\draw (c0) to (a.corner 2);
 		\draw (c1) to[bend left] (a.corner 1);
 		\draw (c2) to[bend left, in=90] (a.corner 6);
 		\draw (c0) to[bend right, in=-90] (a.corner 5);
 		\draw (c1) to[bend right] (a.corner 4);
 		\draw (c2) to (a.corner 3);
 	\end{tikzpicture}
 	\vspace*{-2em}
 	\caption{Visualisation of a graph with order $n$ from Theorem~\ref{thm:largeMaxDeg}. A vertex $v_i$ is adjacent to $c_j$ if $i\equiv j\bmod\: 3$.}
 	\label{fig:infiniteFamily}
 \end{figure}
 
 \begin{theorem}\label{thm:largeMaxDeg}
 	For every integer $n \geq 10$ which is congruent to $1\bmod\:3$, there exists a hypohamiltonian $K_2$-hamiltonian graph of order $n$, size $2n - 5$ and maximum degree $(n-1)/3$.
 \end{theorem}
 \begin{proof}
 	Consider the vertex set 
	$$V_n := \{c, c_0, c_1, c_2\}\cup \{v_i\}_ {i=0}^{n-5}$$ 
	and the edge set
	$$E_n := \{c_0c, c_1c, c_2c, v_0v_{n-5}\}\cup\{v_iv_{i+1}\}_{i=0}^{n-6}\cup\{v_ic_{i\bmod 3}\}_{i=0}^{n-5}.$$
	We define the family $$\mathcal{F} := \{(V_n,E_n)\mid n=3m+10, m\geq 0\}.$$
	Every graph $G$ in $\mathcal{F}$ is non-hamiltonian. Indeed, suppose that $G$ contains a hamiltonian cycle $\mathfrak{h}$ and let the vertices of $G$ be labelled as above. In the remainder of this proof, we will assume all indices are taken modulo $n-4$. Without loss of generality assume that $\mathfrak{h}$ contains $c_1cc_2$ and $v_ic_0v_j$. Then $i\equiv j\equiv 0 \bmod 3$ and $\mathfrak{h}$ must contain either $v_iv_{i+1}\ldots v_{j-1}$ or $v_{i}v_{i-1}\ldots v_{j+1}$. Indeed, suppose $\mathfrak{h}$ contains $v_iv_{i+1}\ldots v_{j-k}$, but not $v_{j-k}v_{j-k+1}$, for $k > 1$, then $\mathfrak{h}$ can only contain either the vertices $v_{j-k+1}, \ldots, v_{j-1}$ or the vertices $v_{j+1},\ldots, v_{i-1}$. In both cases $\mathfrak{h}$ is not hamiltonian. Similarly, we get a contradiction if $\mathfrak{h}$ contains $v_{i}v_{i-1}\ldots v_{j+k}$, but not $v_{j+k}v_{j+k-1}$. Now suppose $\mathfrak{h}$ does contain $v_iv_{i+1}\ldots v_{j-1}$, then $\mathfrak{h}$ must also contain $v_jv_{j+1}\ldots v_{i-1}$. However, $i-1
		\equiv j-1\equiv 2 \bmod 3$, which means there is no pair of edges $v_{i-1}c_1$, $v_{j-1}c_2$ or $v_{i-1}c_2$, $v_{j-1}c_1$ for which both edges are present in $G$, a contradiction. In the case where $\mathfrak{h}$ contains  $v_{i}v_{i-1}\ldots v_{j+1}$, we have the same situation.
		
		We now show that any vertex-deleted subgraph contains a hamiltonian cycle. Denote the order of $G$ by $n$. Suppose $v = c$, then $G - v$ has hamiltonian cycle $$v_0c_0v_3v_4c_1v_1v_2c_2v_5\ldots v_{n-5}.$$
		Suppose $v = c_2$, then $G - v$ has the hamiltonian cycle 
		$$v_0c_0cc_1v_1\ldots v_{n-5}.$$
		For the vertices $c_0$ and $c_1$, the arguments are analogous. Suppose $v = v_i$ and assume without loss of generality that $i\equiv 0\bmod 3$, then $G - v$ has the hamiltonian cycle 
		$$v_{i-1}c_2v_{i+2}v_{i+1}c_1cc_0v_{i+3}v_{i+4}\ldots v_{i-2}.$$
	 The remaining cases are analogous, hence, $G - v$ is hamiltonian for all $v\in V(G)$.
		
		Similarly, we show $G$ is $K_2$-hamiltonian.
		Suppose $v = c, w = c_2$. Then $G - v - w$ has the hamiltonian cycle 
		$$v_0v_1c_1v_4v_3v_2c_2v_5\ldots v_{n-5}.$$
		Suppose $v = c_1$, $w = v_i$, where $i\equiv 1\bmod 3$. Then $G - v - w$ has the hamiltonian cycle 
		$$v_{i-1}c_0cc_2v_{i+1}v_{i+2}\ldots v_{i-2}.$$ 
		Suppose $v = v_i$, $w = v_{i+1}$, such that $i\equiv 0 \bmod 3$. Then $G - v  - w$ has a hamiltonian cycle 
		$$v_{i-1}c_2v_{i+2}v_{i+3}c_0cc_1v_{i+4}v_{i+5}\ldots v_{i-2}.$$ All other cases are analogous to one of these, hence $G$ is $K_2$-hamiltonian. Hence, $G$ is a hypohamiltonian $K_2$-hamiltonian graph.
		
		Now let $n\geq 10$ and $n\equiv 1\bmod 3$, i.e.\ $n = 3m+10$ for some integer $m\geq 0$. Then $(V_n, E_n)$ has maximum degree $m + 3$. 
 \end{proof}

Experiments indicate that for the same orders there also exists an infinite family of hypohamiltonian $K_2$-hypohamiltonian graphs in which a member of order $n$ has size $2n-4$, thus slightly beating the bound above. However, the proofs of e.g.\ non-hamiltonicity are far more tedious, hence we have decided to omit this.

\section*{Acknowledgements}
Several of the computations for this work were carried out using the supercomputer infrastructure provided by the VSC (Flemish Supercomputer Center), funded by the Research
Foundation - Flanders (FWO) and the Flemish Government. The research of Jan Goedgebeur and Jarne Renders was supported by Internal Funds of KU Leuven. The research of Carol T. Zamfirescu was supported by a Postdoctoral Fellowship of the Research Foundation - Flanders (FWO).

\clearpage

\appendix
\section{Appendix}
\subsection{Certificates for the proof of Theorem~\ref{thm:134Planar}}\label{app:planarThm}
For the proof of Theorem~\ref{thm:134Planar}, we require four tuples satisfying the gluing property. Note that for readability, we removed all subscripts from the following figures. 

In Fig.~\ref{fig:g50}, we see that $(G_{50}, a_{50}, a'_{50}, b_{50}, b'_{50})$ satisfies the gluing property, such that $G_{50} - a'_{50}$ has at least two hamiltonian cycles, one containing $b_{50}b'_{50}$ and one not containing it. Moreover, we see that in this embedding $a_{50}, a'_{50}, b_{50}, b'_{50}$ are co-facial. Also note that $G_{50} - v$ is hamiltonian for all $v\in \{a_{50}, a'_{50}, b_{50}, b'_{50}\}$ and that $v$ is not among the vertices of the extendable $5$-cycle and their neighbours.
\begin{figure}[!htb]
	\centering
\begin{minipage}{60mm}
\centering
	\begin{tikzpicture}[scale=0.050]
		\definecolor{marked}{rgb}{0.25,0.5,0.25}
		\node [circle,fill,scale=0.25] (50) at (38.980281,41.636865) {};
		\node [circle,fill,scale=0.25] (49) at (45.744355,47.512005) {};
		\node [circle,fill,scale=0.25] (48) at (48.339635,41.125983) {};
		\node [circle,fill,scale=0.25] (47) at (40.860326,34.668438) {};
		\node [circle,fill,scale=0.25] (46) at (32.052724,34.014509) {};
		\node [circle,fill,scale=0.25] (45) at (23.725352,40.594666) {};
		\node [circle,fill,scale=0.25] (44) at (25.503221,48.022887) {};
		\node [circle,fill,scale=0.25] (43) at (32.165119,54.429344) {};
		\node [circle,fill,scale=0.25] (42) at (38.162870,59.047716) {};
		\node [circle,fill,scale=0.25] (41) at (51.496884,56.544395) {};
		\node [circle,fill,scale=0.25] (40) at (55.318279,48.891386) {};
		\node [circle,fill,scale=0.25] (39) at (59.323593,33.391234) {};
		\node [circle,fill,scale=0.25] (38) at (48.114847,27.618269) {};
		\node [circle,fill,scale=0.25] (37) at (42.106877,21.967917) {};
		\node [circle,fill,scale=0.25] (36) at (32.604477,26.187799) {};
		\node [circle,fill,scale=0.25] (35) at (17.094105,37.376111) {};
		\path (35) ++(0.5,4) node {$b$};
		\node [circle,fill,scale=0.25] (34) at (17.809341,56.432001) {};
		\node [circle,fill,scale=0.25] (33) at (32.880352,65.454172) {};
		\node [circle,fill,scale=0.25] (32) at (41.606213,68.744251) {};
		\node [circle,fill,scale=0.25] (31) at (55.308062,71.308878) {};
		\node [circle,fill,scale=0.25] (30) at (57.014407,61.918871) {};
		\node [circle,fill,scale=0.25] (29) at (65.770920,45.008685) {};
		\node [circle,fill,scale=0.25] (28) at (67.569224,34.852356) {};
		\node [circle,fill,scale=0.25] (27) at (57.821600,25.983447) {};
		\node [circle,fill,scale=0.25] (26) at (45.631961,15.081231) {};
		\node [circle,fill,scale=0.25] (25) at (24.604068,25.084296) {};
		\path (25) ++(1,4) node {$a$};
		\node [circle,fill,scale=0.25] (24) at (8.950650,48.268110) {};
		\path (24) ++(0,4) node {$b'$};
		\node [circle,fill,scale=0.25] (23) at (21.477472,67.497700) {};
		\node [circle,fill,scale=0.25] (22) at (35.097579,77.327065) {};
		\node [circle,fill,scale=0.25] (21) at (43.516911,79.922344) {};
		\node [circle,fill,scale=0.25] (20) at (63.727394,72.514559) {};
		\node [circle,fill,scale=0.25] (19) at (72.933484,63.073464) {};
		\node [circle,fill,scale=0.25] (18) at (71.666496,52.559517) {};
		\node [circle,fill,scale=0.25] (17) at (75.947685,29.549402) {};
		\node [circle,fill,scale=0.25] (16) at (60.028610,17.288241) {};
		\node [circle,fill,scale=0.25] (15) at (30.172679,13.528150) {};
		\path (15) ++(2.5,4) node {$a'$};
		\node [circle,fill,scale=0.25] (14) at (0.000000,49.994891) {};
		\node [circle,fill,scale=0.25] (13) at (23.306428,79.043628) {};
		\node [circle,fill,scale=0.25] (12) at (43.813222,87.677530) {};
		\node [circle,fill,scale=0.25] (11) at (60.447532,82.783282) {};
		\node [circle,fill,scale=0.25] (10) at (80.412792,65.321344) {};
		\node [circle,fill,scale=0.25] (9) at (81.485644,42.413405) {};
		\node [circle,fill,scale=0.25] (8) at (71.247574,15.377542) {};
		\node [circle,fill,scale=0.25] (7) at (25.002556,6.702770) {};
		\node [circle,fill,scale=0.25] (6) at (25.002556,93.297229) {};
		\node [circle,fill,scale=0.25] (5) at (73.188923,83.059160) {};
		\node [circle,fill,scale=0.25] (4) at (90.180852,49.586185) {};
		\node [circle,fill,scale=0.25] (3) at (74.997445,6.702770) {};
		\node [circle,fill,scale=0.25] (2) at (74.997445,93.297229) {};
		\node [circle,fill,scale=0.25] (1) at (99.999999,49.994891) {};
		\draw [black] (50) to (47);
		\draw [black] (50) to (44);
		\draw [black] (50) to (48);
		\draw [black] (49) to (48);
		\draw [black] (49) to (43);
		\draw [black] (49) to (40);
		\draw [black] (48) to (39);
		\draw [black] (47) to (38);
		\draw [black] (47) to (46);
		\draw [black] (46) to (45);
		\draw [black] (46) to (36);
		\draw [black] (45) to (35);
		\draw [black] (45) to (44);
		\draw [black] (44) to (34);
		\draw [black] (43) to (34);
		\draw [black] (43) to (42);
		\draw [black] (42) to (33);
		\draw [black] (42) to (41);
		\draw [black] (41) to (40);
		\draw [black] (41) to (30);
		\draw [black] (40) to (29);
		\draw [black] (39) to (27);
		\draw [black] (39) to (28);
		\draw [black] (38) to (27);
		\draw [black] (38) to (37);
		\draw [black] (37) to (26);
		\draw [black] (37) to (36);
		\draw [black] (36) to (25);
		\draw [black] (35) to (24);
		\draw [black] (35) to (25);
		\draw [black] (34) to (23);
		\draw [black] (34) to (24);
		\draw [black] (33) to (23);
		\draw [black] (33) to (32);
		\draw [black] (32) to (22);
		\draw [black] (32) to (30);
		\draw [black] (31) to (20);
		\draw [black] (31) to (30);
		\draw [black] (31) to (21);
		\draw [black] (30) to (18);
		\draw [black] (29) to (28);
		\draw [black] (29) to (18);
		\draw [black] (28) to (17);
		\draw [black] (27) to (16);
		\draw [black] (26) to (15);
		\draw [black] (26) to (16);
		\draw [black] (25) to (15);
		\draw [black] (24) to (14);
		\draw [black] (23) to (13);
		\draw [black] (22) to (21);
		\draw [black] (22) to (13);
		\draw [black] (21) to (12);
		\draw [black] (20) to (19);
		\draw [black] (20) to (11);
		\draw [black] (19) to (10);
		\draw [black] (19) to (18);
		\draw [black] (18) to (9);
		\draw [black] (17) to (8);
		\draw [black] (17) to (9);
		\draw [black] (16) to (8);
		\draw [black] (15) to (7);
		\draw [black] (14) to (6);
		\draw [black] (14) to (7);
		\draw [black] (13) to (6);
		\draw [black] (12) to (11);
		\draw [black] (12) to (6);
		\draw [black] (11) to (5);
		\draw [black] (10) to (4);
		\draw [black] (10) to (5);
		\draw [black] (9) to (4);
		\draw [black] (8) to (3);
		\draw [black] (7) to (3);
		\draw [black] (6) to (2);
		\draw [black] (5) to (2);
		\draw [black] (4) to (1);
		\draw [black] (3) to (1);
		\draw [black] (2) to (1);
		\draw [red, ultra thick] (7) -- (3) -- (1) -- (2) -- (5) -- (10) -- (4) -- (9) -- (18) -- (19) -- (20) -- (11) -- (12) -- (6) -- (13) -- (23) -- (33) -- (32) -- (22) -- (21) -- (31) -- (30) -- (41) -- (42) -- (43) -- (49) -- (40) -- (29) -- (28) -- (17) -- (8) -- (16) -- (26) -- (37) -- (36) -- (25) -- (35) -- (45) -- (46) -- (47) -- (38) -- (27) -- (39) -- (48) -- (50) -- (44) -- (34) -- (24) -- (14) -- (7);

		\begin{pgfonlayer}{bg}
			\path[fill=darkgray] (5.center) -- (10.center) -- (19.center) -- 
			(20.center) -- (11.center);
		\end{pgfonlayer}
		\path[draw=black, very thick] (3) circle[radius=2];
		\path[draw=black, very thick] (8) circle[radius=2];
		\path[draw=black, very thick] (14) circle[radius=2];
		\path[draw=black, very thick] (39) circle[radius=2];
		\path[draw=black, very thick] (48) circle[radius=2];
	\end{tikzpicture}
\end{minipage}
\begin{minipage}{60mm}
\centering
	\begin{tikzpicture}[scale=0.050]
				\node [circle,fill,scale=0.25] (50) at (38.980281,41.636865) {};
		\node [circle,fill,scale=0.25] (49) at (45.744355,47.512005) {};
		\node [circle,fill,scale=0.25] (48) at (48.339635,41.125983) {};
		\node [circle,fill,scale=0.25] (47) at (40.860326,34.668438) {};
		\node [circle,fill,scale=0.25] (46) at (32.052724,34.014509) {};
		\node [circle,fill,scale=0.25] (45) at (23.725352,40.594666) {};
		\node [circle,fill,scale=0.25] (44) at (25.503221,48.022887) {};
		\node [circle,fill,scale=0.25] (43) at (32.165119,54.429344) {};
		\node [circle,fill,scale=0.25] (42) at (38.162870,59.047716) {};
		\node [circle,fill,scale=0.25] (41) at (51.496884,56.544395) {};
		\node [circle,fill,scale=0.25] (40) at (55.318279,48.891386) {};
		\node [circle,fill,scale=0.25] (39) at (59.323593,33.391234) {};
		\node [circle,fill,scale=0.25] (38) at (48.114847,27.618269) {};
		\node [circle,fill,scale=0.25] (37) at (42.106877,21.967917) {};
		\node [circle,fill,scale=0.25] (36) at (32.604477,26.187799) {};
		\node [circle,fill,scale=0.25] (35) at (17.094105,37.376111) {};
		\path (35) ++(0.5,4) node {$b$};
		\node [circle,fill,scale=0.25] (34) at (17.809341,56.432001) {};
		\node [circle,fill,scale=0.25] (33) at (32.880352,65.454172) {};
		\node [circle,fill,scale=0.25] (32) at (41.606213,68.744251) {};
		\node [circle,fill,scale=0.25] (31) at (55.308062,71.308878) {};
		\node [circle,fill,scale=0.25] (30) at (57.014407,61.918871) {};
		\node [circle,fill,scale=0.25] (29) at (65.770920,45.008685) {};
		\node [circle,fill,scale=0.25] (28) at (67.569224,34.852356) {};
		\node [circle,fill,scale=0.25] (27) at (57.821600,25.983447) {};
		\node [circle,fill,scale=0.25] (26) at (45.631961,15.081231) {};
		\node [circle,fill,scale=0.25] (25) at (24.604068,25.084296) {};
		\path (25) ++(1,4) node {$a$};
		\node [circle,fill,scale=0.25] (24) at (8.950650,48.268110) {};
		\path (24) ++(0,4) node {$b'$};
		\node [circle,fill,scale=0.25] (23) at (21.477472,67.497700) {};
		\node [circle,fill,scale=0.25] (22) at (35.097579,77.327065) {};
		\node [circle,fill,scale=0.25] (21) at (43.516911,79.922344) {};
		\node [circle,fill,scale=0.25] (20) at (63.727394,72.514559) {};
		\node [circle,fill,scale=0.25] (19) at (72.933484,63.073464) {};
		\node [circle,fill,scale=0.25] (18) at (71.666496,52.559517) {};
		\node [circle,fill,scale=0.25] (17) at (75.947685,29.549402) {};
		\node [circle,fill,scale=0.25] (16) at (60.028610,17.288241) {};
		\node [circle,fill,scale=0.25] (15) at (30.172679,13.528150) {};
		\path (15) ++(2.5,4) node {$a'$};
		\node [circle,fill,scale=0.25] (14) at (0.000000,49.994891) {};
		\node [circle,fill,scale=0.25] (13) at (23.306428,79.043628) {};
		\node [circle,fill,scale=0.25] (12) at (43.813222,87.677530) {};
		\node [circle,fill,scale=0.25] (11) at (60.447532,82.783282) {};
		\node [circle,fill,scale=0.25] (10) at (80.412792,65.321344) {};
		\node [circle,fill,scale=0.25] (9) at (81.485644,42.413405) {};
		\node [circle,fill,scale=0.25] (8) at (71.247574,15.377542) {};
		\node [circle,fill,scale=0.25] (7) at (25.002556,6.702770) {};
		\node [circle,fill,scale=0.25] (6) at (25.002556,93.297229) {};
		\node [circle,fill,scale=0.25] (5) at (73.188923,83.059160) {};
		\node [circle,fill,scale=0.25] (4) at (90.180852,49.586185) {};
		\node [circle,fill,scale=0.25] (3) at (74.997445,6.702770) {};
		\node [circle,fill,scale=0.25] (2) at (74.997445,93.297229) {};
		\node [circle,fill,scale=0.25] (1) at (99.999999,49.994891) {};
		\draw [black] (50) to (47);
		\draw [black] (50) to (44);
		\draw [black] (50) to (48);
		\draw [black] (49) to (48);
		\draw [black] (49) to (43);
		\draw [black] (49) to (40);
		\draw [black] (48) to (39);
		\draw [black] (47) to (38);
		\draw [black] (47) to (46);
		\draw [black] (46) to (45);
		\draw [black] (46) to (36);
		\draw [black] (45) to (35);
		\draw [black] (45) to (44);
		\draw [black] (44) to (34);
		\draw [black] (43) to (34);
		\draw [black] (43) to (42);
		\draw [black] (42) to (33);
		\draw [black] (42) to (41);
		\draw [black] (41) to (40);
		\draw [black] (41) to (30);
		\draw [black] (40) to (29);
		\draw [black] (39) to (27);
		\draw [black] (39) to (28);
		\draw [black] (38) to (27);
		\draw [black] (38) to (37);
		\draw [black] (37) to (26);
		\draw [black] (37) to (36);
		\draw [black] (36) to (25);
		\draw [black] (35) to (24);
		\draw [black] (35) to (25);
		\draw [black] (34) to (23);
		\draw [black] (34) to (24);
		\draw [black] (33) to (23);
		\draw [black] (33) to (32);
		\draw [black] (32) to (22);
		\draw [black] (32) to (30);
		\draw [black] (31) to (20);
		\draw [black] (31) to (30);
		\draw [black] (31) to (21);
		\draw [black] (30) to (18);
		\draw [black] (29) to (28);
		\draw [black] (29) to (18);
		\draw [black] (28) to (17);
		\draw [black] (27) to (16);
		\draw [black] (26) to (15);
		\draw [black] (26) to (16);
		\draw [black] (25) to (15);
		\draw [black] (24) to (14);
		\draw [black] (23) to (13);
		\draw [black] (22) to (21);
		\draw [black] (22) to (13);
		\draw [black] (21) to (12);
		\draw [black] (20) to (19);
		\draw [black] (20) to (11);
		\draw [black] (19) to (10);
		\draw [black] (19) to (18);
		\draw [black] (18) to (9);
		\draw [black] (17) to (8);
		\draw [black] (17) to (9);
		\draw [black] (16) to (8);
		\draw [black] (15) to (7);
		\draw [black] (14) to (6);
		\draw [black] (14) to (7);
		\draw [black] (13) to (6);
		\draw [black] (12) to (11);
		\draw [black] (12) to (6);
		\draw [black] (11) to (5);
		\draw [black] (10) to (4);
		\draw [black] (10) to (5);
		\draw [black] (9) to (4);
		\draw [black] (8) to (3);
		\draw [black] (7) to (3);
		\draw [black] (6) to (2);
		\draw [black] (5) to (2);
		\draw [black] (4) to (1);
		\draw [black] (3) to (1);
		\draw [black] (2) to (1);
		\draw [red, ultra thick] (7) -- (3) -- (8) -- (17) -- (9) -- (4) -- (1) -- (2) -- (5) -- (10) -- (19) -- (18) -- (30) -- (41) -- (40) -- (29) -- (28) -- (39) -- (27) -- (16) -- (26) -- (37) -- (38) -- (47) -- (50) -- (48) -- (49) -- (43) -- (42) -- (33) -- (32) -- (22) -- (21) -- (31) -- (20) -- (11) -- (12) -- (6) -- (13) -- (23) -- (34) -- (44) -- (45) -- (46) -- (36) -- (25) -- (35) -- (24) -- (14) -- (7);
		\begin{pgfonlayer}{bg}
			\path[fill=darkgray] (5.center) -- (10.center) -- (19.center) -- 
			(20.center) -- (11.center);
		\end{pgfonlayer}
		\path[draw=black, very thick] (3) circle[radius=2];
		\path[draw=black, very thick] (8) circle[radius=2];
		\path[draw=black, very thick] (14) circle[radius=2];
		\path[draw=black, very thick] (39) circle[radius=2];
		\path[draw=black, very thick] (48) circle[radius=2];
	\end{tikzpicture}
\end{minipage}
	\caption{The planar $K_2$-hypohamiltonian graph $G_{50}$. A face of which the boundary is an extendable $5$-cycle is filled in. All vertices for which the vertex-deleted subgraph is non-hamiltonian are circled. On the left-hand side a hamiltonian cycle of $G_{50} - a'$ not containing $bb'$ and on the right-hand side one which does contain $bb'$ are marked by a thick line.}
	\label{fig:g50}
\end{figure}
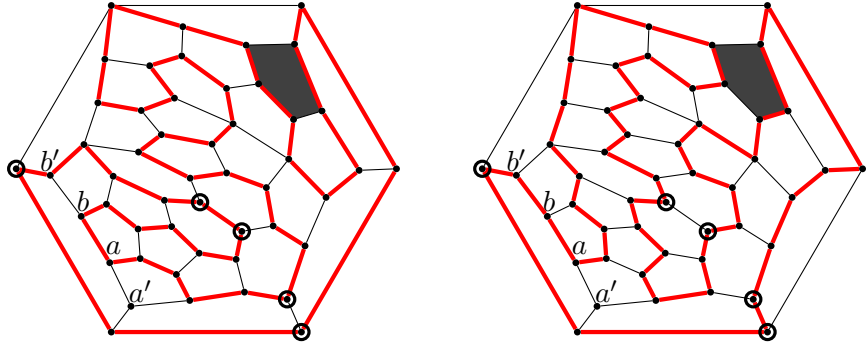

In Fig.~\ref{fig:g52}, we see that $(G_{52}, a_{52}, a'_{52}, b_{52}, b'_{52})$ satisfies the gluing property, such that $G_{52} - a'_{52}$ has at least two hamiltonian cycles, one containing $b_{52}b'_{52}$ and one not containing it. Moreover, we see that in this embedding $a_{52}, a'_{52}, b_{52}, b'_{52}$ are co-facial. Also note that $G_{52} - v$ is hamiltonian for $v\in \{a_{52}, a'_{52}, b_{52}, b'_{52}\}$ and that $v$ is not among the vertices of the extendable $5$-cycle and their neighbours.
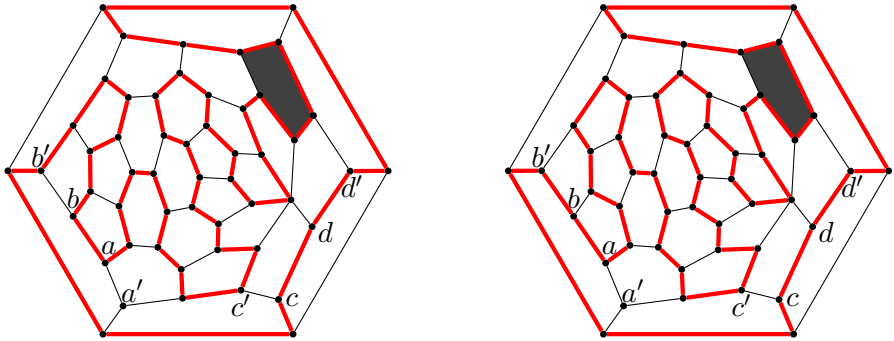
\begin{figure}[!htb]
	\centering
\begin{minipage}{60mm}
	\begin{tikzpicture}[scale=0.050]
    \definecolor{marked}{rgb}{0.25,0.5,0.25}
    \node [circle,fill,scale=0.25] (52) at (38.295997,49.276143) {};
    \node [circle,fill,scale=0.25] (51) at (46.962729,57.101634) {};
    \node [circle,fill,scale=0.25] (50) at (41.885942,39.151914) {};
    \node [circle,fill,scale=0.25] (49) at (32.710553,49.628290) {};
    \node [circle,fill,scale=0.25] (48) at (41.044701,59.380809) {};
    \node [circle,fill,scale=0.25] (47) at (52.215590,61.953439) {};
    \node [circle,fill,scale=0.25] (46) at (50.454855,48.268611) {};
    \node [circle,fill,scale=0.25] (45) at (48.449572,40.413774) {};
    \node [circle,fill,scale=0.25] (44) at (39.430693,29.722195) {};
    \node [circle,fill,scale=0.25] (43) at (29.208645,40.648538) {};
    \node [circle,fill,scale=0.25] (42) at (21.715737,44.512375) {};
    \node [circle,fill,scale=0.25] (41) at (21.608137,55.223516) {};
    \node [circle,fill,scale=0.25] (40) at (28.983662,58.911279) {};
    \node [circle,fill,scale=0.25] (39) at (38.824218,69.906094) {};
    \node [circle,fill,scale=0.25] (38) at (52.743811,70.072386) {};
    \node [circle,fill,scale=0.25] (37) at (59.737844,54.617041) {};
    \node [circle,fill,scale=0.25] (36) at (58.642275,47.847991) {};
    \node [circle,fill,scale=0.25] (35) at (55.375132,35.943461) {};
    \node [circle,fill,scale=0.25] (34) at (55.159932,29.047247) {};
    \node [circle,fill,scale=0.25] (33) at (45.573704,23.892205) {};
    \node [circle,fill,scale=0.25] (32) at (31.957349,30.260199) {};
    \node [circle,fill,scale=0.25] (31) at (17.157389,37.938962) {};
    \node [circle,fill,scale=0.25] (30) at (17.147606,62.021912) {};
    \node [circle,fill,scale=0.25] (29) at (31.673676,69.475692) {};
    \node [circle,fill,scale=0.25] (28) at (45.182430,75.902376) {};
    \node [circle,fill,scale=0.25] (27) at (61.860509,66.502006) {};
    \node [circle,fill,scale=0.25] (26) at (66.692749,54.304021) {};
    \node [circle,fill,scale=0.25] (25) at (64.129901,41.303923) {};
    \node [circle,fill,scale=0.25] (24) at (65.577617,29.409177) {};
    \node [circle,fill,scale=0.25] (23) at (45.935633,16.223224) {};
    \node [circle,fill,scale=0.25] (22) at (25.579573,25.564904) {};
    \node [circle,fill,scale=0.25] (21) at (8.725420,50.019564) {};
    \node [circle,fill,scale=0.25] (20) at (25.550227,74.376406) {};
    \node [circle,fill,scale=0.25] (19) at (46.101925,83.551795) {};
    \node [circle,fill,scale=0.25] (18) at (66.213438,70.062604) {};
    \node [circle,fill,scale=0.25] (17) at (75.349699,58.148294) {};
    \node [circle,fill,scale=0.25] (16) at (74.449767,42.321237) {};
    \node [circle,fill,scale=0.25] (15) at (61.312724,18.365451) {};
    \node [circle,fill,scale=0.25] (14) at (30.294433,14.178813) {};
    \node [circle,fill,scale=0.25] (13) at (0.000000,50.000001) {};
    \node [circle,fill,scale=0.25] (12) at (30.372688,85.801624) {};
    \node [circle,fill,scale=0.25] (11) at (61.068178,81.478041) {};
    \node [circle,fill,scale=0.25] (10) at (80.358014,64.663015) {};
    \node [circle,fill,scale=0.25] (9) at (80.005867,35.190258) {};
    \node [circle,fill,scale=0.25] (8) at (71.172842,15.880858) {};
    \node [circle,fill,scale=0.25] (7) at (24.992662,6.695686) {};
    \node [circle,fill,scale=0.25] (6) at (24.992662,93.304313) {};
    \node [circle,fill,scale=0.25] (5) at (71.192407,84.148490) {};
    \node [circle,fill,scale=0.25] (4) at (90.012712,49.921746) {};
    \node [circle,fill,scale=0.25] (3) at (74.997552,6.695686) {};
    \node [circle,fill,scale=0.25] (2) at (74.997552,93.304313) {};
    \node [circle,fill,scale=0.25] (1) at (99.999999,50.000001) {};
    \path (31) ++(0,5) node {$b$};
    \path (22) ++(1,4) node {$a$};
    \path (21) ++(0,4) node {$b'$};
    \path (15) ++(0,-4) node {$c'$};
    \path (14) ++(2.5,4) node {$a'$};
    \path (9) ++(3.5,-1.5) node {$d$};
    \path (8) ++(3.5,0) node {$c$};
    \path (4) ++(0.5,-4) node {$d'$};

    \draw [black] (52) to (48);
    \draw [black] (52) to (50);
    \draw [black] (52) to (49);
    \draw [black] (51) to (47);
    \draw [black] (51) to (46);
    \draw [black] (51) to (48);
    \draw [black] (50) to (44);
    \draw [black] (50) to (45);
    \draw [black] (49) to (43);
    \draw [black] (49) to (40);
    \draw [black] (48) to (39);
    \draw [black] (47) to (38);
    \draw [black] (47) to (37);
    \draw [black] (46) to (36);
    \draw [black] (46) to (45);
    \draw [black] (45) to (35);
    \draw [black] (44) to (32);
    \draw [black] (44) to (33);
    \draw [black] (43) to (42);
    \draw [black] (43) to (32);
    \draw [black] (42) to (31);
    \draw [black] (42) to (41);
    \draw [black] (41) to (30);
    \draw [black] (41) to (40);
    \draw [black] (40) to (29);
    \draw [black] (39) to (28);
    \draw [black] (39) to (29);
    \draw [black] (38) to (27);
    \draw [black] (38) to (28);
    \draw [black] (37) to (36);
    \draw [black] (37) to (26);
    \draw [black] (36) to (25);
    \draw [black] (35) to (34);
    \draw [black] (35) to (25);
    \draw [black] (34) to (24);
    \draw [black] (34) to (33);
    \draw [black] (33) to (23);
    \draw [black] (32) to (22);
    \draw [black] (31) to (21);
    \draw [black] (31) to (22);
    \draw [black] (30) to (20);
    \draw [black] (30) to (21);
    \draw [black] (29) to (20);
    \draw [black] (28) to (19);
    \draw [black] (27) to (18);
    \draw [black] (27) to (26);
    \draw [black] (26) to (16);
    \draw [black] (25) to (16);
    \draw [black] (24) to (15);
    \draw [black] (24) to (16);
    \draw [black] (23) to (14);
    \draw [black] (23) to (15);
    \draw [black] (22) to (14);
    \draw [black] (21) to (13);
    \draw [black] (20) to (12);
    \draw [black] (19) to (11);
    \draw [black] (19) to (12);
    \draw [black] (18) to (17);
    \draw [black] (18) to (11);
    \draw [black] (17) to (10);
    \draw [black] (17) to (16);
    \draw [black] (16) to (9);
    \draw [black] (15) to (8);
    \draw [black] (14) to (7);
    \draw [black] (13) to (6);
    \draw [black] (13) to (7);
    \draw [black] (12) to (6);
    \draw [black] (11) to (5);
    \draw [black] (10) to (4);
    \draw [black] (10) to (5);
    \draw [black] (9) to (8);
    \draw [black] (9) to (4);
    \draw [black] (8) to (3);
    \draw [black] (7) to (3);
    \draw [black] (6) to (2);
    \draw [black] (5) to (2);
    \draw [black] (4) to (1);
    \draw [black] (3) to (1);
    \draw [black] (2) to (1);
    \draw [red, ultra thick] (7) -- (3) -- (8) -- (9) -- (4) -- (1) -- (2) -- (6) -- (12) -- (19) -- (11) -- (5) -- (10) -- (17) -- (18) -- (27) -- (26) -- (16) -- (25) -- (36) -- (37) -- (47) -- (38) -- (28) -- (39) -- (48) -- (51) -- (46) -- (45) -- (35) -- (34) -- (24) -- (15) -- (23) -- (33) -- (44) -- (50) -- (52) -- (49) -- (43) -- (32) -- (22) -- (31) -- (42) -- (41) -- (40) -- (29) -- (20) -- (30) -- (21) -- (13) -- (7);
\begin{pgfonlayer}{bg}
\path[fill=darkgray] (5.center) -- (10.center) -- (17.center) -- (18.center) -- 
(11.center);
\end{pgfonlayer}
\end{tikzpicture}
\end{minipage}
\begin{minipage}{60mm}
\centering
	\begin{tikzpicture}[scale=0.050]
    \definecolor{marked}{rgb}{0.25,0.5,0.25}
    \node [circle,fill,scale=0.25] (52) at (38.295997,49.276143) {};
    \node [circle,fill,scale=0.25] (51) at (46.962729,57.101634) {};
    \node [circle,fill,scale=0.25] (50) at (41.885942,39.151914) {};
    \node [circle,fill,scale=0.25] (49) at (32.710553,49.628290) {};
    \node [circle,fill,scale=0.25] (48) at (41.044701,59.380809) {};
    \node [circle,fill,scale=0.25] (47) at (52.215590,61.953439) {};
    \node [circle,fill,scale=0.25] (46) at (50.454855,48.268611) {};
    \node [circle,fill,scale=0.25] (45) at (48.449572,40.413774) {};
    \node [circle,fill,scale=0.25] (44) at (39.430693,29.722195) {};
    \node [circle,fill,scale=0.25] (43) at (29.208645,40.648538) {};
    \node [circle,fill,scale=0.25] (42) at (21.715737,44.512375) {};
    \node [circle,fill,scale=0.25] (41) at (21.608137,55.223516) {};
    \node [circle,fill,scale=0.25] (40) at (28.983662,58.911279) {};
    \node [circle,fill,scale=0.25] (39) at (38.824218,69.906094) {};
    \node [circle,fill,scale=0.25] (38) at (52.743811,70.072386) {};
    \node [circle,fill,scale=0.25] (37) at (59.737844,54.617041) {};
    \node [circle,fill,scale=0.25] (36) at (58.642275,47.847991) {};
    \node [circle,fill,scale=0.25] (35) at (55.375132,35.943461) {};
    \node [circle,fill,scale=0.25] (34) at (55.159932,29.047247) {};
    \node [circle,fill,scale=0.25] (33) at (45.573704,23.892205) {};
    \node [circle,fill,scale=0.25] (32) at (31.957349,30.260199) {};
    \node [circle,fill,scale=0.25] (31) at (17.157389,37.938962) {};
    \node [circle,fill,scale=0.25] (30) at (17.147606,62.021912) {};
    \node [circle,fill,scale=0.25] (29) at (31.673676,69.475692) {};
    \node [circle,fill,scale=0.25] (28) at (45.182430,75.902376) {};
    \node [circle,fill,scale=0.25] (27) at (61.860509,66.502006) {};
    \node [circle,fill,scale=0.25] (26) at (66.692749,54.304021) {};
    \node [circle,fill,scale=0.25] (25) at (64.129901,41.303923) {};
    \node [circle,fill,scale=0.25] (24) at (65.577617,29.409177) {};
    \node [circle,fill,scale=0.25] (23) at (45.935633,16.223224) {};
    \node [circle,fill,scale=0.25] (22) at (25.579573,25.564904) {};
    \node [circle,fill,scale=0.25] (21) at (8.725420,50.019564) {};
    \node [circle,fill,scale=0.25] (20) at (25.550227,74.376406) {};
    \node [circle,fill,scale=0.25] (19) at (46.101925,83.551795) {};
    \node [circle,fill,scale=0.25] (18) at (66.213438,70.062604) {};
    \node [circle,fill,scale=0.25] (17) at (75.349699,58.148294) {};
    \node [circle,fill,scale=0.25] (16) at (74.449767,42.321237) {};
    \node [circle,fill,scale=0.25] (15) at (61.312724,18.365451) {};
    \node [circle,fill,scale=0.25] (14) at (30.294433,14.178813) {};
    \node [circle,fill,scale=0.25] (13) at (0.000000,50.000001) {};
    \node [circle,fill,scale=0.25] (12) at (30.372688,85.801624) {};
    \node [circle,fill,scale=0.25] (11) at (61.068178,81.478041) {};
    \node [circle,fill,scale=0.25] (10) at (80.358014,64.663015) {};
    \node [circle,fill,scale=0.25] (9) at (80.005867,35.190258) {};
    \node [circle,fill,scale=0.25] (8) at (71.172842,15.880858) {};
    \node [circle,fill,scale=0.25] (7) at (24.992662,6.695686) {};
    \node [circle,fill,scale=0.25] (6) at (24.992662,93.304313) {};
    \node [circle,fill,scale=0.25] (5) at (71.192407,84.148490) {};
    \node [circle,fill,scale=0.25] (4) at (90.012712,49.921746) {};
    \node [circle,fill,scale=0.25] (3) at (74.997552,6.695686) {};
    \node [circle,fill,scale=0.25] (2) at (74.997552,93.304313) {};
    \node [circle,fill,scale=0.25] (1) at (99.999999,50.000001) {};
    \path (31) ++(0,5) node {$b$};
    \path (22) ++(1,4) node {$a$};
    \path (21) ++(0,4) node {$b'$};
    \path (15) ++(0,-4) node {$c'$};
    \path (14) ++(2.5,4) node {$a'$};
    \path (9) ++(3.5,-1.5) node {$d$};
    \path (8) ++(3.5,0) node {$c$};
    \path (4) ++(0.5,-4) node {$d'$};
    \draw [black] (52) to (48);
    \draw [black] (52) to (50);
    \draw [black] (52) to (49);
    \draw [black] (51) to (47);
    \draw [black] (51) to (46);
    \draw [black] (51) to (48);
    \draw [black] (50) to (44);
    \draw [black] (50) to (45);
    \draw [black] (49) to (43);
    \draw [black] (49) to (40);
    \draw [black] (48) to (39);
    \draw [black] (47) to (38);
    \draw [black] (47) to (37);
    \draw [black] (46) to (36);
    \draw [black] (46) to (45);
    \draw [black] (45) to (35);
    \draw [black] (44) to (32);
    \draw [black] (44) to (33);
    \draw [black] (43) to (42);
    \draw [black] (43) to (32);
    \draw [black] (42) to (31);
    \draw [black] (42) to (41);
    \draw [black] (41) to (30);
    \draw [black] (41) to (40);
    \draw [black] (40) to (29);
    \draw [black] (39) to (28);
    \draw [black] (39) to (29);
    \draw [black] (38) to (27);
    \draw [black] (38) to (28);
    \draw [black] (37) to (36);
    \draw [black] (37) to (26);
    \draw [black] (36) to (25);
    \draw [black] (35) to (34);
    \draw [black] (35) to (25);
    \draw [black] (34) to (24);
    \draw [black] (34) to (33);
    \draw [black] (33) to (23);
    \draw [black] (32) to (22);
    \draw [black] (31) to (21);
    \draw [black] (31) to (22);
    \draw [black] (30) to (20);
    \draw [black] (30) to (21);
    \draw [black] (29) to (20);
    \draw [black] (28) to (19);
    \draw [black] (27) to (18);
    \draw [black] (27) to (26);
    \draw [black] (26) to (16);
    \draw [black] (25) to (16);
    \draw [black] (24) to (15);
    \draw [black] (24) to (16);
    \draw [black] (23) to (14);
    \draw [black] (23) to (15);
    \draw [black] (22) to (14);
    \draw [black] (21) to (13);
    \draw [black] (20) to (12);
    \draw [black] (19) to (11);
    \draw [black] (19) to (12);
    \draw [black] (18) to (17);
    \draw [black] (18) to (11);
    \draw [black] (17) to (10);
    \draw [black] (17) to (16);
    \draw [black] (16) to (9);
    \draw [black] (15) to (8);
    \draw [black] (14) to (7);
    \draw [black] (13) to (6);
    \draw [black] (13) to (7);
    \draw [black] (12) to (6);
    \draw [black] (11) to (5);
    \draw [black] (10) to (4);
    \draw [black] (10) to (5);
    \draw [black] (9) to (8);
    \draw [black] (9) to (4);
    \draw [black] (8) to (3);
    \draw [black] (7) to (3);
    \draw [black] (6) to (2);
    \draw [black] (5) to (2);
    \draw [black] (4) to (1);
    \draw [black] (3) to (1);
    \draw [black] (2) to (1);
    \draw [red, ultra thick] (7) -- (3) -- (8) -- (9) -- (4) -- (1) -- (2) -- (6) -- (12) -- (19) -- (11) -- (5) -- (10) -- (17) -- (18) -- (27) -- (26) -- (16) -- (25) -- (36) -- (37) -- (47) -- (38) -- (28) -- (39) -- (48) -- (51) -- (46) -- (45) -- (35) -- (34) -- (24) -- (15) -- (23) -- (33) -- (44) -- (50) -- (52) -- (49) -- (40) -- (29) -- (20) -- (30) -- (41) -- (42) -- (43) -- (32) -- (22) -- (31) -- (21) -- (13) -- (7);
\begin{pgfonlayer}{bg}
\path[fill=darkgray] (5.center) -- (10.center) -- (17.center) -- (18.center) -- 
(11.center);
\end{pgfonlayer}
\end{tikzpicture}
\end{minipage}

	\caption{The planar $K_2$-hypohamiltonian graph $G_{52}$. A face of which the boundary is an extendable $5$-cycle is filled in. All of its vertex-deleted subgraphs are hamiltonian. On the left-hand side a hamiltonian cycle of $G_{52} - a'$ not containing $bb'$ and on the right-hand side one which does contain $bb'$ are marked by a thick line.}
	\label{fig:g52}
\end{figure}

In Fig.~\ref{fig:g52v2}, we see that $(G_{52}, c_{52}, c'_{52}, d_{52}, d'_{52})$ satisfies the gluing property, such that $G_{52} - c'_{52}$ has at least two hamiltonian cycles, one containing $d_{52}d'_{52}$ and one not containing it. Moreover, we see that in this embedding $c_{52}, c'_{52}, d_{52}, d'_{52}$ are not co-facial. Also note that $G_{52} - v$ is hamiltonian for $v\in \{c_{52}, c'_{52}, d_{52}, d'_{52}\}$. However, $v$ is not disjoint from the vertices of the extendable $5$-cycle and their neighbours, but this is of no concern for the proof.
\begin{figure}[!htb]
	\centering
	\begin{minipage}{60mm}
\centering
	\begin{tikzpicture}[scale=0.050]
    \definecolor{marked}{rgb}{0.25,0.5,0.25}
    \node [circle,fill,scale=0.25] (52) at (38.295997,49.276143) {};
    \node [circle,fill,scale=0.25] (51) at (46.962729,57.101634) {};
    \node [circle,fill,scale=0.25] (50) at (41.885942,39.151914) {};
    \node [circle,fill,scale=0.25] (49) at (32.710553,49.628290) {};
    \node [circle,fill,scale=0.25] (48) at (41.044701,59.380809) {};
    \node [circle,fill,scale=0.25] (47) at (52.215590,61.953439) {};
    \node [circle,fill,scale=0.25] (46) at (50.454855,48.268611) {};
    \node [circle,fill,scale=0.25] (45) at (48.449572,40.413774) {};
    \node [circle,fill,scale=0.25] (44) at (39.430693,29.722195) {};
    \node [circle,fill,scale=0.25] (43) at (29.208645,40.648538) {};
    \node [circle,fill,scale=0.25] (42) at (21.715737,44.512375) {};
    \node [circle,fill,scale=0.25] (41) at (21.608137,55.223516) {};
    \node [circle,fill,scale=0.25] (40) at (28.983662,58.911279) {};
    \node [circle,fill,scale=0.25] (39) at (38.824218,69.906094) {};
    \node [circle,fill,scale=0.25] (38) at (52.743811,70.072386) {};
    \node [circle,fill,scale=0.25] (37) at (59.737844,54.617041) {};
    \node [circle,fill,scale=0.25] (36) at (58.642275,47.847991) {};
    \node [circle,fill,scale=0.25] (35) at (55.375132,35.943461) {};
    \node [circle,fill,scale=0.25] (34) at (55.159932,29.047247) {};
    \node [circle,fill,scale=0.25] (33) at (45.573704,23.892205) {};
    \node [circle,fill,scale=0.25] (32) at (31.957349,30.260199) {};
    \node [circle,fill,scale=0.25] (31) at (17.157389,37.938962) {};
    \node [circle,fill,scale=0.25] (30) at (17.147606,62.021912) {};
    \node [circle,fill,scale=0.25] (29) at (31.673676,69.475692) {};
    \node [circle,fill,scale=0.25] (28) at (45.182430,75.902376) {};
    \node [circle,fill,scale=0.25] (27) at (61.860509,66.502006) {};
    \node [circle,fill,scale=0.25] (26) at (66.692749,54.304021) {};
    \node [circle,fill,scale=0.25] (25) at (64.129901,41.303923) {};
    \node [circle,fill,scale=0.25] (24) at (65.577617,29.409177) {};
    \node [circle,fill,scale=0.25] (23) at (45.935633,16.223224) {};
    \node [circle,fill,scale=0.25] (22) at (25.579573,25.564904) {};
    \node [circle,fill,scale=0.25] (21) at (8.725420,50.019564) {};
    \node [circle,fill,scale=0.25] (20) at (25.550227,74.376406) {};
    \node [circle,fill,scale=0.25] (19) at (46.101925,83.551795) {};
    \node [circle,fill,scale=0.25] (18) at (66.213438,70.062604) {};
    \node [circle,fill,scale=0.25] (17) at (75.349699,58.148294) {};
    \node [circle,fill,scale=0.25] (16) at (74.449767,42.321237) {};
    \node [circle,fill,scale=0.25] (15) at (61.312724,18.365451) {};
    \node [circle,fill,scale=0.25] (14) at (30.294433,14.178813) {};
    \node [circle,fill,scale=0.25] (13) at (0.000000,50.000001) {};
    \node [circle,fill,scale=0.25] (12) at (30.372688,85.801624) {};
    \node [circle,fill,scale=0.25] (11) at (61.068178,81.478041) {};
    \node [circle,fill,scale=0.25] (10) at (80.358014,64.663015) {};
    \node [circle,fill,scale=0.25] (9) at (80.005867,35.190258) {};
    \node [circle,fill,scale=0.25] (8) at (71.172842,15.880858) {};
    \node [circle,fill,scale=0.25] (7) at (24.992662,6.695686) {};
    \node [circle,fill,scale=0.25] (6) at (24.992662,93.304313) {};
    \node [circle,fill,scale=0.25] (5) at (71.192407,84.148490) {};
    \node [circle,fill,scale=0.25] (4) at (90.012712,49.921746) {};
    \node [circle,fill,scale=0.25] (3) at (74.997552,6.695686) {};
    \node [circle,fill,scale=0.25] (2) at (74.997552,93.304313) {};
    \node [circle,fill,scale=0.25] (1) at (99.999999,50.000001) {};
    \path (31) ++(0,5) node {$b$};
    \path (22) ++(1,4) node {$a$};
    \path (21) ++(0,4) node {$b'$};
    \path (15) ++(0,-4) node {$c'$};
    \path (14) ++(2.5,4) node {$a'$};
    \path (9) ++(3.5,-1.5) node {$d$};
    \path (8) ++(3.5,0) node {$c$};
    \path (4) ++(0.5,-4) node {$d'$};
    \draw [black] (52) to (48);
    \draw [black] (52) to (50);
    \draw [black] (52) to (49);
    \draw [black] (51) to (47);
    \draw [black] (51) to (46);
    \draw [black] (51) to (48);
    \draw [black] (50) to (44);
    \draw [black] (50) to (45);
    \draw [black] (49) to (43);
    \draw [black] (49) to (40);
    \draw [black] (48) to (39);
    \draw [black] (47) to (38);
    \draw [black] (47) to (37);
    \draw [black] (46) to (36);
    \draw [black] (46) to (45);
    \draw [black] (45) to (35);
    \draw [black] (44) to (32);
    \draw [black] (44) to (33);
    \draw [black] (43) to (42);
    \draw [black] (43) to (32);
    \draw [black] (42) to (31);
    \draw [black] (42) to (41);
    \draw [black] (41) to (30);
    \draw [black] (41) to (40);
    \draw [black] (40) to (29);
    \draw [black] (39) to (28);
    \draw [black] (39) to (29);
    \draw [black] (38) to (27);
    \draw [black] (38) to (28);
    \draw [black] (37) to (36);
    \draw [black] (37) to (26);
    \draw [black] (36) to (25);
    \draw [black] (35) to (34);
    \draw [black] (35) to (25);
    \draw [black] (34) to (24);
    \draw [black] (34) to (33);
    \draw [black] (33) to (23);
    \draw [black] (32) to (22);
    \draw [black] (31) to (21);
    \draw [black] (31) to (22);
    \draw [black] (30) to (20);
    \draw [black] (30) to (21);
    \draw [black] (29) to (20);
    \draw [black] (28) to (19);
    \draw [black] (27) to (18);
    \draw [black] (27) to (26);
    \draw [black] (26) to (16);
    \draw [black] (25) to (16);
    \draw [black] (24) to (15);
    \draw [black] (24) to (16);
    \draw [black] (23) to (14);
    \draw [black] (23) to (15);
    \draw [black] (22) to (14);
    \draw [black] (21) to (13);
    \draw [black] (20) to (12);
    \draw [black] (19) to (11);
    \draw [black] (19) to (12);
    \draw [black] (18) to (17);
    \draw [black] (18) to (11);
    \draw [black] (17) to (10);
    \draw [black] (17) to (16);
    \draw [black] (16) to (9);
    \draw [black] (15) to (8);
    \draw [black] (14) to (7);
    \draw [black] (13) to (6);
    \draw [black] (13) to (7);
    \draw [black] (12) to (6);
    \draw [black] (11) to (5);
    \draw [black] (10) to (4);
    \draw [black] (10) to (5);
    \draw [black] (9) to (8);
    \draw [black] (9) to (4);
    \draw [black] (8) to (3);
    \draw [black] (7) to (3);
    \draw [black] (6) to (2);
    \draw [black] (5) to (2);
    \draw [black] (4) to (1);
    \draw [black] (3) to (1);
    \draw [black] (2) to (1);
    \draw [red, ultra thick] (8) -- (3) -- (7) -- (14) -- (23) -- (33) -- (44) -- (50) -- (45) -- (46) -- (51) -- (47) -- (38) -- (28) -- (39) -- (48) -- (52) -- (49) -- (43) -- (32) -- (22) -- (31) -- (42) -- (41) -- (40) -- (29) -- (20) -- (30) -- (21) -- (13) -- (6) -- (12) -- (19) -- (11) -- (5) -- (2) -- (1) -- (4) -- (10) -- (17) -- (18) -- (27) -- (26) -- (37) -- (36) -- (25) -- (35) -- (34) -- (24) -- (16) -- (9) -- (8);
\begin{pgfonlayer}{bg}
\path[fill=darkgray] (5.center) -- (10.center) -- (17.center) -- (18.center) -- 
(11.center);
\end{pgfonlayer}
\end{tikzpicture}
\end{minipage}
\begin{minipage}{60mm}
\centering
	\begin{tikzpicture}[scale=0.050]
    \definecolor{marked}{rgb}{0.25,0.5,0.25}
    \node [circle,fill,scale=0.25] (52) at (38.295997,49.276143) {};
    \node [circle,fill,scale=0.25] (51) at (46.962729,57.101634) {};
    \node [circle,fill,scale=0.25] (50) at (41.885942,39.151914) {};
    \node [circle,fill,scale=0.25] (49) at (32.710553,49.628290) {};
    \node [circle,fill,scale=0.25] (48) at (41.044701,59.380809) {};
    \node [circle,fill,scale=0.25] (47) at (52.215590,61.953439) {};
    \node [circle,fill,scale=0.25] (46) at (50.454855,48.268611) {};
    \node [circle,fill,scale=0.25] (45) at (48.449572,40.413774) {};
    \node [circle,fill,scale=0.25] (44) at (39.430693,29.722195) {};
    \node [circle,fill,scale=0.25] (43) at (29.208645,40.648538) {};
    \node [circle,fill,scale=0.25] (42) at (21.715737,44.512375) {};
    \node [circle,fill,scale=0.25] (41) at (21.608137,55.223516) {};
    \node [circle,fill,scale=0.25] (40) at (28.983662,58.911279) {};
    \node [circle,fill,scale=0.25] (39) at (38.824218,69.906094) {};
    \node [circle,fill,scale=0.25] (38) at (52.743811,70.072386) {};
    \node [circle,fill,scale=0.25] (37) at (59.737844,54.617041) {};
    \node [circle,fill,scale=0.25] (36) at (58.642275,47.847991) {};
    \node [circle,fill,scale=0.25] (35) at (55.375132,35.943461) {};
    \node [circle,fill,scale=0.25] (34) at (55.159932,29.047247) {};
    \node [circle,fill,scale=0.25] (33) at (45.573704,23.892205) {};
    \node [circle,fill,scale=0.25] (32) at (31.957349,30.260199) {};
    \node [circle,fill,scale=0.25] (31) at (17.157389,37.938962) {};
    \node [circle,fill,scale=0.25] (30) at (17.147606,62.021912) {};
    \node [circle,fill,scale=0.25] (29) at (31.673676,69.475692) {};
    \node [circle,fill,scale=0.25] (28) at (45.182430,75.902376) {};
    \node [circle,fill,scale=0.25] (27) at (61.860509,66.502006) {};
    \node [circle,fill,scale=0.25] (26) at (66.692749,54.304021) {};
    \node [circle,fill,scale=0.25] (25) at (64.129901,41.303923) {};
    \node [circle,fill,scale=0.25] (24) at (65.577617,29.409177) {};
    \node [circle,fill,scale=0.25] (23) at (45.935633,16.223224) {};
    \node [circle,fill,scale=0.25] (22) at (25.579573,25.564904) {};
    \node [circle,fill,scale=0.25] (21) at (8.725420,50.019564) {};
    \node [circle,fill,scale=0.25] (20) at (25.550227,74.376406) {};
    \node [circle,fill,scale=0.25] (19) at (46.101925,83.551795) {};
    \node [circle,fill,scale=0.25] (18) at (66.213438,70.062604) {};
    \node [circle,fill,scale=0.25] (17) at (75.349699,58.148294) {};
    \node [circle,fill,scale=0.25] (16) at (74.449767,42.321237) {};
    \node [circle,fill,scale=0.25] (15) at (61.312724,18.365451) {};
    \node [circle,fill,scale=0.25] (14) at (30.294433,14.178813) {};
    \node [circle,fill,scale=0.25] (13) at (0.000000,50.000001) {};
    \node [circle,fill,scale=0.25] (12) at (30.372688,85.801624) {};
    \node [circle,fill,scale=0.25] (11) at (61.068178,81.478041) {};
    \node [circle,fill,scale=0.25] (10) at (80.358014,64.663015) {};
    \node [circle,fill,scale=0.25] (9) at (80.005867,35.190258) {};
    \node [circle,fill,scale=0.25] (8) at (71.172842,15.880858) {};
    \node [circle,fill,scale=0.25] (7) at (24.992662,6.695686) {};
    \node [circle,fill,scale=0.25] (6) at (24.992662,93.304313) {};
    \node [circle,fill,scale=0.25] (5) at (71.192407,84.148490) {};
    \node [circle,fill,scale=0.25] (4) at (90.012712,49.921746) {};
    \node [circle,fill,scale=0.25] (3) at (74.997552,6.695686) {};
    \node [circle,fill,scale=0.25] (2) at (74.997552,93.304313) {};
    \node [circle,fill,scale=0.25] (1) at (99.999999,50.000001) {};
    \path (31) ++(0,5) node {$b$};
    \path (22) ++(1,4) node {$a$};
    \path (21) ++(0,4) node {$b'$};
    \path (15) ++(0,-4) node {$c'$};
    \path (14) ++(2.5,4) node {$a'$};
    \path (9) ++(3.5,-1.5) node {$d$};
    \path (8) ++(3.5,0) node {$c$};
    \path (4) ++(0.5,-4) node {$d'$};
    \draw [black] (52) to (48);
    \draw [black] (52) to (50);
    \draw [black] (52) to (49);
    \draw [black] (51) to (47);
    \draw [black] (51) to (46);
    \draw [black] (51) to (48);
    \draw [black] (50) to (44);
    \draw [black] (50) to (45);
    \draw [black] (49) to (43);
    \draw [black] (49) to (40);
    \draw [black] (48) to (39);
    \draw [black] (47) to (38);
    \draw [black] (47) to (37);
    \draw [black] (46) to (36);
    \draw [black] (46) to (45);
    \draw [black] (45) to (35);
    \draw [black] (44) to (32);
    \draw [black] (44) to (33);
    \draw [black] (43) to (42);
    \draw [black] (43) to (32);
    \draw [black] (42) to (31);
    \draw [black] (42) to (41);
    \draw [black] (41) to (30);
    \draw [black] (41) to (40);
    \draw [black] (40) to (29);
    \draw [black] (39) to (28);
    \draw [black] (39) to (29);
    \draw [black] (38) to (27);
    \draw [black] (38) to (28);
    \draw [black] (37) to (36);
    \draw [black] (37) to (26);
    \draw [black] (36) to (25);
    \draw [black] (35) to (34);
    \draw [black] (35) to (25);
    \draw [black] (34) to (24);
    \draw [black] (34) to (33);
    \draw [black] (33) to (23);
    \draw [black] (32) to (22);
    \draw [black] (31) to (21);
    \draw [black] (31) to (22);
    \draw [black] (30) to (20);
    \draw [black] (30) to (21);
    \draw [black] (29) to (20);
    \draw [black] (28) to (19);
    \draw [black] (27) to (18);
    \draw [black] (27) to (26);
    \draw [black] (26) to (16);
    \draw [black] (25) to (16);
    \draw [black] (24) to (15);
    \draw [black] (24) to (16);
    \draw [black] (23) to (14);
    \draw [black] (23) to (15);
    \draw [black] (22) to (14);
    \draw [black] (21) to (13);
    \draw [black] (20) to (12);
    \draw [black] (19) to (11);
    \draw [black] (19) to (12);
    \draw [black] (18) to (17);
    \draw [black] (18) to (11);
    \draw [black] (17) to (10);
    \draw [black] (17) to (16);
    \draw [black] (16) to (9);
    \draw [black] (15) to (8);
    \draw [black] (14) to (7);
    \draw [black] (13) to (6);
    \draw [black] (13) to (7);
    \draw [black] (12) to (6);
    \draw [black] (11) to (5);
    \draw [black] (10) to (4);
    \draw [black] (10) to (5);
    \draw [black] (9) to (8);
    \draw [black] (9) to (4);
    \draw [black] (8) to (3);
    \draw [black] (7) to (3);
    \draw [black] (6) to (2);
    \draw [black] (5) to (2);
    \draw [black] (4) to (1);
    \draw [black] (3) to (1);
    \draw [black] (2) to (1);
    \draw [red, ultra thick] (8) -- (3) -- (7) -- (14) -- (23) -- (33) -- (44) -- (50) -- (52) -- (48) -- (39) -- (28) -- (38) -- (47) -- (51) -- (46) -- (45) -- (35) -- (34) -- (24) -- (16) -- (25) -- (36) -- (37) -- (26) -- (27) -- (18) -- (17) -- (10) -- (5) -- (11) -- (19) -- (12) -- (20) -- (29) -- (40) -- (49) -- (43) -- (32) -- (22) -- (31) -- (42) -- (41) -- (30) -- (21) -- (13) -- (6) -- (2) -- (1) -- (4) -- (9) -- (8);
\begin{pgfonlayer}{bg}
\path[fill=darkgray] (5.center) -- (10.center) -- (17.center) -- (18.center) -- 
(11.center);
\end{pgfonlayer}
\end{tikzpicture}
\end{minipage}
	\caption{The planar $K_2$-hypohamiltonian graph $G_{52}$. A face of which the boundary is an extendable $5$-cycle is filled in. All vertices for which the vertex-deleted subgraph is non-hamiltonian are encircled. On the left-hand side a hamiltonian cycle of $G_{52} - c'$ not containing $dd'$ and on the right-hand side one which does contain $dd'$ are marked by a thick line.}
	\label{fig:g52v2}
\end{figure}
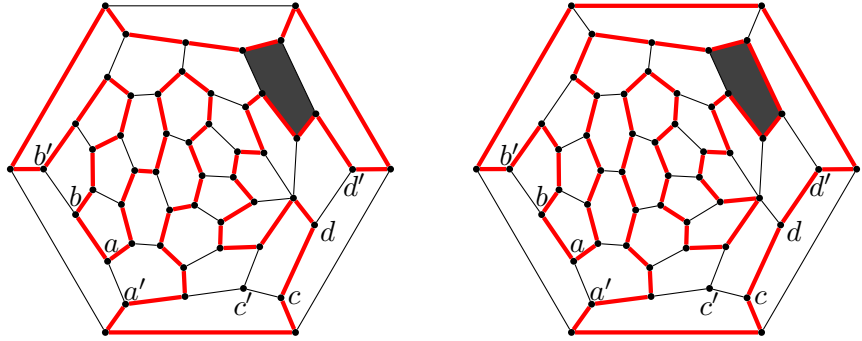

In Fig.~\ref{fig:g53}, we see that $(G_{53}, a_{53}, a'_{53}, b_{53}, b'_{53})$ satisfies the gluing property, such that $G_{53} - a'_{53}$ has at least two hamiltonian cycles, one containing $b_{53}b'_{53}$ and one not containing it. Moreover, we see that in this embedding $a_{53}, a'_{53}, b_{53}, b'_{53}$ are not co-facial. Also note that $G_{53} - v$ is hamiltonian for $v\in \{a_{53}, a'_{53}, b_{53}, b'_{53}\}$ and that $v$ is not among the vertices of the extendable $5$-cycle and their neighbours.

\begin{figure}[!htb]
	\centering
	\begin{minipage}{60mm}
\centering
	\begin{tikzpicture}[scale=0.050]
		\definecolor{marked}{rgb}{0.25,0.5,0.25}
		\node [circle,fill,scale=0.25] (53) at (60.944676,88.850605) {};
    	\node [circle,fill,scale=0.25] (52) at (54.184141,85.045931) {};
	    \node [circle,fill,scale=0.25] (51) at (45.855801,85.095862) {};
    	\node [circle,fill,scale=0.25] (50) at (39.105252,88.980423) {};
    	\node [circle,fill,scale=0.25] (49) at (49.999999,99.995004) {};
    	\node [circle,fill,scale=0.25] (48) at (80.976631,75.209704) {};
    	\node [circle,fill,scale=0.25] (47) at (67.315756,78.654880) {};
    	\node [circle,fill,scale=0.25] (46) at (55.702016,75.639102) {};
    	\node [circle,fill,scale=0.25] (45) at (49.950069,70.236666) {};
    	\node [circle,fill,scale=0.25] (44) at (44.248052,75.699019) {};
    	\node [circle,fill,scale=0.25] (43) at (32.664270,78.834628) {};
    	\node [circle,fill,scale=0.25] (42) at (18.943479,75.319549) {};
    	\node [circle,fill,scale=0.25] (41) at (8.867585,48.856599) {};
    	\node [circle,fill,scale=0.25] (40) at (14.649492,85.355499) {};
    	\node [circle,fill,scale=0.25] (39) at (85.350507,85.355499) {};
    	\node [circle,fill,scale=0.25] (38) at (91.082483,49.106250) {};
    	\node [circle,fill,scale=0.25] (37) at (83.453162,60.080885) {};
    	\node [circle,fill,scale=0.25] (36) at (73.127619,72.034150) {};
    	\node [circle,fill,scale=0.25] (35) at (61.493908,70.656079) {};
    	\node [circle,fill,scale=0.25] (34) at (49.930097,60.420409) {};
    	\node [circle,fill,scale=0.25] (33) at (38.426203,70.775911) {};
    	\node [circle,fill,scale=0.25] (32) at (26.792490,72.173954) {};
    	\node [circle,fill,scale=0.25] (31) at (16.436989,60.050927) {};
    	\node [circle,fill,scale=0.25] (30) at (17.245857,35.225684) {};
    	\node [circle,fill,scale=0.25] (29) at (0.000000,49.995005) {};
    	\node [circle,fill,scale=0.25] (28) at (99.999999,49.995005) {};
    	\node [circle,fill,scale=0.25] (27) at (82.704213,35.345516) {};
    	\node [circle,fill,scale=0.25] (26) at (75.444377,37.861992) {};
    	\node [circle,fill,scale=0.25] (25) at (76.043538,50.304572) {};
    	\node [circle,fill,scale=0.25] (24) at (67.126023,57.095065) {};
    	\node [circle,fill,scale=0.25] (23) at (57.339723,59.861192) {};
    	\node [circle,fill,scale=0.25] (22) at (42.520471,59.921108) {};
    	\node [circle,fill,scale=0.25] (21) at (32.684243,57.344715) {};
    	\node [circle,fill,scale=0.25] (20) at (23.816658,50.474334) {};
    	\node [circle,fill,scale=0.25] (19) at (24.465750,37.941880) {};
    	\node [circle,fill,scale=0.25] (18) at (21.679650,20.206712) {};
    	\node [circle,fill,scale=0.25] (17) at (14.649492,14.644497) {};
   		\node [circle,fill,scale=0.25] (16) at (85.350507,14.644497) {};
    	\node [circle,fill,scale=0.25] (15) at (78.440182,20.296584) {};
    	\node [circle,fill,scale=0.25] (14) at (65.897742,28.694827) {};
    	\node [circle,fill,scale=0.25] (13) at (67.585379,48.097662) {};
    	\node [circle,fill,scale=0.25] (12) at (58.298381,51.982223) {};
    	\node [circle,fill,scale=0.25] (11) at (41.521869,52.062111) {};
    	\node [circle,fill,scale=0.25] (10) at (32.254843,48.227480) {};
    	\node [circle,fill,scale=0.25] (9) at (34.072299,28.714800) {};
    	\node [circle,fill,scale=0.25] (8) at (36.059517,18.029758) {};
    	\node [circle,fill,scale=0.25] (7) at (49.999999,0.004995) {};
    	\node [circle,fill,scale=0.25] (6) at (63.970439,18.049731) {};
    	\node [circle,fill,scale=0.25] (5) at (59.816256,33.967445) {};
    	\node [circle,fill,scale=0.25] (4) at (49.910125,41.846414) {};
    	\node [circle,fill,scale=0.25] (3) at (40.063911,34.027361) {};
    	\node [circle,fill,scale=0.25] (2) at (50.039943,11.159379) {};
    	\node [circle,fill,scale=0.25] (1) at (49.930097,29.044338) {};
		\path (41) ++(3.5,0) node {$b$};
		\path (31) ++(-3,1) node {$b'$};
		\path (30) ++(1,4) node {$a$};
		\path (18) ++(3,4) node {$a'$};

		\draw [black] (53) to (47);
		\draw [black] (53) to (52);
		\draw [black] (53) to (49);
		\draw [black] (52) to (46);
		\draw [black] (52) to (51);
		\draw [black] (51) to (50);
		\draw [black] (51) to (44);
		\draw [black] (50) to (43);
		\draw [black] (50) to (49);
		\draw [black] (49) to (39);
		\draw [black] (49) to (40);
		\draw [black] (48) to (36);
		\draw [black] (48) to (39);
		\draw [black] (48) to (37);
		\draw [black] (47) to (35);
		\draw [black] (47) to (36);
		\draw [black] (46) to (35);
		\draw [black] (46) to (45);
		\draw [black] (45) to (34);
		\draw [black] (45) to (44);
		\draw [black] (44) to (33);
		\draw [black] (43) to (32);
		\draw [black] (43) to (33);
		\draw [black] (42) to (31);
		\draw [black] (42) to (40);
		\draw [black] (42) to (32);
		\draw [black] (41) to (29);
		\draw [black] (41) to (31);
		\draw [black] (41) to (30);
		\draw [black] (40) to (29);
		\draw [black] (39) to (28);
		\draw [black] (38) to (27);
		\draw [black] (38) to (37);
		\draw [black] (38) to (28);
		\draw [black] (37) to (25);
		\draw [black] (36) to (24);
		\draw [black] (35) to (23);
		\draw [black] (34) to (22);
		\draw [black] (34) to (23);
		\draw [black] (33) to (22);
		\draw [black] (32) to (21);
		\draw [black] (31) to (20);
		\draw [black] (30) to (18);
		\draw [black] (30) to (19);
		\draw [black] (29) to (17);
		\draw [black] (28) to (16);
		\draw [black] (27) to (26);
		\draw [black] (27) to (15);
		\draw [black] (26) to (14);
		\draw [black] (26) to (25);
		\draw [black] (25) to (13);
		\draw [black] (24) to (12);
		\draw [black] (24) to (13);
		\draw [black] (23) to (12);
		\draw [black] (22) to (11);
		\draw [black] (21) to (10);
		\draw [black] (21) to (11);
		\draw [black] (20) to (10);
		\draw [black] (20) to (19);
		\draw [black] (19) to (9);
		\draw [black] (18) to (17);
		\draw [black] (18) to (8);
		\draw [black] (17) to (7);
		\draw [black] (16) to (15);
		\draw [black] (16) to (7);
		\draw [black] (15) to (6);
		\draw [black] (14) to (5);
		\draw [black] (14) to (6);
		\draw [black] (13) to (5);
		\draw [black] (12) to (4);
		\draw [black] (11) to (4);
		\draw [black] (10) to (3);
		\draw [black] (9) to (8);
		\draw [black] (9) to (3);
		\draw [black] (8) to (2);
		\draw [black] (7) to (2);
		\draw [black] (6) to (2);
		\draw [black] (5) to (1);
		\draw [black] (4) to (1);
		\draw [black] (3) to (1);
		\draw [black] (2) to (1);
		\draw [red, ultra thick] (8) -- (2) -- (1) -- (5) -- (13) -- (25) -- (26) -- (14) -- (6) -- (15) -- (27) -- (38) -- (37) -- (48) -- (39) -- (28) -- (16) -- (7) -- (17) -- (29) -- (41) -- (30) -- (19) -- (20) -- (31) -- (42) -- (40) -- (49) -- (50) -- (51) -- (44) -- (45) -- (34) -- (23) -- (35) -- (46) -- (52) -- (53) -- (47) -- (36) -- (24) -- (12) -- (4) -- (11) -- (22) -- (33) -- (43) -- (32) -- (21) -- (10) -- (3) -- (9) -- (8);
		\begin{pgfonlayer}{bg}
			\path[fill=darkgray] (6.center) -- (14.center) -- (26.center) -- 
			(27.center) -- (15.center);
		\end{pgfonlayer}
		\path[draw=black, very thick] (4) circle[radius=2];
		\path[draw=black, very thick] (21) circle[radius=2];
		\path[draw=black, very thick] (24) circle[radius=2];
		\path[draw=black, very thick] (32) circle[radius=2];
		\path[draw=black, very thick] (36) circle[radius=2];
	\end{tikzpicture}
\end{minipage}
\begin{minipage}{60mm}
\centering
	\begin{tikzpicture}[scale=0.050]
		\definecolor{marked}{rgb}{0.25,0.5,0.25}
		\node [circle,fill,scale=0.25] (53) at (60.944676,88.850605) {};
    	\node [circle,fill,scale=0.25] (52) at (54.184141,85.045931) {};
	    \node [circle,fill,scale=0.25] (51) at (45.855801,85.095862) {};
    	\node [circle,fill,scale=0.25] (50) at (39.105252,88.980423) {};
    	\node [circle,fill,scale=0.25] (49) at (49.999999,99.995004) {};
    	\node [circle,fill,scale=0.25] (48) at (80.976631,75.209704) {};
    	\node [circle,fill,scale=0.25] (47) at (67.315756,78.654880) {};
    	\node [circle,fill,scale=0.25] (46) at (55.702016,75.639102) {};
    	\node [circle,fill,scale=0.25] (45) at (49.950069,70.236666) {};
    	\node [circle,fill,scale=0.25] (44) at (44.248052,75.699019) {};
    	\node [circle,fill,scale=0.25] (43) at (32.664270,78.834628) {};
    	\node [circle,fill,scale=0.25] (42) at (18.943479,75.319549) {};
    	\node [circle,fill,scale=0.25] (41) at (8.867585,48.856599) {};
    	\node [circle,fill,scale=0.25] (40) at (14.649492,85.355499) {};
    	\node [circle,fill,scale=0.25] (39) at (85.350507,85.355499) {};
    	\node [circle,fill,scale=0.25] (38) at (91.082483,49.106250) {};
    	\node [circle,fill,scale=0.25] (37) at (83.453162,60.080885) {};
    	\node [circle,fill,scale=0.25] (36) at (73.127619,72.034150) {};
    	\node [circle,fill,scale=0.25] (35) at (61.493908,70.656079) {};
    	\node [circle,fill,scale=0.25] (34) at (49.930097,60.420409) {};
    	\node [circle,fill,scale=0.25] (33) at (38.426203,70.775911) {};
    	\node [circle,fill,scale=0.25] (32) at (26.792490,72.173954) {};
    	\node [circle,fill,scale=0.25] (31) at (16.436989,60.050927) {};
    	\node [circle,fill,scale=0.25] (30) at (17.245857,35.225684) {};
    	\node [circle,fill,scale=0.25] (29) at (0.000000,49.995005) {};
    	\node [circle,fill,scale=0.25] (28) at (99.999999,49.995005) {};
    	\node [circle,fill,scale=0.25] (27) at (82.704213,35.345516) {};
    	\node [circle,fill,scale=0.25] (26) at (75.444377,37.861992) {};
    	\node [circle,fill,scale=0.25] (25) at (76.043538,50.304572) {};
    	\node [circle,fill,scale=0.25] (24) at (67.126023,57.095065) {};
    	\node [circle,fill,scale=0.25] (23) at (57.339723,59.861192) {};
    	\node [circle,fill,scale=0.25] (22) at (42.520471,59.921108) {};
    	\node [circle,fill,scale=0.25] (21) at (32.684243,57.344715) {};
    	\node [circle,fill,scale=0.25] (20) at (23.816658,50.474334) {};
    	\node [circle,fill,scale=0.25] (19) at (24.465750,37.941880) {};
    	\node [circle,fill,scale=0.25] (18) at (21.679650,20.206712) {};
    	\node [circle,fill,scale=0.25] (17) at (14.649492,14.644497) {};
   		\node [circle,fill,scale=0.25] (16) at (85.350507,14.644497) {};
    	\node [circle,fill,scale=0.25] (15) at (78.440182,20.296584) {};
    	\node [circle,fill,scale=0.25] (14) at (65.897742,28.694827) {};
    	\node [circle,fill,scale=0.25] (13) at (67.585379,48.097662) {};
    	\node [circle,fill,scale=0.25] (12) at (58.298381,51.982223) {};
    	\node [circle,fill,scale=0.25] (11) at (41.521869,52.062111) {};
    	\node [circle,fill,scale=0.25] (10) at (32.254843,48.227480) {};
    	\node [circle,fill,scale=0.25] (9) at (34.072299,28.714800) {};
    	\node [circle,fill,scale=0.25] (8) at (36.059517,18.029758) {};
    	\node [circle,fill,scale=0.25] (7) at (49.999999,0.004995) {};
    	\node [circle,fill,scale=0.25] (6) at (63.970439,18.049731) {};
    	\node [circle,fill,scale=0.25] (5) at (59.816256,33.967445) {};
    	\node [circle,fill,scale=0.25] (4) at (49.910125,41.846414) {};
    	\node [circle,fill,scale=0.25] (3) at (40.063911,34.027361) {};
    	\node [circle,fill,scale=0.25] (2) at (50.039943,11.159379) {};
    	\node [circle,fill,scale=0.25] (1) at (49.930097,29.044338) {};
		\path (41) ++(3.5,0) node {$b$};
		\path (31) ++(-3,1) node {$b'$};
		\path (30) ++(1,4) node {$a$};
		\path (18) ++(3,4) node {$a'$};
		
		\draw [black] (53) to (47);
		\draw [black] (53) to (52);
		\draw [black] (53) to (49);
		\draw [black] (52) to (46);
		\draw [black] (52) to (51);
		\draw [black] (51) to (50);
		\draw [black] (51) to (44);
		\draw [black] (50) to (43);
		\draw [black] (50) to (49);
		\draw [black] (49) to (39);
		\draw [black] (49) to (40);
		\draw [black] (48) to (36);
		\draw [black] (48) to (39);
		\draw [black] (48) to (37);
		\draw [black] (47) to (35);
		\draw [black] (47) to (36);
		\draw [black] (46) to (35);
		\draw [black] (46) to (45);
		\draw [black] (45) to (34);
		\draw [black] (45) to (44);
		\draw [black] (44) to (33);
		\draw [black] (43) to (32);
		\draw [black] (43) to (33);
		\draw [black] (42) to (31);
		\draw [black] (42) to (40);
		\draw [black] (42) to (32);
		\draw [black] (41) to (29);
		\draw [black] (41) to (31);
		\draw [black] (41) to (30);
		\draw [black] (40) to (29);
		\draw [black] (39) to (28);
		\draw [black] (38) to (27);
		\draw [black] (38) to (37);
		\draw [black] (38) to (28);
		\draw [black] (37) to (25);
		\draw [black] (36) to (24);
		\draw [black] (35) to (23);
		\draw [black] (34) to (22);
		\draw [black] (34) to (23);
		\draw [black] (33) to (22);
		\draw [black] (32) to (21);
		\draw [black] (31) to (20);
		\draw [black] (30) to (18);
		\draw [black] (30) to (19);
		\draw [black] (29) to (17);
		\draw [black] (28) to (16);
		\draw [black] (27) to (26);
		\draw [black] (27) to (15);
		\draw [black] (26) to (14);
		\draw [black] (26) to (25);
		\draw [black] (25) to (13);
		\draw [black] (24) to (12);
		\draw [black] (24) to (13);
		\draw [black] (23) to (12);
		\draw [black] (22) to (11);
		\draw [black] (21) to (10);
		\draw [black] (21) to (11);
		\draw [black] (20) to (10);
		\draw [black] (20) to (19);
		\draw [black] (19) to (9);
		\draw [black] (18) to (17);
		\draw [black] (18) to (8);
		\draw [black] (17) to (7);
		\draw [black] (16) to (15);
		\draw [black] (16) to (7);
		\draw [black] (15) to (6);
		\draw [black] (14) to (5);
		\draw [black] (14) to (6);
		\draw [black] (13) to (5);
		\draw [black] (12) to (4);
		\draw [black] (11) to (4);
		\draw [black] (10) to (3);
		\draw [black] (9) to (8);
		\draw [black] (9) to (3);
		\draw [black] (8) to (2);
		\draw [black] (7) to (2);
		\draw [black] (6) to (2);
		\draw [black] (5) to (1);
		\draw [black] (4) to (1);
		\draw [black] (3) to (1);
		\draw [black] (2) to (1);
		\draw [red, ultra thick] (8) -- (2) -- (1) -- (4) -- (11) -- (21) -- (32) -- (43) -- (50) -- (51) -- (52) -- (53) -- (49) -- (39) -- (48) -- (36) -- (47) -- (35) -- (46) -- (45) -- (44) -- (33) -- (22) -- (34) -- (23) -- (12) -- (24) -- (13) -- (5) -- (14) -- (6) -- (15) -- (27) -- (26) -- (25) -- (37) -- (38) -- (28) -- (16) -- (7) -- (17) -- (29) -- (40) -- (42) -- (31) -- (41) -- (30) -- (19) -- (20) -- (10) -- (3) -- (9) -- (8);
		\begin{pgfonlayer}{bg}
			\path[fill=darkgray] (6.center) -- (14.center) -- (26.center) -- 
			(27.center) -- (15.center);
		\end{pgfonlayer}
		\path[draw=black, very thick] (4) circle[radius=2];
		\path[draw=black, very thick] (21) circle[radius=2];
		\path[draw=black, very thick] (24) circle[radius=2];
		\path[draw=black, very thick] (32) circle[radius=2];
		\path[draw=black, very thick] (36) circle[radius=2];
	\end{tikzpicture}
\end{minipage}
	\caption{The planar $K_2$-hypohamiltonian graph $G_{53}$. A face of which the boundary is an extendable $5$-cycle is filled in. All vertices for which the vertex-deleted subgraph is non-hamiltonian are circled. On the left-hand side a hamiltonian cycle of $G_{53} - a'$ not containing $bb'$ and on the right-hand side one which does contain $bb'$ are marked by a thick line.}
	\label{fig:g53}
\end{figure}
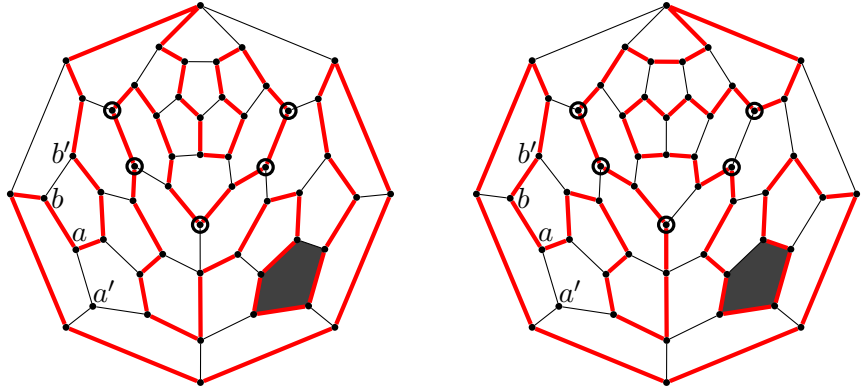
\end{document}